\documentclass[11pt]{article}

\usepackage[backend=biber,style=numeric]{biblatex}
\usepackage{authblk}
\usepackage{csquotes} 
\usepackage[dvipdfmx]{graphicx}
\usepackage{color}
\usepackage{bm}
\usepackage{amsmath,amsthm,amssymb,url}
\usepackage{multirow,bigdelim}
\usepackage{hyperref}
\usepackage{enumitem}
\usepackage[normalem]{ulem}

\usepackage{pgfplots}
\pgfplotsset{compat=1.15}
\usepackage{mathrsfs}
\usetikzlibrary{arrows}

\usepackage{etoolbox}
\let\bbordermatrix\bordermatrix
\patchcmd{\bbordermatrix}{8.75}{4.75}{}{}
\patchcmd{\bbordermatrix}{\left(}{\left[}{}{}
\patchcmd{\bbordermatrix}{\right)}{\right]}{}{}

\newtheorem{thm}{Theorem}[section]
\newtheorem{theorem}[thm]{Theorem}

\newtheorem{lemma}[thm]{Lemma}
\newtheorem*{lemma*}{Lemma}

\newtheorem{claim}[thm]{Claim}
\newtheorem{corollary}[thm]{Corollary}

\newtheorem{conjecture}[thm]{Conjecture}

\newtheorem*{theorem*}{Theorem}

\DeclareMathOperator{\rank}{rank}

\newcommand{\cl}{\mathrm{cl}}
\newcommand{\sm}{\setminus}

\newcommand{\cA}{{\cal A}}

\newcommand{\cC}{{\cal C}}

\newcommand{\cL}{{\cal L}}
\newcommand{\cK}{{\cal K}}
\newcommand{\cT}{{\cal T}}
\newcommand{\cS}{{\cal S}}
\newcommand{\cU}{{\cal U}}

\newcommand{\cR}{{\cal R}}
\newcommand{\cV}{{\cal V}}
\newcommand{\cW}{{\cal W}}

\newcommand{\R}{{\mathbb R}}

\author[1]{James Cruickshank}

\author[2]{Bill Jackson}
\author[3]{Shin-ichi Tanigawa} 
\affil[1]{School of Mathematical and Statistical Sciences, University of Galway, H91 TK33, Ireland. Email address: james.cruickshank@universityofgalway.ie}
\affil[2]{School of Mathematical Sciences, Queen Mary University of London, Mile End Road, London, E1 4NS, United Kingdom. Email address: b.jackson@qmul.ac.uk}
\affil[3]{Department of Mathematical Informatics, University of Tokyo, Japan. Email address: tanigawa@mist.i.u-tokyo.ac.jp}

\begin{document}
\title{Global Rigidity of Triangulated Manifolds}
\date{}
\maketitle

\begin{abstract}
    We prove that if $G$ is the graph of a connected triangulated $(d-1)$-manifold, for $d\geq 3$, then $G$ is generically globally rigid 
    in $\mathbb R^d$ if and only if it is $(d+1)$-connected and, if $d=3$, $G$ is not planar. 
    The special case $d=3$ verifies a conjecture of Connelly. Our results actually apply to a much larger class of simplicial complexes, namely the circuits of the simplicial matroid. We also give two significant applications of our main theorems. We show that the characterisation of pseudomanifolds with extremal edge numbers given by the 
    Lower Bound Theorem extends to circuits of the simplicial matroid. We also prove the generic 
    case of a conjecture of Kalai concerning the reconstructability of a polytope from its 
    space of stresses. The proofs of our main results adapt earlier ideas of Fogelsanger and Whiteley to the setting of global rigidity. In particular we verify a special case of Whiteley's vertex splitting conjecture for global rigidity.

    

\end{abstract}

\section{Introduction}
The study of the rigidity of triangulated surfaces can be traced back to a claim made by Euler in 1776    that ``A closed spatial figure allows not changes, as long as it is not ripped apart." Giving a formal proof for Euler's claim turned out to be a challenging problem. A celebrated result of  Cauchy in 1813 implies that every triangulated convex spherical surface  in $\R^3$ is rigid.
Gluck~\cite{G} showed  that every generic triangulated spherical surface in $\R^3$ is rigid.
Rather surprisingly,  Connelly \cite{C77,C82} disproved Euler's original claim by constructing a flexible triangulated spherical surface in $\R^3$.
It is natural to also consider  the rigidity of triangulations of other closed surfaces.
Fogelsanger~\cite{F} was able to extend Gluck's theorem by showing that every generic triangulated surface in $\R^3$ is  rigid. 

In this paper we will relax the hypothesis at the end of Euler's claim and ask
whether a closed spacial figure can change if it is ripped apart and then re-assembled. We will consider the version of the problem in which the faces of a triangulated surface in $\R^3$ are re-assembled so that all adjacencies are preserved and we will allow the re-assembled faces  to pass through each other. 
The re-assembled surface may differ from the original  surface even when the original surface is generic. 
For example, if 
a triangulated surface has a vertex $v$ of degree three which does not lie
in the affine plane spanned by its three neighbours,
then we can obtain a non-congruent  surface 
by reflecting $v$ in this plane.
The same trick can be applied whenever the underlying graph of the triangulated surface has a vertex-separator of size three. This suggests that we need to add the condition that the underlying graph is  4-connected to ensure that a triangulated surface is uniquely defined by its faces and their adjacencies, but this is still not strong enough to guarantee uniqueness.
A result of 
Hendrickson \cite{H} (Theorem~\ref{thm:hendrickson} below) implies that every generic triangulated spherical surface can be  changed to a non-congruent triangulated surface by re-assembling its faces, no matter how high the connectivity of its underlying graph is. 
We will show that  spherical surfaces provide the only such examples:   
every generic triangulated non-spherical surface  in $\R^3$ is uniquely defined  by its faces and their adjacencies whenever  its underlying graph is 4-connected.
(Combining our result with a recent result by Gortler, Theran, and Thurston~\cite{GTT}, it further implies that  
every generic triangulated non-spherical surface  in $\R^3$ is uniquely determined  by its list of edge lengths and its number of vertices.)

We can use the language of rigidity theory to give a more precise  description of our results.
A $d$-dimensional {\em framework} is defined as a pair  $(G,p)$  consisting of a graph $G$ and a map $p:V(G)\rightarrow \R^d$.
We will also say that $(G,p)$ is a {\em realisation} of $G$ in $\R^d$.
Two realisations $(G, p)$ and $(G, q)$ 
of $G$ 
in $\R^d$ 
are {\em congruent} 
if $(G, p)$ can be obtained from $(G, q)$ by an isometry of $\R^d$, i.e., a combination of translations, rotations
and reflections. The framework $(G, p)$ is {\em globally rigid} if every framework which has the
same edge lengths as $(G, p)$ is congruent to $(G, p)$. It is {\em rigid} if every continuous motion of
the vertices of $(G, p)$ in $\R^d$ which preserves the edge lengths results in a framework which
is congruent to $(G, p)$.

A framework $(G, p)$ is {\em generic} if the multiset of coordinates of the points $p(v), v\in V (G)$, 
is algebraically independent over $\mathbb{Q}$. 
Gluck~\cite{G} showed that the rigidity of a generic framework 
depends only on its underlying graph.
This enables us to define
a graph $G$ as being {\em rigid in $\R^d$} if some (or equivalently every) generic realisation of $G$
in $\R^d$ is rigid.
Analogous, but much deeper, results of Connelly \cite{C} and Gortler, Healy and Thurston \cite{GHT} 
imply that the global rigidity of a generic framework depends only on its underlying graph and allows us to define a graph
$G$ as being {\em globally rigid} in $\R^d$ if some (or equivalently every) generic realisation of $G$ in $\R^d$ is globally rigid.
Finding a combinatorial characterisation of rigid or globally rigid graphs in $\mathbb R^d$ for $d \geq 3$ is a major open problem in discrete geometry.
One family of graphs for which generic rigidity in $\R^d$ is well understood is the graphs (1-skeletons) of triangulated $(d-1)$-manifolds. Whiteley \cite{W84} extended Gluck's Theorem by showing that the graphs of $(d-1)$-polytopes are rigid in $\R^d$ for all $d\geq 3$. This result was in turn extended to 
graphs of triangulated  $(d-1)$-manifolds for all $d\geq 4$ by Kalai \cite{K} and subsequently for all $d\geq 3$ by Fogelsanger \cite{F}.

A fundamental necessary condition for any non-complete graph $G$ to be rigid in $\R^d$ is that 
\begin{equation}\label{eq:count}
    |E(G)|\geq d|V(G)|-\binom{d+1}{2}.
\end{equation}
Thus Fogelsanger's Rigidity Theorem implies that the graph 
of any triangulated connected $(d-1)$-manifold 
must satisfy (\ref{eq:count}) whenever $d\geq 3$.

The  following result of Hendrickson~\cite{H} gives two necessary conditions for global rigidity.
\begin{theorem}[Hendrickson's necessary conditions]\label{thm:hendrickson}
Let $G$ be a globally rigid graph in $\R^d$.
Then $G$ is either a complete graph on at most $d+1$ vertices or $G$ is $(d+1)$-connected and $G-e$ is rigid in $\R^d$ for all $e\in E(G)$.
\end{theorem} 
Hendrickson's second necessary condition and (\ref{eq:count}) tell us that a  
non-complete graph $G$ which is globally rigid in $\R^d$
must satisfy $|E(G)|\geq d|V(G)|-\binom{d+1}{2}+1$.
This implies, in particular, that a triangulation of the 2-sphere (equivalently, a plane triangulation) cannot be globally rigid in $\R^3$.
Since this argument  applies only to spherical surfaces, Connelly \cite{C14} asked if the global rigidity of  triangulations of non-spherical surfaces is characterised by their 4-connectivity. 
We will answer both this question and its generalisation to triangulated $(d-1)$-manifolds for $d \geq 3$.

\begin{theorem}\label{thm:globalrigid_3_surface}
Let $G$ be the graph of a triangulated $(d-1)$-manifold for some $d\geq 3$.
Then $G$ is globally rigid in $\R^d$ if and only if $G=K_d,K_{d+1}$ or $G$ is $(d+1)$-connected and, when $d=3$, $G$ is non-planar.
\end{theorem}

 Theorem~\ref{thm:globalrigid_3_surface} had already been proved when $d=3$ for surfaces of small genus in~\cite{JT,S}.
Their approach, 
based on Whiteley's vertex splitting operation (defined in Section~\ref{sec:pre}), 
is to  show that any 4-connected triangulation of a fixed surface can be constructed  by a sequence of vertex splitting operations from a finite list of irreducible triangulations, and then check that every irreducible triangulation in the list is globally rigid. This approach was developed for the surfaces $S_1$ and $N_1$ in \cite{JT}.
Shinoki~\cite{S} subsequently implemented an algorithm for enumerating the irreducible triangulations on a fixed surface and confirmed Theorem~\ref{thm:globalrigid_3_surface} for the surfaces $S_2, N_2,N_3, N_4$.
This approach is inevitably limited by the fact that the number of irreducible triangulations of a fixed surface grows rapidly with respect to its genus.
We will instead adopt the approach taken by Fogelsanger in his PhD thesis~\cite{F}.
His ingenious idea is to consider a more general family 
of abstract simplicial complexes.
This family can be defined as follows.

A  (non-trivial) {\em simplicial $k$-circuit} is an abstract simplicial $k$-complex 
with the property that 
each of its $(k-1)$-faces belongs to an even number of $k$-simplices and is such that none of its sub-complexes have this property.\footnote{Fogelsanger proved his results in the more general context of  `minimal cycles' in chain complexes  over an arbitrary abelian group. Our ‘simplicial circuits’  correspond to the special case when  the abelian group is $\mathbb{Z}_2$. 
This simplification allows us to work directly with simplicial complexes instead of chain complexes.
Our terminology is motivated by the fact that simplicial $k$-circuits form the set of circuits in the $k$-simplicial matroid, 
see \cite[Chapter 10]{welsh}.}
Fogelsanger showed that, if  the graph $G$ of a simplicial $k$-circuit $\cS$ cannot be obtained from that of a smaller simplicial $k$-circuit by vertex splitting, then 
$G$ can be covered by graphs of simplicial $k$-circuits which do have this property.
(We will refer to the set of simplicial $k$-circuits which generate this cover as a  {\em Fogelsanger decomposition} of $\cS$.) This allowed him to apply induction to the  smaller circuits to show that
the graph of every simplicial $k$-circuit is rigid in $\R^{k+1}$. 
Since every  triangulated connected $k$-manifold is a simplicial $k$-circuit,\footnote{This follows from the fact that, if a simplicial $k$-complex is a triangulated connected $k$-manifold, then it is strongly connected and every $(k-1)$-face belongs to exactly two $k$-simplices, see Lemma~\ref{lem:pseudo} below.} his result implies, in particular, that 
the graph of any  triangulated connected $k$-manifold is rigid in $\R^{k+1}$.
We will give a detailed description of the Fogelsanger decomposition as well as  a self contained proof of his rigidity theorem in Section~\ref{sec:circuits}.

Our main results extend Fogelsanger's theorem to global rigidity. They immediately imply Theorem \ref{thm:globalrigid_3_surface}.
  
\begin{theorem}\label{thm:globrigid_3_main}
Let $G$ be the graph of a simplicial $2$-circuit.
Then $G$ is globally rigid  in $\R^{3}$ if and only if $G=K_3,K_4$ or $G$ is 4-connected and non-planar.
\end{theorem}
\begin{theorem}\label{thm:globrigid_d_main}
Let $G$ be the graph of a simplicial $k$-circuit for some $k\geq 3$.
Then $G$ is globally rigid  in $\R^{k+1}$ if and only if $G=K_{k+1},K_{k+2}$ or $G$ is $(k+2)$-connected.
\end{theorem}

There are two major difficulties in adapting Fogelsanger's approach to prove Theorems~\ref{thm:globrigid_3_main} and~\ref{thm:globrigid_d_main}.
The first  is that the graphs of the simplicial $k$-circuits in a Fogelsanger decomposition may not be $(k+2)$-connected.
This difficulty will be resolved by introducing the concept of the strong $(k+1)$-cleavage property for graphs and then proving a gluing-type theorem for graphs with  this property in Section~\ref{sec:strong_cleavage}.
The second  difficulty is that 
Whiteley's vertex splitting theorem for rigidity (Lemma~\ref{lem:split} below) is not known to hold for global rigidity.
Indeed,  a global rigidity analogue of Whiteley's vertex splitting theorem is a long standing conjecture (Conjecture~\ref{conj:vertex_splitting}) in rigidity theory.
We will  resolve this difficulty by verifying the conjecture for the graphs of simplicial $k$-circuits in (Section~\ref{sec:coincidence}). 

The paper is organised as follows.
In Section~\ref{sec:pre}, we review basic results on rigidity.
In Section~\ref{sec:circuits}, we will give a detailed discussion about simplicial circuits. In particular we give a self-contained exposition of Fogelsanger decompositions.
We then move to the proof of our main theorem.
In Section~\ref{sec:cleavage}, we first discuss the $d$-cleavage property of graphs, which captures the structure of minimum vertex separators in simplicial circuits and will provide a foundation for our graph theoretical proof.
We also give the proofs of Theorem~\ref{thm:LBT} and Theorem~\ref{thm:redundant} (given below).
In Section~\ref{sec:strong_cleavage}, we introduce the strong $d$-cleavage property, which combines the notion of global rigidity and the $d$-cleavage property, and prove a gluing-type theorem for graphs with the strong $d$-cleavage property. 
In Section~\ref{sec:globproof}, we give proofs of our main theorems, Theorems~\ref{thm:globrigid_3_main} and~\ref{thm:globrigid_d_main}, assuming that vertex splitting preserves global rigidity on simplicial circuits.
The proofs of our main theorems are completed in Section~\ref{sec:coincidence} by showing that vertex splitting does  indeed preserve global rigidity for simplicial circuits.
We remark that section~\ref{sec:coincidence} can be skipped if the truth of the vertex splitting conjecture is assumed.

\subsubsection*{Applications: The Lower Bound Theorem and Stresses of Polytopes}

Fogelsanger's Rigidity Theorem implies the edge-inequality part of the
celebrated Lower Bound Theorem for triangulated $(d-1)$-manifolds due to Barnette~\cite{B71,B73}. 
This asserts, in particular, that, if $G$ is the graph of a connected triangulated $(d-1)$-manifold, then 
(\ref{eq:count}) holds. The full theorem gives similar inequalities for the number of faces of any fixed size. 
Kalai~\cite{K} extended Barnette's theorem 
by characterising when equality can hold in the face counts. Both  results were subsequently extended to pseudomanifolds by Tay~\cite{T}.
We will show in Section \ref{sec:cleavage} that these extremal results  for edge counts can be further extended from (pseudo)manifolds to simplicial circuits by combining Theorems  \ref{thm:globrigid_3_main} and \ref{thm:globrigid_d_main} 
with the $(k+1)$-cleavage property of simplicial $k$-circuits. 

\begin{theorem}\label{thm:LBT}
Let $G$ be the graph of a simplicial $(d-1)$-circuit $\cS$ for some $d\geq 3$.
Then $G$ satisfies (\ref{eq:count}) and equality holds if and only if 
$\cS$ is a stacked $(d-1)$-sphere or $d=3$ and $G$ is planar.
\end{theorem}
We refer the reader to~\cite{NN,Z} for other recent developments in the study of simplicial $d$-polytopes with few edges.
Motivated by the Lower Bound Theorem, Zheng~\cite{Z20} has recently shown that the graph of any $(d+1)$-connected homology $(d-1)$-sphere is redundantly rigid in $\mathbb{R}^d$ for $d\geq 4$.
Our global rigidity characterisation leads to the following complete characterisation of redundant edges in the larger family of simplicial $(d-1)$-circuits. 
\begin{theorem}\label{thm:redundant}
Let $G$ be the graph of a simplicial $(d-1)$-circuit for some $d\geq 3$
and let $e\in E(G)$.
If $d\geq 4$, $G-e$ is rigid in $\mathbb{R}^d$ if and only if $e$ is contained in a $(d+1)$-connected subgraph in $G$.
If $d=3$, $G-e$ is rigid in $\mathbb{R}^3$ if and only if $e$ is contained in a 4-connected non-planar subgraph in $G$.
\end{theorem} 

The $d$-cleavage property discussed in Section~\ref{sec:cleavage} allows us to  enumerate all maximal (non-planar) $(d+1)$-connected subgraphs of a simplicial $(d-1)$-circuit   by repeated applications of a max flow/min cut algorithm. Thus the necessary and sufficient conditions for an edge to be redundant  in Theorem~\ref{thm:redundant} can be checked efficiently. 

As a final application of Theorem 1.4 in Section 8 we prove the following result which is the generic case of a conjecture of Kalai, 
see \cite[Conjecture 5.7]{NZ}.

\begin{theorem}\label{thm:kalai} 
Suppose that $d \geq 4$ and let $P$ be a simplicial $d$-polytope whose graph is $(d+1)$-connected and whose vertex set is generic. Then $P$ is uniquely determined up to affine equivalence by its graph and the space of stresses of the underlying framework.
\end{theorem}

\subsubsection*{Notation}

The symmetric difference of two sets $X$ and $Y$ is denoted by $X\triangle Y$. 
For a finite set $X$, $K(X)$ denotes the set of edges of the complete graph on $X$.
If $X=\{v_1,\dots, v_k\}$, we simply denote $K(\{v_1,\dots, v_k\})$ by $K(v_1,\dots, v_k)$.
The vertex and edge sets of a graph $G$   are denoted by $V(G)$ and $E(G)$, respectively.
For $v\in V(G)$, the set of neighbours of $v$ in $G$ is denoted by $N_G(v)$ and its degree is  $d_G(v)=|N_G(v)|$.
When the graph $G$ is clear from the context, we will often suppress reference to $G$ in these notations.
For $X\subseteq V(G)$, let $G[X]$ be the subgraph of $G$ induced by $X$.
We say that $X$ is a {\em  $k$-clique of $G$} if $G[X]$ is a complete graph on $k$ vertices.
For a connected graph $G$, a set $X\subseteq V(G)$ is called a {\em separator} if $G-X$ is not connected.
It is called a {\em $k$-vertex separator} if $|X|=k$, and called a {\em $k$-clique separator} if $X$ is a $k$-clique.   
All graphs considered will be {\em simple} i.e., without loops or multiple edges. We use $K_n$ to denote the complete graph on $n$ vertices and $C_n$ the cycle on $n$ vertices.

\section{Preliminary results on rigidity}\label{sec:pre}
Let $d$ be a positive integer.
A basic step in deciding whether  a given $d$-dimensional framework $(G,p)$  is rigid is to check its infinitesimal rigidity by calculating the rank of its {\em rigidity matrix}. This is the matrix of size $|E|\times d|V(G)|$ in which each row is indexed by an edge, sets of $d$ consecutive columns are indexed by the vertices, and the row indexed by the  edge $e=uv$ has the form:
\[
\bbordermatrix{
  & &  u & & v & \cr
 e=uv & 0 \dots 0 & p(u)-p(v) & 0\dots 0 & p(v)-p(u) & 0\dots 0
}.
\]
The space of {\em infinitesimal motions}  of $(G,p)$ is given by the right kernel of $R(G,p)$,  and  $(G,p)$ is  {\em infinitesimally rigid} if we have $\rank R(G,p)=d|V(G)|-\binom{d+1}{2}$ when $|V(G)|\geq d$ and $\rank R(G,p)=\binom{|V(G)|}{2}$ when $|V(G)|\leq d-1$, since this will imply that every vector in $\ker R(G,p)$ is an infinitesimal isometry of $\R^d$.
It is a fundamental fact in rigidity theory that rigidity and infinitesimal rigidity are equivalent properties for generic frameworks.
Since the rank of $R(G,p)$ will be maximised whenever $(G,p)$ is generic, 
this implies that a graph $G$ is rigid in $\R^d$ if and only if some  realisation of $G$ in $\R^d$ is infinitesimally rigid.

We say that a graph $G$ is {\em minimally rigid} if it is rigid and $G-e$ is not rigid for every $e\in E(G)$.
If $G$ is minimally rigid in $\R^d$ with $|V(G)|\geq d
$, then
$R(G,p)$ is row independent with $\rank R(G,p)=d|V(G)|-\binom{d+1}{2}$ for some $d$-dimensional realisation $(G,p)$. This in turn
implies that $|E(G)|=d|V(G)|-\binom{d+1}{2}$ and 
\begin{equation}\label{eq:3_6_sparse}
|E(H)|\leq d|V(H)|-\binom{d+1}{2} \quad \text{for all subgraphs $H$ of $G$ with $|V(H)|\geq d$}.
\end{equation}

We next introduce two well-known graph operations that preserve infinitesimal rigidity.
Let $G=(V,E)$ be a graph.
The $d$-dimensional {\em $0$-extension} operation creates a new graph from $G$ by adding a new vertex $v$ and $d$ new edges incident to $v$. It is well-known that the $d$-dimensional 0-extension operation preserves generic rigidity in $\mathbb R^d$.
This follows from the following more general statement for arbitrary frameworks.
\begin{lemma}\label{lem:0extension}
Suppose that $(G,p)$ is a framework in $\R^d$ and $v$ is a vertex in $G$.
If $(G-v, p|_{V-v})$ is infinitesimally rigid in $\R^d$ and $\{p(v)-p(u): u\in N_G(v)\}$ spans $\R^d$,
then $(G,p)$ is infinitesimally rigid in $\R^d$.
\end{lemma}

Given a graph $G=(V,E)$ and $v\in V$ with neighbour set  $N_G(v)$ the {\em ($d$-dimensional) vertex splitting operation} constructs a new graph $G'$ by deleting $v$, adding two new vertices $v'$ and $v''$ with $N_G(v')\cup N_G(v'')=N_G(v)\cup \{v',v''\}$ and $|N_G(v')\cap N_G(v'')|\geq d-1$.  Whiteley \cite{Wsplit} showed that vertex splitting preserves generic rigidity in $\R^d$. More precisely he proved

\begin{lemma}\label{lem:split} 
Suppose that $(G,p)$ is an infinitesimally rigid framework  in $\R^d$ and that $G'$ is obtained from $G$ by a vertex splitting operation which splits a vertex $v$ of $G$ into two vertices $v',v''$. Suppose further that, for some $X\subseteq  N_{G'}(v')\cap N_{G'}(v'')$ with $|X|=d-1$, the points in $\{p(w)\,:\,w\in X+v\}$ are affinely independent in $\R^d$. 
Then $(G',p')$ is infinitesimally rigid for some $p'$ with $p'(v')=p(v)$ and $p'(w)=p(w)$ for all $w\in V(G)\setminus \{v\}$.
\end{lemma}

Suppose that $(G_1,p_1)$ and $(G_2,p_2)$ are infinitesimally rigid frameworks in $\mathbb R^d$ such that, for $i=1,2$ the affine span of $p_i(V(G_i))$ has dimension at least $d-1$, and such that $p_1(v) = p_2(v)$ for $v\in V(G_1) \cap V(G_2)$. Define $p:V(G_1) \cup V(G_2)$ by $p(v) = p_i(v)$ for $v \in V(G_i)$. The {\em infinitesimal rigidity gluing property for frameworks} says that $(G_1 \cup G_2,p)$ is infinitesimally rigid in $\mathbb R^d$ if and only if the affine span of $p(V(G_1) \cap V(G_2))$ has dimension at least $d-1$.

In particular, the
{\em rigidity gluing property} states that,
for any two rigid graphs $G_1$ and $G_2$ in $\mathbb{R}^d$ such that $|V(G_i)| \geq d$ for $i=1,2$, $G_1\cup G_2$ is rigid in $\mathbb{R}^d$ if and only if
$|V(G_1)\cap V(G_2)|\geq d$.
This intuitively clear fact is one of the characteristic properties of graph rigidity. 
It is also a key tool in the proof of Fogelsanger's Rigidity Theorem.
The analogous {\em global rigidity gluing property} states that,
for any two globally rigid graphs $G_1$ and $G_2$ in $\mathbb{R}^d$ such that $|V(G_i)| \geq d+1$ for $i=1,2$, $G_1\cup G_2$ is globally rigid in $\mathbb{R}^d$ if and only if $|V(G_1)\cap V(G_2)|\geq d+1$.

The final result of this section, due to Tanigawa \cite[Theorem 4.11]{ST}, extends the global rigidity gluing property. Given two graphs $G,H$ and a set of vertices $X\subseteq V(G)\cap V(H)$, we say that {\em $(H,X)$ is a rooted minor of $(G,X)$} if $H$ can be obtained from $G$ by a sequence of edge deletions and contractions in such a way that each contracted edge $e$ is incident with at most one vertex in $X$ and that, if $e$ is incident with  $x\in X$, then we contract $e$ onto $x$ so that the set $X$ is preserved by each edge contraction. 

\begin{theorem}[{\cite[Theorem 4.11]{ST}}]\label{lem:globunion} Let  $G_1,G_2$ be graphs, $X=V(G_1) \cap V(G_2)$ and $H$ be a graph on $X$. Suppose that: $|X| \geq d+1$; $G_1$ is rigid in $\R^d$; $(H, X)$ is a rooted minor of $(G_2, X)$; either $G_1\cup H$ and $G_2\cup K(X)$ are globally rigid in $\R^d$, or $G_1\cup K(X)$ and $G_2\cup H$ are globally rigid in $\R^d$.
Then $G_1\cup G_2$ is globally rigid in $\R^d$.
\end{theorem}

\section{Simplicial circuits}\label{sec:circuits}

\subsection{Definitions and elementary properties}

Throughout the remainder of this paper we will be interested in non-trivial simplicial $k$-circuits as defined in the introduction. However, our analysis will focus mostly on either the set of facets (i.e., faces of maximal dimension) or on the graph of these objects. 
Thus we will identify a pure abstract simplicial $k$-complex with its set of $k$-simplices. Furthermore, it will also be useful for technical reasons explained in Section \ref{sec:contractions} to allow our objects to have multiple copies of the same $k$-simplex.
Thus for an integer $k\geq 0$ we will refer to a multiset ${\cal S}$ of $(k+1)$-sets
as a {\em simplicial $k$-multicomplex}, and
its elements as {\em $k$-simplices}. 
In situations where the elements of $\cS$ are known to be distinct, we will refer to $\cS$ as a {\em simplicial $k$-complex}. We emphasise here that all of our main results concern objects which are simplicial $k$-complexes and are thus the set of facets of an abstract simplicial $k$-complex in the standard meaning of this term.

For $0\leq r\leq k$, an {\em $r$-face} of a simplicial $k$-multicomplex $\cS$ is a set $R$ with $|R|=r+1$ and $R\subseteq S$ for some $S\in \cS$. 
We also define the empty set to be the unique $(-1)$-face of $\cS$. 
We will often refer to the $0$-faces of $\cS$ as {\em vertices} and the $1$-faces as {\em edges}.
The {\em graph} $G({\cal S})$ of ${\cal S}$ is the simple graph 
with vertex set $V(\cS)=\bigcup_{S\in \cS} S$  and edge set $E(\cS)=\{uv\,:\,\{u,v\}\subseteq S\in \cS \}$. 
We will refer to a $(k+1)$-clique of $G(\cS)$ which is not a $k$-simplex in $\cS$ as a {\em non-facial $(k+1)$-clique} of $G(\cS)$. 
Given a vertex $v\in V(\cS)$, we use $\cS_v$ to denote the multiset of $k$-simplices of $\cS$ which contain $v$. 
More generally, given a subset $F\subseteq V(\cS)$ of $\cS$, define $\cS_F$ to be the multiset of $k$-simplices of $\cS$ which contain $F$ . We  set $\cS_F = \emptyset$ in cases where $F \not\subseteq V(\cS)$. We also define the {\em link} of $F$ in $\cS$ to be the multiset $lk_\cS(F) = \{ U\setminus F: U \in \cS_F\}$.

For $k\geq 1$ two $k$-simplices $S,S'\in \cS$ are {\em adjacent} if $|S\cap S'|\geq k-1$.
The simplicial $k$-complex $\cS$ is {\em strongly connected} if any two $k$-simplices in $\cS$ are joined by a {\em strong path} i.e., a sequence of $k$-simplices in which consecutive pairs are adjacent. It is a {\em $k$-pseudomanifold} if it is strongly connected and every $(k-1)$-face belongs to exactly two $k$-simplices. 
It is a  {\em triangulated $k$-manifold} if its geometric realisation is a topological $k$-manifold.
These definitions imply that every triangulated connected $k$-manifold is a $k$-pseudomanifold, see for example \cite{K}.

For 
$k\geq 0$
we define the boundary $\partial \cS$ 
of a simplicial $k$-complex $\cS$
to be the set of all $(k-1)$-faces of $\cS$
which are contained in an odd number of $k$-simplices of $\cS$ i.e., $\partial S = \{F \subseteq V(\cS): |F| = k, |\cS_F| \equiv 1 \mod 2\}$. Note that our definition excludes $\partial \cS$ from containing multiple copies of the same $(k-1)$-simplex. 
Thus, for $k\geq 1$,
$\partial S$ is a simplicial $(k-1)$-complex.
For $k=0$, $\partial\cS = \{\emptyset\}$ if $|\cS|$ is odd and $\partial \cS = \emptyset$ is $|S|$ if even.
For example, if ${\cal S}_1=\{\{v_1,v_2,v_3\},\{v_1,v_3,v_4\}\}$ and ${\cal S}_2=\{\{v_1,v_2,v_3\},\{v_1,v_2,v_3\}\}$,
then $\partial \cS_1=\{\{v_1,v_2\}, \{v_2,v_3\}, \{v_3,v_4\},\{v_4,v_1\}\}$ and $\partial \cS_2=\emptyset$.

 The following elementary observations will be used repeatedly throughout the remainder of the paper. 

\begin{lemma}
    \label{lem_symm_diff}
    Suppose that $\cS$ and $\cT$ are simplicial $k$-complexes. 
    \begin{itemize}
        \item[(a)] For any $F \subseteq V(\cS) \cup V(\cT)$, $(\cS \triangle \cT)_F = \cS_F \triangle \cT_F$.
        \item[(b)] $\partial (\cS \triangle \cT)
        = \partial \cS \triangle \partial \cT$.
    \end{itemize}
\end{lemma}

\begin{proof}
    Suppose that $U \in \cS \cup \cT$ and $F \subseteq U$. Then 
    $U \in (\cS \triangle \cT)_F \Leftrightarrow U \in \cS \triangle \cT \Leftrightarrow U \in \cS_F \triangle \cT_F$ which proves (a). 
    For (b), let $F$ be a $(k-1)$-face of  $\cS \cup \cT$. Then, using (a), $F \in \partial (\cS \triangle \cT) \Leftrightarrow |\cS_F \triangle \cT_F| \equiv 1 \mod 2 \Leftrightarrow |\cS_F| + |\cT_F| \equiv 1 \mod2 \Leftrightarrow F \in \partial \cS \triangle \partial \cT_F$.
\end{proof}

For $k\geq 0$ a simplicial $k$-multicomplex ${\cal S}$ is said to be a {\em simplicial $k$-cycle} if 
$\partial {\cal S}=\emptyset$ i.e., $\cS$ is 
a $k$-cycle in the  corresponding chain complex. 
For example a simplicial 1-cycle is the edge set of an even multigraph (without loops) i.e., a multigraph in which each vertex has even degree.
We say that $\cS$ is a {\em simplicial $k$-circuit} if it is a non-empty simplicial $k$-cycle and no proper nonempty subfamily of $\cS$ is a simplicial $k$-cycle. 
Note that a simplicial $k$-multicomplex consisting of two copies of the same $k$-simplex is a simplicial $k$-circuit. Such a simplicial $k$-circuit is called {\em trivial}.
Hence a non-trivial simplicial $k$-circuit cannot contain two copies of the same $k$-simplex. 
We emphasise in particular that a non-trivial simplicial $k$-circuit is a simplicial $k$-complex.
For example, every non-trivial simplicial 1-circuit is the edge set of a `cycle graph' i.e., a connected 2-regular graph with at least three vertices.
The set of facial triangles in a triangulated connected 2-manifold is an example of a simplicial 2-circuit.

Simplicial $k$-cycles have the nice property that, 
if $\cS$ and $\cS'$ are simplicial $k$-cycles with $\cS'\subseteq \cS$, then $\cS\sm \cS'$ is a simplicial $k$-cycle. This implies that every nonempty simplicial $k$-cycle can be partitioned into simplicial $k$-circuits. We can use this observation to show that every pseudomanifold is a simplicial circuit.

\begin{lemma}\label{lem:pseudo}
Let ${\cal S}$ be a $k$-pseudomanifold for some $k \geq 1$. 
Then   ${\cal S}$  is a simplicial $k$-circuit.
\end{lemma}
\begin{proof}
Since $\cS$ is a {$k$-pseudomanifold}, every $(k-1)$-face belongs to exactly two $k$-simplices of $\cS$. This implies that $\partial \cS=\emptyset$ so $\cS$ is a simplicial $k$-cycle and hence $\cS$ can be partitioned into 
simplicial $k$-circuits. In particular, $\cS$ contains a 
simplicial $k$-circuit $\cC$. 

Suppose, for a contradiction, that $\cS\neq\cC$. 
Then we can choose two $k$-simplices $S\in \cC$ and $S'\in \cS\sm \cC$. Let $\cT$ be the set of all $k$-simplices in $\cS$ which can be reached by a strong path beginning at $S$. The facts that $\cC$ is a simplicial $k$-circuit and every $(k-1)$-face of $\cS$ belongs to exactly two $k$-simplices of $\cS$
imply that $\cT\subseteq \cC$. This in turn implies that $S'\not\in \cT$ and contradicts the hypothesis that $\cS$ is a pseudomanifold.
\end{proof} 

Note that for $k\geq 2$ the class of simplicial $k$-circuits is strictly bigger than the class of $k$-pseudomanifolds. For example, suppose that $\cT$ is a triangulation of the 2-sphere and suppose that $u,v \in V(\cT)$ are distinct vertices with exactly one common neighbour $w$. We can construct a non-planar simplicial 2-circuit by identifying $u$ and $v$. More precisely, let $\cT'$ be obtained from $\cT$ by replacing every 2-simplex $S$ that contains $v$ with $S-v+u$. Then, one readily checks that $\cT'$ is a simplicial 2-circuit and that the 1-face $\{u,w\}$ is contained in four $2$-simplices. Thus $\cT'$ is not a 2-pseudomanifold.  

Our next two lemmas are elementary observations that we will use repeatedly in the rest of the paper. The first
corresponds to the standard result in algebraic topology that says that the square of the boundary homomorphism in the chain complex of a simplicial set is zero. We give an elementary proof for completeness.
\begin{lemma}\label{lem:1}
Let ${\cal S}$ be a 
simplicial $k$-multicomplex with $k \geq 1$. 
Then   $\partial {\cal S}$  is a simplicial $(k-1)$-cycle.
\end{lemma}
\begin{proof}
    Let $A$ be a $(k-2)$-face of $\partial \cS$ and let 
$X_A = \{(B,C): C \in \cS, A \subseteq B \subseteq C, |B| = k\}$. Given $C$ there are exactly two choices for $B$, so $|X_A|$ is even. On the other hand, given $B$ there are an odd number of choices for $C$ if and only if $B \in \partial \cS$. Therefore $|\{B \in \partial \cS: A \subseteq B\}|  \equiv |X_A| \equiv 0 \mod 2$ as required.
\end{proof}

\begin{lemma}
    \label{lem:links}
    Suppose that $\cS$ is a simplicial $k$-cycle and $F\subseteq V(\cS)$ is an $i$-face of $\cS$ for some $0\leq i \leq k-1$. Then $lk_\cS(F)$ is a simplicial $(k-i-1)$-cycle.
\end{lemma}
\begin{proof}
    Let $W$ be a $(k-i-2)$-face of $lk_\cS(F)$. 
    The mapping $Y \mapsto Y \cup F$ is a bijection between the $(k-i-1)$-simplices of $lk_\cS(F)$ that contain $W$ and the $k$-simplices of $\cS$ that contain $W \cup F$. 
    The lemma follows immediately from this.
\end{proof}

Let $\cK_k$ be the simplicial $k$-complex obtained by taking every $(k+1)$-subset of a set of size $k+2$. For example,
$\cK_1$ is the set of edges of a triangle and $\cK_2$ is the set of faces of a tetrahedron.
It is not difficult to see that $\cK_k$ is a pseudomanifold and hence  
$\cK_k$ is  a simplicial $k$-circuit by Lemma~\ref{lem:pseudo}.
We will continue to use $K_k$ for the complete graph on $k$ vertices. Thus $G(\cK_k) = K_{k+1}$.

\begin{lemma}\label{lem:1.1.1}
For $k\geq 0$, let ${\cal S}$ be a nonempty simplicial $k$-cycle without repeated $k$-simplices. 
Then $|V(\cS)|\geq k+2$, and $|V(\cS)|=k+2$ if and only if 
$\cS$ is isomorphic to $\cK_k$.
\end{lemma}
\begin{proof}
If $|V(\cS)|\leq k+1$ then $\cS$ has at most one $k$-simplex and so cannot be a nonempty simplicial $k$-cycle, contradicting our hypothesis. So $|V(\cS)|\geq k+2$.
Now suppose that $|V(\cS)| = k+2$. Pick any $U \in  \cS$ and let $w$ be the unique vertex in  $V(\cS)\setminus U$. Since $\cS$ is a $k$-cycle that does not contain any repeated $k$-simplex, it follows that, for each $u \in U$, there must be a $k$-simplex of $\cS$, different from $U$, that contains $U-u$. In other words $U-u+w \in \cS$. Therefore every $(k+1)$-subset of $V(\cS)$ is a $k$-simplex of $\cS$ and so $\cS$ is isomorphic to $\cK_k$. 
\end{proof}

Let $\cL_k$ be the simplicial $k$-complex obtained from  the union of two copies of $\cK_k$ with exactly one $k$-simplex in common by  deleting their common $k$-simplex.
For example, $\cL_1$ is the edge set of a 4-cycle, and $\cL_2$ is the set of facets of a hexahedron. 
It is not difficult to see that $\cL_k$ is a pseudomanifold and hence  
is  a simplicial $k$-circuit by Lemma~\ref{lem:pseudo}.
We will show that $\cL_k$ is the unique simplicial $k$-circuit on $k+3$ vertices whose graph is not complete.

\begin{lemma}\label{lem:1.1.3}
For $k\geq 1$, let ${\cal S}$ be a non-trivial simplicial 
$k$-cycle without repeated $k$-simplices.
Suppose that $|V(\cS)|=k+3$.
If $k\leq 2$, then $\cS$ is isomorphic to $\cL_k$.
If $k\geq 3$, then either $G(\cS)\cong K_{k+3}$ or $\cS$ is isomorphic to $\cL_k$.
\end{lemma}
\begin{proof}
Suppose that $G(\cS)\not\cong K_{k+3}$.
Choose two non-adjacent vertices $u$ and $v$ in $G(\cS)$ and put $X=V(\cS)\sm\{u,v\}$.
Lemma~\ref{lem:1} tells us that $\partial (\cS_u)$ is a simplicial $(k-1)$-cycle. 
Note that since $\cS$ is a simplicial $k$-cycle that does not have any repeated $k$-simplex, it follows that $\partial (\cS_u) = \{U-u:  U \in \cS_u\}$. Also $v \not\in V(\partial (\cS_u))$ since $uv \not\in E(\cS)$. Therefore $V(\partial (\cS_u)) \subseteq X$.
Lemma~\ref{lem:1.1.1} now implies that equality holds and $\partial (\cS_u)$ is isomorphic to $\cK_{k-1}$. This in turn gives $\partial (\cS_u)=\partial \{X\}$ and hence 
$\cS_u\triangle\{ X\}$ 
is a simplicial $k$-cycle.
Lemma~\ref{lem:1.1.1} now implies that $\cS_u\triangle\{ X\}$ is isomorphic to $\cK_k$. By symmetry we also have that $\cS_v\triangle\{ X\}$ is isomorphic to $\cK_k$. Hence $\cS$ is isomorphic to $\cL_k$.

It is straightforward to check that $\cL_k$ is the unique simplicial $k$-circuit with $k+3$ vertices when $k=1,2$.
\end{proof}

A fundamental  property of graphs of triangulated connected 2-manifolds is that  
they are $3$-connected and every 3-vertex-separator is a non-facial 3-clique.
Our next lemma shows that this property extends to non-trivial simplicial $k$-circuits.
\begin{lemma}\label{lem:1.1.2}
Let ${\cal S}$ be a non-trivial simplicial $k$-circuit with $k\geq 2$. 
\begin{enumerate}
    \item[(a)]  $G(\cS)$ is $(k+1)$-connected.
    \item[(b)] If $X$ is a $(k+1)$-vertex separator of $G(\cS)$ then $X$ is a non-facial $(k+1)$-clique of $G(\cS)$ and there are simplicial $k$-circuits $\cS'$ and $\cS''$ such that $\cS' \cap \cS'' = \{X\}$ and $\cS = \cS' \triangle \cS''$.
\end{enumerate}

\end{lemma}
\begin{proof}
Suppose $G(\cS)=G_1\cup G_2$ with $V(G_1)\sm  V(G_2)\neq \emptyset\neq  V(G_2)\sm  V(G_1)$. Let $X=V(G_1)\cap V(G_2)$. Let $\cS_1=\{S\in \cS:S\subseteq V(G_1)\}$ and $\cS_2=\cS\triangle \cS_1$ and note that $\cS_1,\cS_2$ are both proper subsets of $\cS$. Then $\cS=\cS_1\triangle \cS_2$ and $V(\cS_1)\cap V(\cS_2)\subseteq X$.
Since $\cS$ is a simplicial $k$-circuit, we have $\partial \cS=\emptyset\neq \partial \cS_1$. This implies that  $\partial \cS_1=\partial \cS_2\neq \emptyset$ and $V(\partial \cS_1)\subseteq X$. Since $\partial \cS_1$ is a  simplicial $(k-1)$-cycle without repeated $(k-1)$-simplices by  Lemma~\ref{lem:1}, we can  use Lemma~\ref{lem:1.1.1} and the fact that every simplicial $(k-1)$-cycle contains a simplicial $(k-1)$-circuit to deduce that $|X|\geq |V(\partial \cS_1)|\geq k+1$ and hence $G(\cS)$ is $(k+1)$-connected. This proves (a).

Suppose $|X|=k+1$. Then Lemma~\ref{lem:1.1.1} tells us that  $\partial \cS_1\cong\cK_{k-1}$ and, since $k \geq 2$, $X$ is a clique in $G(\cS)$. In addition, since  $\partial \{X\}\cong\cK_{k-1}$, we have $\partial \cS_1=\partial \{X\}$ and hence   $\partial (\cS_1\triangle \{X\})=\emptyset$. The hypothesis that ${\cal S}$ is simplicial $k$-circuit now gives $X\not\in\cS$ so $X$ is a  non-facial $k$-clique of $G(\cS)$. Finally Let $S' = \cS_1 \cup \{X\}$ and $\cS'' = \cS_2 \cup \{X\}$. Then $\cS'$ and $\cS''$ are non-empty simplicial $k$-cycles, $\cS = \cS' \triangle \cS''$ and $\cS'\cap \cS'' = \{X\}$. Suppose, for a contradiction that  $\cR \subsetneq \cS'$ is a nonempty simplicial $k$-cycle.  Then $\cS'\sm \cR'$ is also nonempty simplicial $k$-cycle and one of $\cR$ or $\cS'\sm \cR$ is contained in $\cS_1$, contradicting the
fact that $\cS$ is a simplicial $k$-circuit. Therefore $\cS'$, and similarly $\cS''$, is a simplicial $k$-circuit.
\end{proof}

Since our main aim is to characterise global rigidity for graphs of simplicial $k$-circuits when $k\geq 2$, henceforth we will assume (unless otherwise stated) that this is the case.  We observe in passing that the only graphs of simplicial $1$-circuits that are (globally) rigid in $\R^2$ are $K_2$ and $K_3$.


\subsection{Contractions}
\label{sec:contractions}
We define a contraction operation for two vertices in a simplicial multicomplex.
This operation will be used in inductive proofs throughout this paper. We will see that simplicial cycles are closed under this operation. The same is not true for  simplicial circuits but Fogelsanger showed that  this problem can be circumvented  by choosing a partition of the contracted complex into simplicial circuits and then using the pre-images of these simplicial circuits  to construct the Fogelsanger decomposition of the original circuit. We will obtain some preliminary results on these pre-images in this subsection and use them to define the Fogelsanger decomposition in the next subsection. Although the presentation is simplified, the results in both subsections are essentially due to Fogelsanger~\cite{F}.

Given  a simplicial $k$-multicomplex 
$\cS$ and distinct vertices $u,v\in V(\cS)$, 
we can obtain a new 
simplicial $k$-multicomplex
$\cS/uv$  from $\cS$ by deleting all $k$-simplices in $\cS_{\{u,v\}}$, and replacing  every $k$-simplex $S\in \cS$ which contains $v$ but not $u$ by a new $k$-simplex $S'=S-v+u$.  We say that $\cS/uv$ is obtained from $\cS$ by {\em contracting $v$ onto $u$}. 
For each $k$-simplex $S\in \cS$ with $\{u,v\}\not\subseteq \cS$, we  denote the $k$-simplex in $\cS/uv$ corresponding to $S$ by $\gamma_{uv}(S)$.
Thus $\gamma_{uv}$  is a  bijection between the $k$-simplices of $\cS$ which do not contain $\{u,v\}$ and the $k$-simplices of $\cS/uv$.
Note that if  $S_1,S_2\in \cS$ with $u\in S_1$, $v\in S_2$ and $S_1-u=S_2-v$ then 
$\gamma_{uv}(S_2)$ and $\gamma_{uv}(S_1)$ are two distinct copies of the same $k$-simplex in  $\cS/uv$. 
Thus our contraction operation may transform a simplicial $k$-complex into a simplicial $k$-multicomplex -- this is precisely the reason why we allow repetitions of $k$-simplices. We could of course delete repetitions from our new family but we prefer not to do this because of the following lemma and corollary which we will use repeatedly.

\begin{lemma}
    \label{lem_cont_k-1}
    Let $\cS$ be a simplicial $k$-multicomplex, $u,v \in V(\cS)$ be distinct vertices  and $F \subseteq V(\cS)-v$, $|F| = k$. Then
    \begin{equation}
    \label{eqn_mult_cont}
        |(\cS/uv)_F| = \left\{\begin{array}{ll}
             |\cS_F| & \text{ if }u \not\in F, \\
             |\cS_F| + |\cS_{F-u+v}| -2|\cS_{F+v}| & \text{ if }u \in F.
            \end{array}\right.
    \end{equation}
\end{lemma}
\begin{proof}
    If $u\not\in F$, then $\cS_F$ is contained in $\gamma_{uv}^{-1}(\cS/uv)$ and it follows that $\gamma_{uv}$ induces a bijection between $\cS_F$ and $(\cS/uv)_F$. 
    On the other hand, if $u \in F$ then 
    $(\cS/uv)_F$ is in one-to-one correspondence with the multiset of $k$-faces of $\cS$ that contain exactly one of $F$ or $F-u+v$.
\end{proof}

\begin{corollary}
    \label{cor:cont_preserve}
    Suppose that $\cS$ is a simplicial $k$-cycle for some $k \geq 1$  and that $u,v \in V(\cS)$ are distinct vertices. Then $\cS/uv$ is a simplicial $k$-cycle.
\end{corollary}

\begin{proof}
    Pick any $F\subseteq V(\cS/uv)$ with $|F|=k$. 
    Since $\cS$ is simplicial $k$-cycle the terms $|\cS_F|$ and $|\cS_{F-u+v}|$ on the right hand side of (\ref{eqn_mult_cont}) are even. Hence, $|(\cS/uv)_F|$ is even. 
\end{proof}


We will also consider the analogous contraction operation for two vertices $u,v$ in a graph $G$:
we construct a new simple graph $G/uv$  from $G-v$ by  adding  a new edge from $u$ to every vertex in $N_G(v)\sm (N_G(u)\cup\{u\})$. 

We often denote $G/uv$ by $G/e$ when $e=uv$ is an edge of $G$ and say that $G/e$ is obtained from $G$ by {\em contracting $e$ onto $u$}.
These definitions are illustrated in Figure \ref{fig:1}.
Note that our definition of edge contraction for graphs does not create multiple edges.

\begin{figure}
\begin{center}
\begin{minipage}{0.45\textwidth}
\centering
\includegraphics[scale=1]{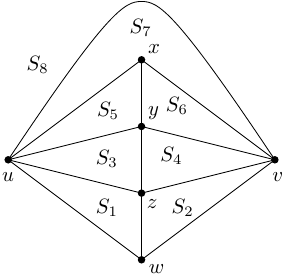}
\par
(a)
\end{minipage}
\begin{minipage}{0.45\textwidth}
\centering
\includegraphics[scale=1]{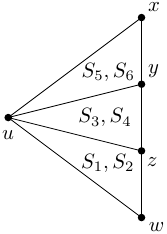}
\par
(b)
\end{minipage}
\end{center}
\caption{Figure (a) shows a simplicial 2-circuit $\cT=\{S_1, S_2, S_3, S_4, S_5, S_6, S_7, S_8\}$ defined by a triangulation of the 2-sphere. 
    Figure (b) shows the simplicial 2-cycle $\cT/uv=\{S_1, S_2, S_3, S_4, S_5, S_6\}$ obtained by contracting $uv$ onto $u$. It contains three distinct triangles, each of multiplicity two, and can be viewed   as the family of faces of the  triangulation of the `pinched surface' consisting of three triangulated spheres $\cT_1=\{S_1,S_2\}, \cT_2=\{S_3,S_4\}, \cT_3=\{S_5,S_6\}$ in which $\cT_1$ and $\cT_3$ are attached to $\cT_2$ along the edges $uy$ and $uz$, respectively. We have $G(\cT)/uv=G(\cT/uv)$.
}
\label{fig:1}
\end{figure}

The edge-contraction operations for complexes and graphs are compatible in the sense that $G(\cS/e)=G(\cS)/e$ whenever $\cS$ has the property that every $f\in E(\cS)-e$ lies in at least one $k$-simplex in $\cS$ which does not contain $e$. 
Our next result implies that this property always holds when $\cS$ is  a nonempty simplicial $k$-cycle without repeated $k$-simplices.

\begin{lemma}\label{lem:1.1.4}
For $k\geq 2$ let $\cS$ be a nonempty simplicial $k$-cycle that does not contain any repeated $k$-simplex. Let $uv \in E(\cS)$ and $x,y \in V(\cS)$ be  distinct vertices such that $\{x,y\} \neq \{u,v\}$. 
Then 
\begin{itemize}
    \item[(a)] $\cS$ contains a $k$-simplex which contains $\{u,v\}$ but not $\{x,y\}$.
    \item[(b)] $|N_{G(\cS)}(u)\cap N_{G(\cS)}(v)|\geq k$, with equality if and only if $lk_\cS(uv) \cong \cK_{k-2}$.
\end{itemize}
\end{lemma}
\begin{proof}

Let $U$ be a $k$-simplex of $\cS$ that contains $\{u,v\}$. We may assume that $U$ also contains $\{x,y\}$.
Since $\{u,v\}\neq \{x,y\}$, we may assume, by symmetry, that
$x \not\in \{u,v\}$. Since $\cS$ is a cycle without repeated $k$-simplices there must be a $k$-simplex $U'$, distinct from $U$ that contains the $(k-1)$-face $U-x$. This proves (a).

For (b), we observe that $lk_\cS(uv)$ is a nonempty simplicial $(k-2)$-cycle without repeated $(k-2)$-simplices and $V(lk_\cS(uv)) \subseteq N_{G(\cS)}(u) \cap N_{G(\cS)}(v)$. Now statement (b) follows from Lemma~\ref{lem:1.1.1}

\end{proof}

The idea behind the Fogelsanger decomposition of a simplicial $k$-circuit $\cS$ is that, if $\cS/e$ is not a simplicial $k$-circuit for $e=uv\in E(\cS)$, we choose a  partition $\{\cS_1', \dots, \cS_m'\}$ of $\cS/e$ into simplicial $k$-circuits
and consider the pre-image $\cS_i:=\gamma_e^{-1}(\cS_i')$.
Although each $\cS_i$ may not be a simplicial $k$-circuit, Fogelsanger found an explicit way to transform it into a simplicial circuit by constructing a set of $k$-simplices $\cS_i^*$, each containing $\{u,v\}$,  such that $\cS_i  \triangle \cS^*_{i}$ is a simplicial $k$-circuit.
Lemma~\ref{lem:1.5} describes this transformation.

First we record an elementary observation that will be used in the proof. 

\begin{lemma}
    \label{lem:stars}
Suppose that $\cT$ is a simplicial $k$-complex and $\emptyset \neq Z \subseteq V(\cT)$ such that $\cT = \cT_Z$. Then, for any $K \subseteq V(\cT)$ such that $Z \subseteq K$,  
\begin{equation*}
\label{eqn_Kchar}
K \in \cT 
\Leftrightarrow K-z_0 \in \partial \cT\text{ for some }z_0 \in Z
\Leftrightarrow K-z \in \partial \cT\text{ for every }z \in Z.
\end{equation*}
\end{lemma}
\begin{proof}
    If $K \in \cT$ and $z \in Z \subseteq K$ then, since $\cT = \cT_Z$, $K$ is the unique element of $\cT$ than contains $K-z$ and so $K-z \in \partial \cT$. On the other hand, if $K \subseteq V(\cT)$ satisfies $Z \subseteq K$ and $K -z$ is a $(k-1)$-face of $\cT$ then, since $\cT = \cT_Z$, it follows that $K \in \cT$.
\end{proof}

\begin{lemma}\label{lem:1.5}
Suppose ${\cal S}$ is a simplicial $k$-complex and $u,v\in V(\cS)$. 
Let  
$${\cal S}^\dagger=\{ K\subseteq V(\cS): \{u,v\}\subseteq  K\mbox{ and $K-u,K-v$ are $(k-1)$-faces of $\cS$}\}$$
and put ${\cal S}^*=\{ K\in {\cal S}^\dagger: K-u,K-v\in  \partial \cS\}$.
\begin{itemize}
\item[(a)] If $\cS\triangle \cS'$ is a  simplicial $k$-cycle for some $\cS'\subseteq \cS^\dagger$ then  $\cS'=\cS^*$.
\item[(b)] $\cS/uv$ is a simplicial $k$-cycle  if and only if $\cS\triangle \cS^*$ is a  simplicial $k$-cycle.
\item[(c)] If $\cS/uv$ is a simplicial $k$-circuit  then $\cS\triangle \cS^*$ is a simplicial $k$-circuit.
\end{itemize}
\end{lemma}
\begin{proof}

\smallskip
\noindent  
(a) Suppose $\cS\triangle \cS'$ is a  simplicial $k$-cycle for some $\cS'\subseteq \cS^\dagger$.  Then $\partial(\cS\triangle \cS')=\emptyset$ so $\partial \cS=\partial \cS'$. 
By definition $\cS^\dagger$ is a simplicial $k$-complex, so $\cS^*$ and $\cS'$ must also be simplicial $k$-complexes. Now, applying Lemma~\ref{lem:stars}, with $\cT = \cS'$ and $Z = \{u,v\}$, 
we have that $K\in \cS'$ if and only if $K-u, K-v\in \partial \cS'=\partial \cS$. Hence, by the definition of $\cS^*$, $K \in \cS'$ if and only if $K\in \cS^*$, and $\cS'=\cS^*$ follows.

\smallskip

\noindent 
(b) We first suppose that $\cS\triangle \cS^*$ is a simplicial $k$-cycle. Then $\cS/uv$ is a simplicial $k$-cycle since $\cS/uv=(\cS\triangle \cS^*)/uv$ and contraction preserves the property of being a simplicial $k$-cycle.

We next suppose that $\cS/uv$ is a simplicial $k$-cycle. 
 We will show that $\partial (\cS \triangle \cS^*)$ is empty.
Let $F$ be a $(k-1)$-face of $\cS\triangle \cS^*$.   

 First consider the case $\{u,v\}\cap F = \emptyset$. Then, using Lemma~\ref{lem:stars} we see that $F \not\in \partial \cS^*$. Using Lemma~\ref{lem_cont_k-1} and the fact that $\cS/uv$ is a simplicial $k$-cycle, we have 
$|\cS_F| = |(\cS/uv)_F| \equiv 0 \mod 2$. So $F \not\in \partial \cS$ and so $F \not\in\partial(\cS \triangle \cS^*)$.

Next suppose that $\{u,v\} \cap F = \{u\}$.
Using Lemma~\ref{lem_cont_k-1} and the fact that 
$\cS/uv$ is a simplicial $k$-cycle we see that 
$|\cS_{F-u+v}|+|\cS_F| \equiv 0 \mod 2$. 
Therefore, $F \in \partial \cS \Leftrightarrow F-u+v \in \partial \cS$.
By the definition of $\cS^*$, this is equivalent to  $F+v\in \cS^*$.
By Lemma 3.11, this is equivalent to saying that $F+v\in \partial \cS^*$ and $F-u+v\in \partial \cS^*$.
It follows that 
$F \not\in \partial (\cS \triangle \cS^*)$
and $F-u+v \not\in \partial (\cS \triangle \cS^*)$.

Now suppose that $\{u,v\} \cap F =\{v\}$. Then applying the argument of the previous paragraph to $F-v+u$ shows that $F \not\in \partial (\cS \triangle \cS^*)$.

The previous three paragraphs tell us that every $(k-1)$-face in 
$\partial(\cS\triangle \cS^*)$ contains $\{u,v\}$. 
Now applying Lemma~\ref{lem:stars} to $\partial (\cS \triangle \cS^*)$ we see that $\partial (\partial (\cS \triangle \cS^*))$ is nonempty, contradicting Lemma~\ref{lem:1}. 
\smallskip

\noindent
(c) Suppose $\cS/uv$ is a simplicial $k$-circuit.
Then $\cS\triangle \cS^*$ is a simplicial $k$-cycle by (b). Let $\cR'\subseteq \cS\triangle \cS^*$ be a simplicial $k$-circuit. 
Let $\cR'' = \cR/uv $ if $\{u,v\} \subseteq V(\cR')$, or $\cR'' = \gamma_{uv}(\cR')$ if $\{u,v\} \not\subseteq V(\cR')$. Then 
$\cR''\subseteq \cS/uv$
is a simplicial $k$-cycle  and, since $\cS/uv$ is a simplicial $k$-circuit, 
$\cR''=\cS/uv$. 
This tells us that the sets of $k$-simplices of $\cR'$ and $\cS$ which do not contain $\{u,v\}$ are the same and hence that $\cR'=\cS\triangle \cS'$ for some $\cS'\subseteq \cS^\dagger$. We can now use (a) to deduce that $\cS'=\cS^*$ and hence $\cR'=\cS\triangle \cS^*$. 
\end{proof}

We will frequently apply Lemma~\ref{lem:1.5} to a subcomplex of a larger simplicial complex. Consider, for example, the simplicial $2$-circuit $\cT$ in Figure 
\ref{fig:1} and let $\cS=\{S_1,S_2\}$. Then $\cS/uv$ is a (trivial) simplicial 2-circuit. We have $\cS^*=\{S_8,K_1\}$ where $K_1=\{u,v,z\}$ and 
$\cS\triangle \cS^* =\{S_1, S_2, S_8, K_1\}$ is a simplicial $k$-circuit.

\subsection{Fogelsanger decompositions}\label{sec:fog}
We can now give Fogelsanger's  remarkable decomposition 
result for simplicial $k$-circuits.
Since this is a key tool in our analysis, we will provide a detailed description of its properties. 
We will see that Fogelsanger's Rigidity Theorem follows easily from  these properties.

\paragraph{Definition.} Let ${\cal S}$ be a non-trivial simplicial $k$-circuit and  
$uv \in E(\cS)$.
Then $\cS/uv$ is a simplicial $k$-cycle and we may choose a partition $\{\cS_1',\ldots,\cS_m'\}$ of $\cS/uv$ into simplicial $k$-circuits. 

For all $1\leq i\leq m$, let 
$\cS_i=\gamma_{uv}^{-1}(\cS_i')$. 
Then $\{u,v\} \subseteq V(\cS_i)$ since, if not, $\cS_i$ would be a simplicial $k$-circuit properly contained in $\cS$, contradicting the fact that $\cS$ is a simplicial $k$-circuit.
Let 
\begin{equation*}
\cS_{i}^*=\{ K\subseteq V(\cS_i): \{u,v\}\subseteq K \mbox{ and }K-u,K-v\in  \partial \cS_i\} \mbox{ and }\cS_i^+=\cS_i\triangle \cS_{i}^*.
\end{equation*}
Then $\cS_i^+=\cS_i\cup \cS_{i}^*  $, since $\cS_i\cap \cS_{i}^*=\emptyset$, and  $\cS_i^+$ is a 
simplicial $k$-circuit  by Lemma~\ref{lem:1.5}(c).
We say that $\{\cS_1^+,  \dots, \cS_m^+\}$ is a {\em Fogelsanger decomposition} for $\cS$ with respect to $uv$. 

\medskip

The definition above can also be used for nonadjacent vertices $u,v$ but we will not need this case.
We also note that, in general, the Fogelsanger decomposition of $\cS$ need not be unique since the partition $\{\cS_1',\ldots,\cS_m'\}$ of $\cS/uv$ into simplicial $k$-circuits need not be unique.

\paragraph{Example 1.} Consider the simplicial 2-circuit $\cT$ on the left of Figure \ref{fig:1}. The simplicial 2-cycle $\cT/e$ on the right can be decomposed into  three (trivial) simplicial 2-circuits: 
$\cT_1'=\{S_1,S_2\}$; $\cT_2'=\{S_3,S_4\}$; $\cT_3'=\{S_5,S_6\}$. 
Then
 $\cT_1=\{S_1,S_2\}$; $\cT_2=\{S_3,S_4\}$;
$\cT_3=\{S_5,S_6\}$. This gives rise to a  Fogelsanger decomposition for $\cT$
 with respect to  $\{u,v\}$ consisting of three simplicial 2-circuits: $\cT_1^+=\{S_1,S_2,S_8,K_1\}$ where $K_1=\{u,v,z\}$; $\cT_2^+=\{S_3,S_4,K_1,K_2\}$ where $K_2=\{u,v,y\}$;
$\cT_3^+=\{S_5,S_6,S_7,K_2\}$. 

\medskip

\paragraph{Example 2.} Figure~\ref{fig:fog} is an example of a Fogelsanger decomposition taken from \cite{F}.
(a) shows a simplicial 2-circuit $\cS$ given by the set of faces in a triangulation of the torus with a pinched singular point.
(b) shows the simplicial 2-cycle $\cS/uv$ obtained by contracting $uv$ onto $u$. It is given by the  multiset of faces  in a triangulation of the torus with a pinched singular point and two pinched singular edges, $ux$ and $uz$.
The face $\{u,x,z\}$ has multiplicity two in $\cS/uv$ and all other faces have multiplicity one.
A partition of $\cS/uv$ into three simplicial 2-circuits is shown in (c), and
(d) shows the corresponding Fogelsanger decomposition of $\cS$.

\begin{figure}[t]
\centering
\begin{minipage}{0.45\textwidth}
\centering
\includegraphics[scale=0.8]{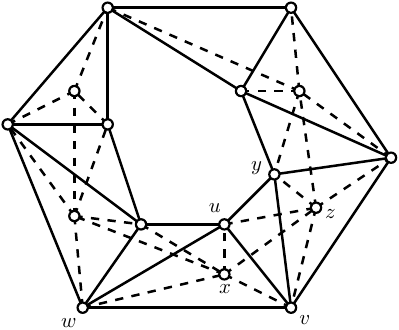}
\par(a)
\end{minipage}
\begin{minipage}{0.45\textwidth}
\centering
\includegraphics[scale=0.8]{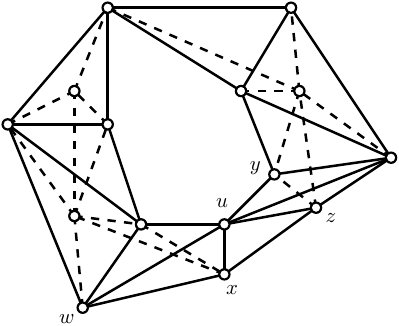}
\par(b)
\end{minipage}

\medskip

\begin{minipage}{0.45\textwidth}
\centering
\includegraphics[scale=0.8]{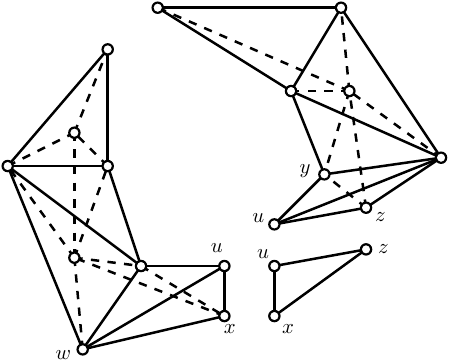}
\par(c)
\end{minipage}
\begin{minipage}{0.45\textwidth}
\centering
\includegraphics[scale=0.8]{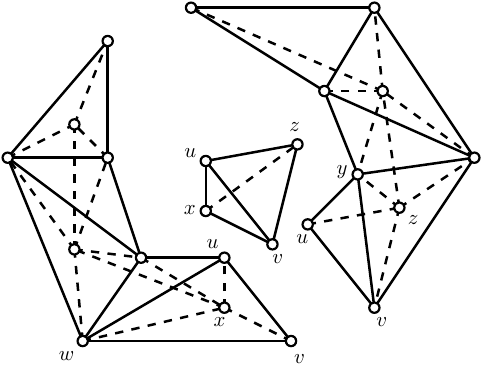}
\par(d)
\end{minipage}
\caption{An example of a Fogelsanger decomposition from~\cite{F}.}
\label{fig:fog}
\end{figure}
\medskip

Our next lemma gives some key properties of Fogelsanger decompositions.

\begin{lemma}\label{lem:3}
Suppose  $\cS$ is a non-trivial simplicial $k$-circuit, $uv \in E(\cS)$. Let $\{\cS_1^+, \dots, \cS_m^+\}$  be a Fogelsanger decomposition of $\cS$ with respect to $uv$.
Then: 
\begin{enumerate}[label=(\alph*)]\setlength{\parskip}{0.03cm}\setlength{\itemsep}{0.03cm}
\item \label{en1:0} $\cS_i^+/uv$ is a simplicial $k$-circuit  for all $1\leq i\leq m$;
\item \label{en1:a} $\cS_i^+$ is a non-trivial simplicial $k$-circuit for all $1\leq i\leq m$ and each $K\in \cS_i^+\sm\cS$ is a clique of  $G(\cS)$ which contains $\{u,v\}$;
\item \label{en1:b} 
each $k$-simplex of $\cS$ that does not contain $\{u,v\}$ lies in a unique $\cS_i^+$. Moreover $\cS = \triangle_{i=1}^m \cS_i^+$;
\item \label{en1:c} $uv \in E(\cS_i^+)$  for all $1\leq i\leq m$ and $\bigcup_{i=1}^m E(\cS_i^+)=E(\cS)$;
\item \label{en1:d} for all non-empty $ I\subsetneq \{1,2,\ldots,m\}$, there exists $j\in \{1,2,\ldots,m\}\sm I$  and a non-facial $(k+1)$-clique $K$ of $G(\cS)$ such that $K\in (\triangle_{i\in I}\cS_i ^+)\cap \cS_{j}^+$ and $\{u,v\}\subseteq K$.
\end{enumerate}
\end{lemma}
\begin{proof}
We adopt the notation used in the definition of a Fogelsanger decomposition. \\[1mm]
\ref{en1:0} By construction, $\cS_i^+/uv = \cS_i/uv = \cS_i'$ which is a simplicial $k$-circuit.
\\[1mm]
\ref{en1:a} The facts that each $\cS_i^+$ is a simplicial $k$-circuit and each $K\in \cS_i^+\sm\cS$ is a non-facial clique of $G(\cS)$ which contains $\{u,v\}$ follow from the definition of a Fogelsanger decomposition. It only remains to show each $\cS_i^+$ is non-trivial. We have $|\cS_i|=|\cS_i'|\geq 2$ since $\cS_i'$ is a simplicial $k$-circuit, and $\cS_i$ does not contain multiple copies of the same $k$-simplex since $\cS_i\subseteq \cS$. Since 
$\cS_i^+=\cS_i\triangle \cS_i^*=\cS_i\cup \cS_i^*$ by definition, we have $\cS_i\subseteq \cS_i^+$. Hence $\cS_i^+$ is a simplicial $k$-circuit with at least two distinct $k$-simplices so is non-trivial.
\\[1mm] 
\ref{en1:b} Choose $S\in \cS$ with $\{u,v\}\not\subseteq S$. Then $\gamma_{uv}(S)$ 
belongs to a unique simplicial $k$-circuit $\cS_i'$ in the partition of $\cS/e$ into simplicial $k$-circuits. Hence $S\in \cS_i$ and  $S\not\in \cS_j$ for all $i\neq j$.  Furthermore $S\not \in \cS_j^*$ for all $1\leq j\leq m$ since each $k$-simplex of $\cS_j^*$ contains $\{u,v\}$. Hence $S\in \cS_i^+$ and  $S\not\in \cS_j^+$ for all $i\neq j$.

For the second part of \ref{en1:b} let $\cR=\cS\triangle (\triangle_{i=1}^m \cS_i^+)$. Each $k$-simplex in $\cR$ contains $\{u,v\}$ by the previous paragraph. Also $\cR$ is a simplicial $k$-cycle without repeated $k$-simplices. If $\cR \neq \emptyset$ then, by Lemma~\ref{lem:stars}, $\partial \cR \neq \emptyset$ contradicting the fact that $\cR$ is a simplicial $k$-cycle. 
Therefore $\cR = \emptyset$ and $\cS =  \triangle_{i=1}^m \cS_i^+$.
\\[1mm]
\ref{en1:c} 
If $m=1$ then  ${\cal S}={\cal S}_1^+$ by part (c)  and \ref{en1:c} holds trivially. Hence we may assume that $m\geq 2$.
We first show that $uv \in E(\cS_i^+)$. Since $\cS_i^+$ is non-trivial for all $1\leq i\leq m$ by \ref{en1:a}, Lemma~\ref{lem:1.1.4} implies that each $\cS_i^+$ contains a $k$-simplex which does not contain $\{u,v\}$. We can now use the fact that  $m\geq 2$ and the first part of \ref{en1:b} to deduce that $\cS_i^+\not\subseteq \cS$. The second part of \ref{en1:a} now gives $uv\in E(\cS_i^+)$. 

The assertion that $\bigcup_{i=1}^m E(\cS_i^+)=E(\cS)$ follows immediately from the fact that each $S\in\cS$ belongs to at least one $\cS_i^+$ by \ref{en1:b}.
\\[1mm] 
\ref{en1:d} Let  $\cS_I=\triangle_{i\in I}\cS_i^+$. 
Since $\cS_j^+\setminus \bigcup_{k\neq j} \cS_k^+\neq \emptyset$ for all $1\leq j\leq m$, 
$\emptyset \neq \cS_I\neq \cS$.
Also, since $\cS_I$ is a  non-empty simplicial $k$-cycle and $\cS$ is a simplicial $k$-circuit, 
$\cS_I\not\subseteq \cS$.  Hence we can choose $K\in \cS_I\setminus \cS$.
Then 
$K$ is a non-facial clique of $G(\cS)$ which contains $\{u,v\}$ by \ref{en1:a}. We can now use 
\ref{en1:b} to deduce that $K$ belongs to an even number of the sets in the Fogelsanger decomposition of $\cS$. On the other hand, the fact that
$K\in \cS_I$ implies that $K$ belongs to an odd number of the sets $\cS_i^+$ for $i\in I$. Hence $K$ belongs to an odd number of sets  $\cS_j^+$ for $j\not\in I$. 
\end{proof}

We can now see that Fogelsanger's Rigidity Theorem is a corollary of Lemma~\ref{lem:3}.
\begin{theorem}[Fogelsanger's Rigidity Theorem~\cite{F}]\label{thm:fogelsanger} 
Let $\cS$ be a simplicial $k$-circuit with $k\geq 2$. Then $G(\cS)$ is rigid in $\mathbb{R}^{k+1}$.
\end{theorem}
\begin{proof}
The proof is by induction on $|V(\cS)|$. The theorem holds for the trivial simplicial $k$-circuit so we may suppose that $\cS$ is non-trivial.
Take any edge $uv\in E(\cS)$, and consider a Fogelsanger decomposition $\{\cS_1^+, \dots, \cS_m^+\}$  with respect to $uv$.
Each $\cS_i^+/uv$ is a simplicial $k$-circuit by Lemma~\ref{lem:3}\ref{en1:0} and hence $G(\cS_i^+/uv)$ is rigid by induction. 
We have $S_i^+$ is a non-trivial simplicial $k$-circuit and $uv\in E(S_i^+)$ by Lemma~\ref{lem:3}\ref{en1:a} and  \ref{en1:c}. We can now apply 
the first part of Lemma~\ref{lem:1.1.4} to deduce that $u$ and $v$ have at least $k$ common neighbours in $S_i^+$, and then use Lemma~\ref{lem:split} and the fact that $G(\cS_i^+/uv)=G(\cS_i^+)/uv$ to deduce that  each $G(\cS_i^+)$ is rigid.
Lemma~\ref{lem:3}\ref{en1:d} implies that  we can reorder the parts of the Fogelsanger decomposition so that, for all $2 \leq i \leq m$, 
    \begin{equation}\label{eq:fogint}
    \left|V(\cS_i) \cap \mbox{$\bigcup_{j=1}^{i-1}$} V(\cS_j)\right| \geq k+1.\end{equation}
The rigidity of $G(\cS)$ now follows from the gluing property of rigidity and a simple induction argument together with the fact that $\bigcup_{i=1}^m G(\cS_i^+)=G(\cS)$ by Lemma~\ref{lem:3}\ref{en1:c}.
\end{proof}

\subsection{Subcomplexes with small boundary} 
In this short subsection we derive some  structural properties of a subcomplex of a simplicial circuit which has a small boundary. 
Lemma~\ref{lem:1.1.2} implies that non-trivial simplicial $k$-circuits are $(k+1)$-connected. Since Theorems~\ref{thm:globrigid_3_main} and~\ref{thm:globrigid_d_main} both require $(k+2)$-connectivity, we will need a tool to analyse the case when a simplicial $k$-circuit in a Fogelsanger's decomposition is not $(k+2)$-connected. 
In such a case, our approach will be to decompose the simplicial circuit into smaller circuits. Lemmas~\ref{lem:34sep_a} and \ref{lem:34sep_b} below will enable us to do this.

\begin{lemma}\label{lem:34sep_a}
Let $\cS$ be a simplicial $k$-circuit for some $k \geq 1$, $\cS_1\subsetneq \cS$, and $X=V(\partial \cS_1)$ with $|X|=k+1$. Then either $\cS_1=\{X\}$, or, $X \not\in \cS_1$ and $\cS_1\cup \{X\}$ is a simplicial $k$-circuit.

\end{lemma}
\begin{proof} Suppose that $\cS_1 \neq \{X\}$.
Since  $\partial \cS_1$ is a simplicial $(k-1)$-cycle by Lemma~\ref{lem:1} and $|V(\partial \cS_1)|=|X|=(k-1)+2$,
Lemma~\ref{lem:1.1.1} tells us that $\partial S_1$ is the set of all subsets of $X$ of size $k$. 
This implies that $\partial S_1=\partial \{X\}$ so $\cS_1\triangle \{X\}$ is a 
simplicial $k$-cycle, which is nonempty since $\cS_1 \neq \{X\}$.  
If $X \in \cS_1$ then $\cS_1 \triangle \{X\} \subsetneq \cS$, contradicting the fact that $\cS$ is a simplicial circuit. So $X \not\in \cS_1$ and $\cS_1 \triangle \{X\} = \cS_1 \cup \{X\}$ is a nonempty simplicial $k$-cycle.

To see that $\cS_1\triangle \{X\}$ is a simplicial $k$-circuit, 
suppose, for a contradiction,  that $\cR$ is a nonempty simplicial $k$-cycle with $\cR\subsetneq \cS_1\cup \{X\}$. Let $\cR' = (\cS_1 \cup \{X\}) \sm \cR$. Then $\cR,\cR'$ are both nonempty 
simplicial $k$-cycles, one of which is contained in $\cS_1$, contradicting the fact that $\cS$ is a simplicial $k$-circuit.
\end{proof}

Our next result will be used to handle the situation when $G(\cS)$ is $(k+2)$-connected and $G(\cS)/xy$ is not for some $xy\in E(\cS)$.

\begin{lemma}\label{lem:34sep_b}
Let $\cS$ be a simplicial $k$-circuit with $k\geq2$, $G=G(\cS)$, $\cS_1\subsetneq \cS$, and $X=V(\partial \cS_1)$. Suppose that $|X|=k+2$ and $V(\cS_1)\neq V(\cS)  \neq V(\cS\sm \cS_1)$.
\begin{itemize}
    \item[(a)] If either $X$ is not a clique in $G$ or $k=2$, then $\partial \cS_1=\cL_{k-1}$.
    \item[(b)] Suppose $\partial \cS_1=\cL_{k-1}$, $w$ and $z$ are two non-adjacent vertices of $G(\partial \cS_1)$ and $\cS/xy$ is a simplicial $k$-circuit for some $x,y\in X$.  Then $\{x,y\}\neq \{w,z\}$  and  $\cS_1\triangle \{X-w,X-z\}$ is a simplicial $k$-circuit. 
\end{itemize}
\end{lemma}

\begin{proof} 
Let $\cS_2=\cS\setminus \cS_1$. Since $\partial \cS=\emptyset$, we have 
$\partial \cS_1=\partial \cS_2$.
\\[1mm]
(a) Since  $\partial \cS_1$ is a simplicial $(k-1)$-cycle by Lemma~\ref{lem:1} and $|V(\partial \cS_1)|=|X|=(k-1)+3$,
Lemma~\ref{lem:1.1.3} tells us that $\partial \cS_1=\cL_{k-1}$ unless $k\geq 3$ and $X$ is a clique in $G$. 
\\[1mm]
(b)
Let $\cT$ be the copy of $\cK_{k}$ on the vertex set $X$. Then $\partial (\cT_{\{w,z\}}) = \partial \{X-w,X-z\} = \partial \cS_1$. Therefore  $\cS_1\triangle \cT_{\{w,z\}}= \cS_1 \triangle \{X-w,X-z\}$ is a simplicial $k$-cycle. Since $\cS_1/wz$ = $\cS_1\triangle (\cT_{\{w,z\}})/wz$, it follows that $\cS_1/wz$ is a simplicial $k$-cycle, which is properly contained in $\cS/wz$ by $V(\cS)\sm V(\cS_1) \neq \emptyset$. Therefore $\{x,y\} \neq \{w,z\}$ since $\cS/xy$ is a simplicial $k$-circuit.

It remains to  show that $\cS_1\triangle \{X-w,X-z\}$ is a simplicial $k$-circuit.
Suppose, for a contradiction, that $\cR$ is  a  non-empty simplicial $k$-cycle with $\cR\subsetneq \cS_1\triangle \{X-w,X-z\}$. Then $(\cS_1\triangle \{X-w,X-z\})\setminus \cR$ is also a nonempty simplicial $k$-cycle.
Since $\cS$ is a simplicial $k$-circuit, 
both $\cR$ and $(\cS_1\triangle \{X-w,X-z\})\setminus \cR$ 
must contain one of $X-w$, $X-z$.
Replacing $\cR$ by $(\cS_1\triangle \{X-w,X-z\}) \sm \cR$  and relabelling $w,z$ if necessary, we may suppose that $\cR$ contains $X-z$ and not $X-w$, and that $\{x,y\}\subseteq X-z$.
Then $\cR/xy$ is a simplicial $k$-cycle properly contained in $\cS/xy$. 
This  contradicts the hypothesis that $\cS/xy$ is a simplicial $k$-circuit.
\end{proof}

\subsection{Simplicial 2-circuits with few edges} 
As noted in the introduction,  Fogelsanger's Rigidity Theorem (Theorem \ref{thm:fogelsanger}) implies that every simplicial $k$-circuit $\cT$  has $|E(\cT)|\geq (k+1)|V(\cT)|-\binom{k+2}{2}$ whereas 
 Hendrickson's necessary condition for global rigidity (Theorem~\ref{thm:hendrickson}) asserts that, if 
 $G(\cT)$ is globally rigid in $\R^3$ and non-complete, then $|E(\cT)|\geq (k+1)|V(\cT)|-\binom{k+2}{2}+1$.
 To prove Theorem \ref{thm:globrigid_3_main} using an inductive argument, it will be  useful to know which simplicial $k$-circuits $\cT$  satisfy  $|E(\cT)|=(k+1)|V(\cT)|-\binom{k+2}{2}$ when $k=2$. Our next result does this and hence verifies Theorem \ref{thm:LBT} when $k=2$, i.e., $d=3$. 
  (A similar result for $k\geq 3$ is not needed to prove Theorem \ref{thm:globrigid_d_main}  since  global rigidity is completely characterised by $(k+2)$-connectivity in this case. Once we have proved  Theorem \ref{thm:globrigid_d_main}, we can use it to show that Theorem \ref{thm:LBT} holds when $k\geq 3$, i.e., $d\geq 4$.)
 
\begin{theorem}\label{thm:plane}
Suppose $\cT$ is a simplicial 2-circuit with $|E(\cT)|= 3|V(\cT)|-6$. Then $G(\cT)$ is a plane triangulation and 
 $\cT$ is the set of faces in a plane embedding of $G(\cT)$.
 \end{theorem}
 \begin{proof} We proceed by induction on $|V(\cT)|$.
The theorem holds when $|V(\cT)|=3$, since, in that case, $\cT$ is a trivial simplicial $2$-circuit and hence $G(\cT)\cong K_3$. It also holds
when $|V(\cT)|=4$, since  $\cT\cong \cK_2$ in that case, by Lemma~\ref{lem:1.1.1}. Hence we may suppose that
$|V(\cT)|\geq 5$. 

Suppose that $G(\cT)$ has a 3-vertex separator $N$.
Then Lemma~\ref{lem:1.1.2}  implies that $N$ is a non-facial $3$-clique in $G(\cT)$ and there exists a partition $(\cT_1,\cT_2)$ of $\cT$ such that $V(\cT_1)\cap V(\cT_2)=N$ and $\cT_1\cup\{N\},\cT_2\cup\{N\}$ are both simplicial 2-circuits. Since $G(\cT_i\cup \{N\})$
is rigid in $\R^3$ by  Theorem \ref{thm:fogelsanger}, we have
$|E(\cT_i\cup \{N\})|=|E(\cT_i)|\geq 3|V(\cT_i)|-6$ for $i=1,2$. Thus
$$|E(\cT)|=|E(\cT_1)|+|E(\cT_2)|-3\geq 3|V(\cT_1)|-6+3|V(\cT_2)|-6-3=3|V(\cT)|-6.$$
Equality must hold throughout and hence $|E(\cT_i)|= 3|V(\cT_i)|-6$ for $i=1,2$. By induction, $\cT_i$ is the set of faces in a plane embedding of $G(\cT_i)$. We can now construct the required embedding of $G(\cT)$ by choosing an embedding of $G(\cT_1)$ with $N$ on its outer face and substituting this into the face of an embedding of $G(\cT_2)$ which is labelled by $N$. Hence we may assume that $G(\cT)$ is 4-connected.

Let $v$ be a vertex of minimum degree in $G(\cT)$. The hypothesis that 
$|E(\cT)|= 3|V(\cT)|-6$ and the fact that $G(\cT)$ is 4-connected imply that $4\leq d(v)\leq 5$. In addition, Theorem  
\ref{thm:fogelsanger} implies that $G(\cT)$ is minimally rigid so 
$G(\cT)$ satisfies (\ref{eq:3_6_sparse}) for $d=3$, i.e., $G(\cT)$ is $(3,6)$-sparse.

Recall that $\cT_v$ denotes the set of 2-simplices in $\cT$ that contain $v$.
 Note that since $\cT$ is a nontrivial simplicial $2$-circuit, and using Lemma~\ref{lem:stars}, we have 
\begin{equation}
    \label{eqn_bdTv1}
    \partial \cT_v = \{ U -v: U \in \cT_v\} = lk_\cT(v).  
\end{equation}
Also, by Lemma~\ref{lem:1}, 
\begin{equation}
    \label{eqn_bdTv2}
    \partial G(\cT_v) \text{ is an even subgraph of $G(\cT)$.}  
\end{equation}

The proof now splits  into two cases depending on whether $d(v)=4$ or $5$.

\begin{figure}
    \centering 
\begin{tikzpicture}[line cap=round,line join=round,>=triangle 45,x=1cm,y=1cm]
\draw [line width=1pt] (-4.44,0.3)-- (-4.24,2);
\draw [line width=1pt] (-4.24,2)-- (-2.92,0.46);
\draw [line width=1pt] (-2.92,0.46)-- (-4.44,0.3);
\draw [line width=1pt] (-2.92,0.46)-- (-4.78,-1.1);
\draw [line width=1pt] (-4.78,-1.1)-- (-6.26,0.24);
\draw [line width=1pt] (-6.26,0.24)-- (-4.24,2);
\draw [line width=1pt] (-6.26,0.24)-- (-4.44,0.3);
\draw [line width=1pt] (-4.44,0.3)-- (-4.78,-1.1);
\draw [line width=1pt] (0.68,0.4)-- (2.7,2.16);
\draw [line width=1pt] (2.7,2.16)-- (4.04,0.6);
\draw [line width=1pt] (4.04,0.6)-- (2,-0.96);
\draw [line width=1pt] (2,-0.96)-- (0.68,0.4);
\draw [line width=1pt] (2,-0.96)-- (2.7,2.16);
\begin{scriptsize}
\draw [fill=black](-4.44,0.3) circle (2pt);
\draw (-4.58,0.53) node {$v$};
\draw [fill=black] (-4.24,2) circle (2pt);
\draw (-4.28,2.23) node {$w$};
\draw  [fill=black](-2.92,0.46) circle (2pt);
\draw (-2.76,0.69) node {$x$};
\draw [fill=black](-4.78,-1.1) circle (2pt);
\draw (-4.85,-0.77) node {$y$};
\draw [fill=black](-6.26,0.24) circle (2pt);
\draw (-6.3,0.57) node {$z$};
\draw [fill=black](0.68,0.4) circle (2pt);
\draw (0.74,0.63) node {$z$};
\draw [fill=black](2.7,2.16) circle (2pt);
\draw (2.86,2.39) node {$w$};
\draw [fill=black](4.04,0.6) circle (2pt);
\draw (4.2,0.8) node {$x$};
\draw [fill=black](2,-0.96) circle (2pt);
\draw (1.9,-0.63) node {$y$};

\draw (-3.9,0.73) node {$T_1$};
\draw (-4.0,-0.13) node {$T_2$};
\draw (-4.96,-0.33) node {$T_3$};
\draw (-4.96,0.73) node {$T_4$};
\draw (3.16,0.5) node {$T_5$};
\draw (1.56,0.53) node {$T_6$};
\end{scriptsize}
\end{tikzpicture}
    
    \caption{Case 1 in the proof of Theorem \ref{thm:plane}. The figure on the left represents $\cT_v$ in this case and the figure on the right represents $\{T_5,T_6\}$.}
    \label{fig:case1}
\end{figure}
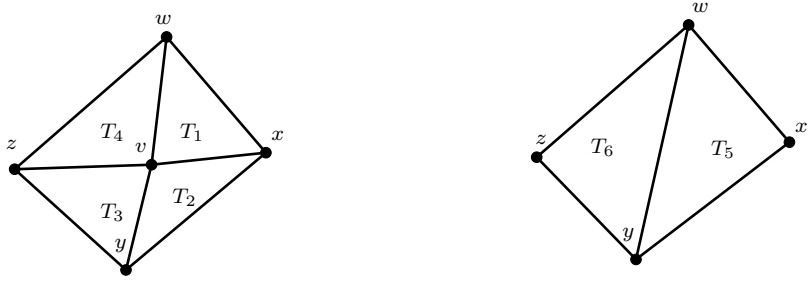

\paragraph{Case 1: $d_{G(\cT)}(v)=4$.} Let $N_{G(\cT)}(v)=\{w,x,y,z\}$.

Using (\ref{eqn_bdTv1}) and (\ref{eqn_bdTv2}) we see that 
$G(\partial \cT_v)$ is an even graph whose vertex set is $\{w,x,y,z\}$ and so must be a cycle graph of length 4. Using (\ref{eqn_bdTv1}) again and relabelling vertices if necessary it follows that $\cT_v = \{T_i: i = 1,2,3,4\}$, 
where $T_1=\{v,w,x\}$, $T_2=\{v,x,y\}$, $T_3=\{v,y,z\}$ and $T_4=\{v,z,w\}$.
See Figure \ref{fig:case1}.

Since $G(\cT)$ is $(3,6)$-sparse, it does not contain a subgraph isomorphic to $K_5$.
Hence, without loss of generality, we may assume $wy\notin E(\cT)$.
Let $T_5=\{w,x,y\}$,  $T_6=\{w,y,z\}$ and $\cS=\cT_v\cup \{T_5,T_6\}$. Then $\cS$ is a simplicial 2-cycle. 
Since $wy\notin E(\cT)$, $T_5$ and $T_6$ are missing in $\cT$.

Let $\cU=\cT\triangle \cS$. Then $\cU$ is a simplicial 2-cycle since both $\cT$ and $\cS$ are simplicial $2$-cycles. Since $wy \not\in \cT$, it follows that $\{T_5,T_6\} \cap \cT = \emptyset$, $\cU = (\cT \sm \cT_v) \cup \{T_5,T_6\}$ and $V(\cU) = V(\cT) \sm \{v\}$.
Suppose  $\cU$ is a simplicial 2-circuit. Then $|E(\cU)|\geq 3|V(\cU)|-6$ by Theorem \ref{thm:fogelsanger}, 
and, since $E(\cU)\setminus \{wy\}\subseteq E(\cT\setminus \cT_v)$, we have
 $$|E(\cT)|=|E(\cT\setminus \cT_v)|+4\geq |E(\cU)|+3\geq 3|V(\cU)|-3=3|V(\cT)|-6.$$
Equality must hold throughout and hence $|E(\cU)|=3|V(\cU)|-6$. By induction $\cU$ is the set of faces of a plane embedding of $G(\cU)$. We can extend this embedding to the required plane embedding of $G(\cT)$ by deleting $wy$ and then putting $v$ and its incident edges in the face of $G(\cU)-wy$ which corresponds to $T_5\cup T_6$.  
 
It remains to consider the subcase when  $\cU$ is not a simplicial 2-circuit. Let $\cU_1 \subsetneq \cU$ be a simplicial 2-circuit. Then
$\cU\sm \cU_1$ is a simplicial 2-cycle
and hence we may choose another simplicial 2-circuit $\cU_2$ with 
$\cU_2\subseteq \cU\sm \cU_1$. Since $\cT$ is a simplicial 2-circuit, $\cU_i\not\subseteq \cT$ and hence
$\cU_i\cap\{T_5,T_6\}\neq \emptyset$ for $i=1,2$.
Relabelling if necessary and using the fact that $\cU_1,\cU_2$ are disjoint, we may assume that  
$\cU_1\cap\{T_5,T_6\}=\{T_5\}$ and 
$\cU_2\cap\{T_5,T_6\}=\{T_6\}$.
The facts that $\cU_1$ is a simplicial 2-cycle, $T_5\in \cU_1$ and $\cU_1-T_5\subseteq \cT$, imply that $wy$ is contained in some 2-simplex of $\cT$. This contradicts $wy\notin E(\cT)$.

\begin{figure}
    \centering

\begin{tikzpicture}[line cap=round,line join=round,>=triangle 45,x=1cm,y=1cm]
\draw [line width=1pt] (-4.7,4.68)-- (-6.6,3.24);
\draw [line width=1pt] (-6.6,3.24)-- (-5.68,1.74);
\draw [line width=1pt] (-5.68,1.74)-- (-3.84,1.78);
\draw [line width=1pt] (-3.84,1.78)-- (-3.08,3.56);
\draw [line width=1pt] (-3.08,3.56)-- (-4.7,4.68);
\draw [line width=1pt] (-4.8,3.1)-- (-4.7,4.68);
\draw [line width=1pt] (-4.8,3.1)-- (-3.08,3.56);
\draw [line width=1pt] (-4.8,3.1)-- (-3.84,1.78);
\draw [line width=1pt] (-4.8,3.1)-- (-5.68,1.74);
\draw [line width=1pt] (-4.8,3.1)-- (-6.6,3.24);
\draw [line width=1pt] (2.28,4.82)-- (0.6,3.58);
\draw [line width=1pt] (0.6,3.58)-- (1.22,1.66);
\draw [line width=1pt] (1.22,1.66)-- (3.64,1.72);
\draw [line width=1pt] (3.64,1.72)-- (4.28,3.66);
\draw [line width=1pt] (4.28,3.66)-- (2.28,4.82);
\draw [line width=1pt] (2.28,4.82)-- (1.22,1.66);
\draw [line width=1pt] (2.28,4.82)-- (3.64,1.72);
\begin{scriptsize}
\draw [fill=black] (-4.7,4.68) circle (2pt);
\draw[color=black] (-4.68,5.05) node {$x_{1}$};
\draw [fill=black] (-6.6,3.24) circle (2pt);
\draw[color=black] (-6.86,3.41) node {$x_{2}$};
\draw [fill=black] (-5.68,1.74) circle (2pt);
\draw[color=black] (-6,1.73) node {$x_{3}$};
\draw [fill=black] (-3.84,1.78) circle (2pt);
\draw[color=black] (-3.5,1.89) node {$x_{4}$};
\draw [fill=black] (-3.08,3.56) circle (2pt);
\draw[color=black] (-2.82,3.97) node {$x_{5}$};
\draw [fill=black] (-4.8,3.1) circle (2pt);
\draw[color=black] (-4.64,3.33) node {$v$};
\draw [fill=black] (2.28,4.82) circle (2pt);
\draw[color=black] (2.3,5.35) node {$x_1$};
\draw [fill=black] (0.6,3.58) circle (2pt);
\draw[color=black] (0.32,3.79) node {$x_2$};
\draw [fill=black] (1.22,1.66) circle (2pt);
\draw[color=black] (1.0,1.57) node {$x_3$};
\draw [fill=black] (3.64,1.72) circle (2pt);
\draw[color=black] (3.9,1.85) node {$x_4$};
\draw [fill=black] (4.28,3.66) circle (2pt);
\draw[color=black] (4.58,3.95) node {$x_5$};
\draw[color=black] (1.2,3.29) node {$T_1$};
\draw[color=black] (2.5,2.9) node {$T_2$};
\draw[color=black] (3.6,3.29) node {$T_3$};
\end{scriptsize}
\end{tikzpicture}

    \caption{Case 2 in the proof of Theorem \ref{thm:plane}. The figure on the left represents $\cT_v$ in the case when $G(\partial \cT_v) \cong C_5$, and the figure on the right represents $\{T_1,T_2,T_3\}$.}
    \label{fig:case2}
\end{figure}
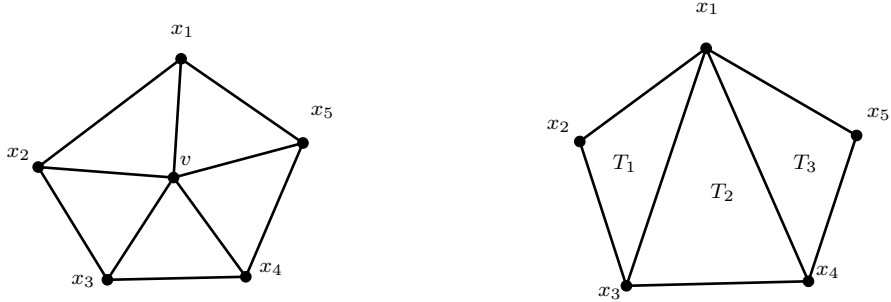
\paragraph{Case 2: $d_{G(\cT)}(v)=5$.}  
Since  $G(\partial \cT_v)$ is an even subgraph of $G(\cT)$ with vertex set contained in $N_{G(\cT)}(v)$,
$G(\partial \cT_v$) is isomorphic to one of $K_5$, $K_5 \sm E(C_4)$, $K_5 \sm E( C_3)$ or $C_5$.
However, since $G(\cT)$ is $(3,6)$-tight it cannot contain a subgraph isomorphic to $K_5$.
 
Suppose $G(\partial \cT_v) \cong K_5\sm E(C_4)$.
Then $\partial \cT_v$ can be decomposed into two disjoint edge-sets $L_1,L_2$ such that 
$G(L_1) \cong G(L_2) \cong C_3$ and $|V(L_1)\cap V(L_2)| = 1$.
Let $\cT_1 = \{\{v,x,y\} \in \cT: xy \in L_1\}$.
We introduce a new vertex $v'$ and put $\cT_1'=\{ \{v',x,y\}: xy \in L_1\}$.
Then $\cT' = (\cT\sm \cT_1) \cup \cT_1'$ is a simplicial 2-cycle since both $G(L_1)$ and $G(L_2)$ are even graphs. (The graph $G(\cT')$ is obtained from $G(\cT)$ by `splitting' the vertex $v$ into two non-adjacent vertices of degree three with one common neighbour.)
Moreover, since any simplicial $2$-cycle contained in 
$\cT'$ induces a simplicial $2$-cycle contained in $\cT$, $\cT'$ is a simplicial 2-circuit. Now, {since $|V(L_1)\cap V(L_2)| = 1$}, we have $|E(\cT')| = |E(\cT)|+1 = 3|V(\cT)|-5 = 3|V(\cT')|-8$,
contradicting the fact that $G(\cT')$ is rigid by Theorem \ref{thm:fogelsanger}. 
Thus $G(\partial \cT_v) \not\cong K_5\sm E(C_4)$.

We  obtain a similar contradiction 
if we suppose that 
$G(\partial \cT_v) \cong K_5\sm E(C_3)$ 
(using a decomposition of $K_5 \sm E(C_3)$ into  a $3$-cycle and a $4$-cycle).
Hence we may assume that  
$G(\partial \cT_v)  \cong C_5$.

Label the neighbours of $v$ as $x_1,x_2,\ldots, x_5$ according to the cyclic ordering given by the 5-cycle
$G(\partial \cT_v)$.
Using (\ref{eqn_bdTv2}) we see that 
\begin{equation}
\notag
\cT_v
= \{\{v,x_1,x_2\},\{v,x_2,x_3\},\{v,v_3,x_4\},\{v,x_4,x_5\},\{v,x_5,x_1\}\}.\end{equation}
See Figure \ref{fig:case2}
Since $G(\cT)$ is $(3,6)$-sparse, $N_{G(\cT)}(v)\cup \{v\}$ induces at most 12 edges, so $N(v)$ induces at most 7 edges in $G(\cT)$.
Hence, the 
subgraph of $G(\cT)$ induced by $N_{G(\cT)}(v)$
has at least one degree-two vertex.
Without loss of generality we assume $x_1$ is such a vertex.
Then $x_1x_3, x_1x_4\notin E(\cT)$.

Let $T_1=\{x_1,x_2,x_3\}$, $T_2=\{x_1,x_3,x_4\}$,  $T_3=\{x_1,x_4,x_5\}$ and $\cS = \cT_v \cup \{T_1,T_2,T_3\}$.
Then  $\cU = \cT \triangle \cS$  is a simplicial cycle
since both $\cT$ and $\cS$ are. 
Since $x_1x_3, x_1x_4 \not\in E(\cT)$, it follows that $\{T_1,T_2,T_3\} \cap \cT = \emptyset$, $\cU = (\cT \sm\cT_v) \cup \{T_1,T_2,T_3\}$ and that $V(\cU)= V(\cT)\sm \{v\}$.

Suppose that $\cU$ is a simplicial 2-circuit. 
Then $|E(\cU)| \geq 3|V(\cU)| -6$  by Theorem \ref{thm:fogelsanger}
and, since $E(\cU)\setminus \{x_1x_3, x_1x_4\}\subseteq E(\cT\setminus \cT_v)$, we have
\[
    |E(\cT)|\geq |E(\cT\setminus \cT_v)|+5 \geq |E(\cU)| + 3 \geq 3|V(\cU)| -6+3  = 3|V(\cT)|-6.
\]
Equality must hold throughout so $|E(\cU)|  = 3|V(\cU)| -6$.
By induction, $\cU$ is the set of faces in a plane embedding of $G(\cU)$.
We can now obtain the required plane embedding of $\cT$ by deleting the edges 
$x_1x_3, x_1x_4$ from the embedding of $G(\cU)$ 
and putting $v$ and its incident edges in the 
resulting pentagonal face.

It remains to consider the subcase when  $\cU$ is not a simplicial 2-circuit. 
Let $\cU_1 \subsetneq \cU$ be a 
simplicial 2-circuit. Then
$\cU\sm \cU_1$ is a simplicial 2-cycle and hence we may choose another simplicial 2-circuit
$\cU_2\subseteq \cU\sm \cU_1$.
The fact that  $\cT$ is a simplicial 2-circuit implies that  $\cU_i\not\subseteq \cT$ and hence 
$\cU_i\cap\{T_1,T_2,T_3\}\neq \emptyset$ for $i=1,2$.
Relabelling if necessary, and using the fact that $\cU_1,\cU_2$ are disjoint, we may suppose that $|\cU_1\cap \{T_1,T_2,T_3\}|=1$. 
The facts that $\cU_1$ is a simplicial 2-circuit and $\cU_1\setminus \{T_1, T_2, T_3\} \subseteq \cT$ imply that $x_1x_3$ or $x_1x_4$ is contained in some 2-simplex of $\cT$, contradicting $x_1x_3, x_1x_4\notin E(\cT)$.
\end{proof}

\section{The \texorpdfstring{$d$}{d}-cleavage property}\label{sec:cleavage}

Lemma~\ref{lem:1.1.2} implies that, if $G$ is the graph of a non-trivial simplicial $k$-circuit, then $G$ is $(k+1)$-connected and every $(k+1)$-vertex-separator of $G$ is a clique. It will be useful for our purposes to study graphs with these two properties.

Let $t\geq 1$ be an integer.  We say that a graph $G$ has the {\em $t$-cleavage property} if $G$ is $t$-connected and every $t$-vertex separator of $G$ is a clique in $G$.
A  subgraph $D$ of a graph $G$ with the $t$-cleavage property is  said to be a {\em $(t+1)$-block} of 
$G$ if   either $D\cong K_{t+1}$ or $D$ is $(t+1)$-connected and, subject to this condition, $D$ is maximal with respect to inclusion.   Let $X$ be a $t$-vertex-separator of $G$, let $H_1,H_2,\ldots,H_m$ be the components of $G-X$ and put $G_i=G[V(H_i)\cup X]$ for  all $1\leq i\leq m$. Then we say that $G_1,G_2,\ldots,G_m$ are the {\em $t$-cleavage graphs of $G$ at $X$}.

\begin{lemma}\label{lem:cleave2}
Suppose $G$ is a graph with the $t$-cleavage property.
Let $X$ be a $t$-vertex-separator of $G$ and  $G_1,G_2,\ldots,G_m$ be the {\em $t$-cleavage graphs} of $G$ at $X$. Then $G_i$ has the $t$-cleavage property for all $1\leq i\leq m$. Moreover each $(t+1)$-block of $G$ is a $(t+1)$-block of $G_i$ for exactly one $1\leq i\leq m$, and each $t$-vertex separator of $G$ other than $X$ is a $t$-vertex separator of $G_i$ for exactly one $1\leq i\leq m$.
\end{lemma}
\begin{proof} The hypothesis  that $G$ has the $t$-cleavage property implies that $G[X]\cong K_t$ and hence each $G_i$ is $t$-connected. In addition,  each $t$-vertex-separator of $G_i$ is a $t$-vertex-separator of $G$ so $G_i$ has the $t$-cleavage property. If $D$ is a $(t+1)$-block of $G$ then $D-X$ is connected so $D-X$ is a subgraph of $G_i-X$ for exactly one $G_i$. Hence $D$ is a $(t+1)$-block of exactly one $G_i$. It remains to show that each $t$-vertex separator $X'$ of $G$ other than $X$ is a $t$-vertex separator of some $G_i$. This follows because the $t$-connectivity of each $G_j$ gives $G-X'$ is connected whenever $X'$ is not contained in some $G_i$.
\end{proof} 

We can apply Lemma~\ref{lem:cleave2} to each $t$-cleavage graph recursively
to obtain the following decomposition result which extends the well known block/cut vertex decomposition of a connected graph to graphs with the $t$-cleavage property.

\begin{theorem} \label{thm:cleavagedec}
Suppose $G$ is a graph with the $t$-cleavage property.  
Let $T$ be the bipartite graph with vertex classes given by the sets of $(t+1)$-blocks and $t$-vertex separators of $G$, in which a $(t+1)$-block $D$ and a $t$-vertex separator $X$ are adjacent if $X\subseteq V(D)$. Then $T$ is a tree.  \hfill \qed
\end{theorem} 

We will refer to the tree $T$ described in Theorem \ref{thm:cleavagedec} as the {\em $(t+1)$-block tree} of $G$. 
Note that Lemma~\ref{lem:1.1.2} implies that if $\cS$ is a simplicial $k$-circuit for some $k\geq 1$, then $G(\cS)$ has the $(k+1)$-cleavage property and that the vertices of the $(k+2)$-block tree of $G(\cS)$ that correspond to $(k+1)$-vertex separators have degree 2.

It is easy to see that the cleavage property is preserved by edge addition i.e., if $G$ has the $t$-cleavage property and $e$ is an edge between two non-adjacent vertices of $G$ then $G+e$ has the $t$-cleavage property. We next give a similar result for the union of two graphs with the $t$-cleavage property.  

\begin{lemma}\label{lem:cleave1}
Let $H_1=(V_1,E_1)$ and $H_2=(V_2,E_2)$ be graphs with the $t$-cleavage property. Suppose that $|V_1\cap V_2|\geq t$ and that, if equality holds, then $V_1 \cap V_2$ is a clique of $G=H_1 \cup H_2$. Then $G$ has the $t$-cleavage property and
each $(t+1)$-block $D$ of $G$  which is not a $(t+1)$-block of $H_1$ or $H_2$, is a union of $(t+1)$-blocks contained in two subtrees  of the $(t+1)$-block trees of  $H_1$ and $H_2$, respectively. 
\end{lemma}
\begin{proof} The fact that $G$ is $t$-connected follows from the hypotheses that $H_1,H_2$ are $t$-connected and $|V_1\cap V_2|\geq t$. The lemma is trivially true if $G$ is $(t+1)$-connected so we may assume that $G$ has a $t$-vertex separator $X$. If $X$ is a separator of either $H_1$ or $H_2$ then $G[X]\cong K_t$ since $H_1,H_2$ have the $t$-cleavage property. On the other hand, if $X$ is a separator of neither $H_1$ nor $H_2$ then $X = V_1\cap V_2$ and we again have $G[X]\cong K_t$ since $|V_1\cap V_2|=t$. Hence $G$ has the $t$-cleavage property.   

Let $G_1,G_2,\ldots,G_m$ be the $t$-cleavage graphs of $G$ at a $t$-vertex separator $X$. Then each $(t+1)$-block of $G$ is a $(t+1)$-block of $G_j$ for some $1\leq j\leq m$. Put $H_{i,j}=H_i\cap G_j$ for all $1\leq i\leq 2$ and $1\leq j\leq m$. 
If $V(G_1)\cap V(H_2)\subseteq X$, then every $(t+1)$-block of $G$ which is contained in  $G_1$ is a $(t+1)$-block of $H_1$ and the lemma follows by applying induction to $G-(G_1-X)=(H_1-(G_1-X))\cup H_2$.
Hence we may assume that $V(H_{i,j})\not \subseteq X$ for all $1\leq i\leq 2$ and all $1\leq j\leq m$.  The lemma now follows by applying induction to each $G_j=H_{1,j}\cup H_{2,j}$.  
\end{proof}

We have seen that the graphs of simplicial $k$-circuits have the $(k+1)$-cleavage property. We close this section with a result which gives useful structural properties of their $(k+2)$-block trees.


\begin{lemma}\label{lem:mincleavage}
Let $G$ be the graph of a  simplicial $k$-circuit and $H$ be a union of $(k+2)$-blocks in a subtree of the $(k+2)$-block tree of $G$. Then $H$ is the graph of a simplicial $k$-circuit. Furthermore, if $G/uv$ is the graph of a simplicial $k$-circuit for some 
$u$ and $v$ belonging to the same $(k+2)$-block of $H$, then $H/uv$ is the graph of a simplicial $k$-circuit.
\end{lemma}
\begin{proof} Since $G$ is the graph of a simplicial $k$-circuit, we have $G=G(\cS)$ for some simplicial $k$-circuit $\cS$.

Suppose first that $H$ is a union of all but one $(k+2)$-block of $G$. Suppose that the missing block is $D$. By assumption $D$ has degree one in the $(k+2)$-block tree of $G$. Let $X = V(D) \cap V(H)$.

Let 
$\cT$ be the set of $k$-simplices $S$ such that $S\subseteq V(H) \cup\{X\}$.
Then $V(\partial \cT) \subseteq X$ since, if $F$ is a $(k-1)$-face of $\cT$ and $F \not\subseteq X$ then $\cT_F = \cS_F$ and so $|\cT_F|$ is even. Since $\partial \cT \neq \emptyset$ and $|X| = k+1$ it follows that $V(\partial \cT) = X$. 
Then $H=G(\cT)$  and $\cT$ is a simplicial $k$-circuit by Lemma~\ref{lem:34sep_a}. Now an easy induction argument gives the first part of the lemma for arbitrary subtrees of the $(k+2)$-block tree of $G$.

Suppose $G/uv$ is the graph of a simplicial $k$-circuit for some 
$u$ and $v$ belonging to the same $(k+2)$-block of $H$. 
Then $G/uv$ is $(k+1)$-connected by Lemma~\ref{lem:1.1.2} and hence  $\{u,v\}$ is contained in no $(k+1)$-vertex separator of $G$. This implies that 
$H/uv$ is the union of $(k+2)$-blocks in a subtree of the $(k+2)$-block tree of $G/uv$
and we can now use the first part of the lemma to deduce that $H/uv$ is the graph of a simplicial $k$-circuit.
\end{proof}

We close this section by showing how Theorems \ref{thm:LBT} and \ref{thm:redundant} follow from Theorems \ref{thm:globrigid_3_main} and \ref{thm:globrigid_d_main}. 
Given a simplicial $k$-multicomplex $\cS$ and a $k$-simplex $U \in \cS$, let $\cV$ be the copy of $\cK_k$ obtained by considering all $(k+1)$-element subsets of $U \cup \{v\}$ where $v$ is a new vertex (i.e., $ v \not\in V(\cS)$). The {\em barycentric subdivision} of $\cS$ at $U$ is the simplicial $k$-multicomplex $\cS \triangle \cV$.
A simplicial $k$-complex is a {\em stacked $k$-sphere} if it can be obtained from a copy of $\cK_k$ by a sequence of barycentric subdivisions at $k$-simplices. We can use 
Lemma~\ref{lem:mincleavage} and an easy induction to show that a
simplicial $k$-multicomplex is a {stacked $k$-sphere} if and only if it is a non-trivial simplicial $k$-circuit and every $(k+2)$-block of its underlying graph is a copy of $K_{k+1}$.

\paragraph{Proof of Theorem \ref{thm:LBT}.}
We already noted in the introduction that $G$ is rigid in $\R^d$ by Fogelsanger's theorem and hence satisfies (\ref{eq:count}). If $\cS$ is a stacked $(d-1)$-sphere or $d=3$ and $G$ is planar then it is easy to check that equality holds in (\ref{eq:count}). To prove the reverse implication we suppose that equality holds in (\ref{eq:count}). Then equality also holds for each $(d+1)$-block of $G$. Theorems \ref{thm:hendrickson}, \ref{thm:globrigid_3_main} and \ref{thm:globrigid_d_main} now imply that
every $(d+1)$-block of $G$ is a copy of $K_{d}$ or $d=3$ and every
$4$-block of $G$  is planar. By Lemma~\ref{lem:1.1.2} this implies that $\cS$ is a stacked $(d-1)$-sphere or $d=3$ and $G$ is planar. 
\qed  

\paragraph{Proof of Theorem \ref{thm:redundant}.}
Suppose $e$ is contained in a (non-planar when $d=3$) $(d+1)$-connected subgraph of $G$.
Then $e$ belongs to a (non-planar when $d=3$) $(d+1)$-block of $G$, which is the graph of a simplicial $(d-1)$-circuit by Lemma~\ref{lem:mincleavage}.
Hence, by Theorems~\ref{thm:hendrickson}, \ref{thm:globrigid_3_main} and \ref{thm:globrigid_d_main}, $G-e$ is rigid in $\mathbb{R}^d$.

Conversely, suppose $e$ is not contained in  a (non-planar when $d=3)$ $(d+1)$-connected subgraph.
We first consider the case when $d=3$.
Let $D_1,\dots, D_k$ be the $4$-blocks of $G$ that contain $e$.
Then, each $D_i$ is planar and, using Lemma~\ref{lem:1.1.2}, $G':=\bigcup_{i=1}^k D_i$ is also planar.
Hence, $G'-e$ is not rigid. Since $G-e$ is obtained from $G'-e$ by gluing subgraphs of $G$ along 3-clique separators and those 3-clique separators do not induce  $e$,
$G-e$ cannot be rigid.
For $d\geq 4$, we can apply the same argument by using copies of $K_{d+1}$ instead of planar graphs.
\qed
 
\section{The strong cleavage property}\label{sec:strong_cleavage}
Our proofs of Theorems \ref{thm:globrigid_3_main} and \ref{thm:globrigid_d_main} use an induction based on taking a Fogelsanger decomposition.
As noted in the introduction, the fact that the graphs of the simplicial $k$-circuits in a
Fogelsanger decomposition need not be $(k+2)$-connected means that we cannot apply the induction hypothesis directly to these graphs. Instead, we apply it to their $(k+2)$-blocks. This leads us to our definition of the 
strong $d$-cleavage property. The fact that the characterisations of global rigidity given by Theorems \ref{thm:globrigid_3_main} and \ref{thm:globrigid_d_main} are different means we have to work with slightly different definitions for $k=2$ and $k\geq 3$.

Let $G$ be a graph with the $d$-cleavage property.
When $d=3$, we say that  $G$ has the {\em strong $3$-cleavage property} 
if every $4$-block of $G$ is a plane triangulation or is globally rigid in $\R^3$. 
When $d\geq 4$,  we say that  $G$ has the {\em strong $d$-cleavage property} if every $(d+1)$-block of $G$ is globally rigid in $\R^d$.

Our goal in this section is to prove a `gluing result' for graphs with the  strong $d$-cleavage property, Corollary \ref{cor:cleavage_globrigid} below, which will enable us to combine the graphs of the simplicial circuits in a Fogelsanger decomposition in such a way that we preserve the strong $d$-cleavage property.

\subsection{Edge addition lemma}
To obtain our gluing result for graphs with the strong $d$-cleavage property,
we first  show that the strong $d$-cleavage property is closed under edge addition. We will need the following concepts. Two vertices $u,v$ in a graph $G$ are {\em globally linked in $\R^d$} 
if, for every generic $(G,p)$, $\|q(u) - q(v)\| = \|p(u) - p(v)\|$ for every $(G,q)$ that is equivalent to $(G,p)$.
It is easy to see that, if $u,v$  are globally linked, then $G$ is globally rigid in $\R^d$
if and only if $G+uv$ is globally rigid in $\R^d$. 
We define the {\em globally linked closure of $G$ in $\R^d$}, $\cl_d^*(G)$, to be the graph obtained from $G$ by adding an edge between all pairs of globally linked vertices of $G$. Thus $G$ is globally rigid in $\R^d$ if and only if $\cl^*_d(G)$ is a complete graph. 
Also, using the fact that a plane triangulation with at least 4 vertices is not globally rigid in $\mathbb R^3$ and Theorem \ref{thm:globplane} below, it follows that if $G$ is a 4-connected plane triangulation then $\cl_3^*(G) = G$.  

Now if $G$ is a graph with the  $d$-cleavage property and $D_1,D_2,\ldots,D_t$ are the $(d+1)$-blocks of $G$, then 
$\cl_{d}^*(D_1),\cl_{d}^*(D_2),\ldots,\cl_{d}^*(D_t)$ are the $(d+1)$-blocks of $\cl_d^*(G)$.
(Note that if $u$ and $v$ are vertices in  distinct components in $G-X$ for some $d$-vertex separator $X$ then $u$ and $v$ are not globally linked in $G$.)
Hence, if $G$ has the strong $d$-cleavage property, then for all $1\leq i\leq t$, either $\cl_d^*(D_i)$ is a complete graph or $d=3$ and $\cl_d^*(D_i)$ is a plane triangulation.  
We will use this observation to  show that the strong $d$-cleavage property is closed under edge addition. 
The case when $d=3$, $G$ is a plane triangulation and $G+e$ is 4-connected follows from the following global rigidity theorem for braced triangulations due to  Jord\'an and Tanigawa \cite{JT}. We will use this result as a base case in our inductive proof when $d=3$.

\begin{theorem}[\cite{JT}] \label{thm:globplane}  
Let $G=(V,E)$ be a plane triangulation and  $B$ be a set of  edges of $K(V)$ such that  $B\not\subseteq  E$ and  $G+B$ is 4-connected. Then $G+B$ is globally rigid in $\R^3$.
\end{theorem}

We also need the following technical lemma when $d=3$.
\begin{lemma}\label{lem:cleavagepath} 
Let $G$ be a graph with the strong $3$-cleavage property.
Suppose that the 4-block tree of $G$ is a path.
Then $\cl_3^*(G)$ contains a spanning plane triangulation $\overline{G}$ whose 4-block tree is also a path.  
Moreover, if $D$ is a 4-block of $\cl_3^*(G)$ and $|V(D)| \neq 5$, then $\overline{D} = D \cap \overline G$ is a 
4-block of $\overline G$.
\end{lemma}
\begin{proof}
    Since the 4-block tree of $\cl_3^*(G)$ is a path, each 4-block of $\cl_3^*(G)$ contains at most two 3-vertex separators of $\cl_3^*(G)$. 
    For each $4$-block $D$ of $G$ we will define 
    a spanning plane triangulation $\overline{D} \subseteq D$ with the property that each of the $3$-vertex separators of $\cl_3^*(G)$ contained in $D$ is a face of $\overline{D}$,  and then put
    $$\overline G = \bigcup_{D\text{ a 4-block of }G}\overline{D}.$$
    If $D$ is a planar 4-block of $\cl_3^*(G)$, let $\overline D = D$.
    If $D$ is a non-planar 4-block of $\cl_d^*(G)$ with $|V(D)| \geq 6$ then $D$ is a complete graph and we can choose a spanning subgraph 
    $\overline D$ of $D$ that is a 4-connected plane triangulation such that each of the 3-vertex-separators of 
    $\cl_3^*(G)$ contained in $V(D)$ induces a face of $\overline D$. 

    It remains to consider the case where $D$ is non-planar and $|V(D)| = 5$. Then $D \cong K_5$ and we can choose an edge $xy$ in $E(D)$ as follows. If $D$ contains a unique 3-vertex-separator, $X$, of $\cl_3^*(G)$ then choose $ x\in X$ and $y \not\in X$. If $X,Y$ are distinct 3-vertex-separators of
    $\cl_3^*(G)$ contained in $D$ then choose $x \in X \sm Y$ and $y \in Y \sm X$.
    Set $\overline D = D - xy$ and observe that $\overline D$ is a 3-connected plane triangulation such that each separator of $\cl_3^*(G)$ contained in $D$ induces a face of $\overline D$.
    Also observe that there is a new 3-vertex separator, $Z = V(D)\sm \{x,y\}$, of $\overline G$, which is not a separator of $\cl_3^*(G)$ and that $x$ and $y$ lie in distinct $3$-cleavage graphs of $\cl_3^*(\overline G)$ at $Z$. 
    It follows that the 4-block tree of $\overline G$ is a path.
\end{proof}

\begin{lemma}\label{lem:braced} 
Let $G=(V,E)$ be a graph with the strong $d$-cleavage property 
for some $d\geq 3$ and $e=uv$ be an  edge of $K(V)$ with $e\notin E(G)$. 
Then $G+e$ has the strong $d$-cleavage property.
In particular, if $G+e$ is $(d+1)$-connected, 
then $G+e$ is globally rigid in $\R^d$. 
\end{lemma}
\begin{proof}
Since the $d$-cleavage property is preserved by edge addition,  $G+e$ has the $d$-cleavage property.
We proceed  by induction on the number of $(d+1)$-blocks in $G$.
Let $P=D_0X_1D_1\ldots X_sD_s$ be a  shortest path in the $(d+1)$-block tree $T$ of $G$ from a $(d+1)$-block $D_0$  which contains $u$ to a $(d+1)$-block $D_k$  which contains $v$.

Suppose $s=0$. Then $u$ and $v$ both belong to $D_0$.
If $D_0$ is globally rigid then $D_0+e$ is globally rigid.
On the other hand, if $D_0$ is a plane triangulation then $d=3$ and 
we can use Theorem~\ref{thm:globplane} and the hypothesis that  $e\notin E(G)$, to again deduce that
$D_0+e$ is globally rigid.  This in turn implies that $G+e$ has the strong $d$-cleavage property 
since its $(d+1)$-block tree can be obtained from that of $G$ by replacing $D_0$ with $D_0+e$.

Thus we may assume $s\geq 1$ and hence $e$ is not an edge of $\cl_d^*(G)$. Replacing $G$ by $\cl_d^*(G)$ and using the fact that  $\cl_d^*(G)\subseteq \cl_d^*(G+e)$, we may assume that each $(d+1)$-block of $G$ is  either a   complete graph or a plane triangulation (and that the second alternative is only possible when $d=3$). 

Let $G'=\bigcup_{i=0}^s D_i$. Then
$G'+e$ is a $(d+1)$-block of $G+e$ and all other $(d+1)$-blocks of $G+e$ are $(d+1)$-blocks of $G$. Hence it will suffice to show that $G'+e$ is globally rigid in $\R^d$.
If  $d=3$ then this follows from
Lemma~\ref{lem:cleavagepath} and Theorem~\ref{thm:globplane}. Hence we may assume that $d\geq 4$.
Let $p$ be a generic realisation of $G'$ in $\R^d$. 
Since $D_i$ is complete for $0\leq i\leq s$ and $|X_i|=d$ for $1\leq i\leq s$, 
the set of equivalent realisations of $(G',p)$ has exactly $2^s$ congruence classes, 
each of which can be  obtained from $p$ by choosing $I\subseteq \{1,2,\ldots,s\}$ and then recursively reflecting the vertices  of $\bigcup_{j\geq i} V(D_j)$ 
in the hyperplanes containing the vertices of $X_i$ for each $i\in I$.  
Since $u,v\notin \bigcup_{j=1}^s X_j$, the composition of these reflections does not preserve the length of the edge $e$ (unless $I= \emptyset$), and hence 
$(G'+e,p)$ is globally rigid in $\R^d$ 
(see \cite[Lemma 3.10]{DGTJ}). 
Thus $G+e$ has the strong $d$-cleavage property.
\end{proof}

\subsection{Gluing results for graphs with the strong \texorpdfstring{$d$}{d}-cleavage property}

We first derive a result for the special case of plane triangulations.

\begin{theorem}\label{thm:planeunion} Suppose $G_1,G_2$ are  distinct plane triangulations and both $G_2$ and $G_1\cup G_2$  are   $4$-connected. Then $G_1\cup G_2$ is globally rigid in $\R^3$. 
\end{theorem}
\begin{proof}
Let $|X|=V(G_1)\cap V(G_2)$. 
The statement follows from Theorem~\ref{thm:globplane} if $X=V(G_1)$.
So we may suppose that $4\leq |X|<|V(G_1)|$.
If $V(G_2) = X$ then, using the fact that no 4-connected plane triangulation can properly contain a plane triangulation with at least 4 vertices, it follows that $G_1 \cup G_2 \neq G_1$ and then, by Theorem \ref{thm:globplane}, that $G_1 \cup G_2$ is globally rigid. Therefore we may assume that $|X| < |V(G_2)|$.

We prove the global rigidity of $G_1\cup G_2$ by applying Theorem \ref{lem:globunion}.
For this, 
we need to show that 
\begin{equation}
\label{eq:rooted_minor}
\begin{split}
&\text{$G_2$ has an $x_1x_2$-path which is internally disjoint from $X$} \\
&\text{for two vertices  $x_1,x_2\in X$ with $x_1x_2\notin E(G_2)$.}
\end{split}
\end{equation}
To see this, let $G'$ be a connected component of $G_2-X$ 
and $Y$ be the set of vertices in $X$ which are adjacent to $G'$. 
The hypothesis that  $G_1\cup G_2$ is $4$-connected implies that $|Y|\geq 4$. 
Since $G_2$ is a  $4$-connected plane triangulation, 
$G_2$ has no $K_4$-subgraph, and hence $G_2[Y]$ is not complete.  Thus we may choose two non-adjacent vertices  $x_1,x_2\in Y$. Then $G_2$ has an $x_1x_2$-path with all its internal vertices disjoint from $X$.

By (\ref{eq:rooted_minor}), 
$(G_2,X)$ has a rooted $(H,X)$-minor where $H$ is the graph with $V(H)=X$ and $E(H)=\{x_1x_2\}$.  
In addition, the 4-connectivity of $G_1\cup G_2$ implies that $G_1\cup K(X)$ is 4-connected and we can now use the facts that $|X|\geq 4$ and  that no 4-connected plane triangulation can properly contain a plane triangulation with at least 4 vertices to deduce that $G_1\cup K(X)\neq G_1$. Then $G_1\cup K(X)$ and $G_2\cup H$ are both globally rigid by Theorem \ref{thm:globplane}.  Hence $G_1\cup G_2$ is globally rigid by Theorem \ref{lem:globunion}. 
\end{proof}

We can now prove the main result of this section.

\begin{theorem} \label{thm:cleavage_globrigid} Let  $G_1$ and $G_2$ be distinct graphs with the strong $d$-cleavage property
for some $d\geq 3$.
Suppose that 
$K_{d}\subseteq G_1\cap G_2$ and  $G=G_1\cup G_2$ is  $(d+1)$-connected.
Then $G$ is globally rigid  in $\R^d$.
\end{theorem}
\begin{proof}
Suppose, for a contradiction, that the theorem is false and that $(G_1,G_2)$ is a counterexample chosen so that the ordered triple defined by the following three parameters is as small as possible in the lexicographic order. 
\\[1mm] 
(A1) The total number of $(d+1)$-blocks in $G_1$ and $G_2$.\\ 
(A2) The total number of planar 4-blocks in $G_1$ and $G_2$ when $d=3$. (This parameter is set to zero when $d\geq 4$.)\\ 
(A3) $|V(G_1)\triangle V(G_2)|$.\\[1mm]

Since $\cl_d^*(G_1)\cup \cl_d^*(G_2)\subseteq \cl_d^*(G_1\cup G_2)$, we may assume that each $(d+1)$-block of $G_i$ is   a complete graph or a planar triangulation for both $i=1,2$, and that the second alternative can only occur when $d=3$.
We refer to each vertex of  $V(G_1)\cap V(G_2)$ as a {\em shared vertex},
and each copy of $K_{d}$  in $G_1\cap G_2$  as a {\em shared $K_{d}$}.
A $(d+1)$-block $D$ in $G_1$ (respectively $G_2$) is said to be a {\em linking block} if 
$D$ contains a shared $K_d$ and $|V(D)\cap V(G_2)|\geq d+1$ (respectively $|V(D)\cap V(G_1)|\geq d+1$). 
A {\em linking block pair of $G$} is a pair $(D_1, D_2)$ where $D_i$ is a linking block in $G_i$ for both $i=1,2$,  
$|V(D_1)\cap V(D_2)|\geq d+1$ and $D_1\cap D_2$ contains a shared $K_d$. (Note that the third condition follows from the second condition when $d\geq 4$ since each $D_i$ is complete in that case.)

We first investigate properties of linking block pairs.
\begin{claim}\label{claim:glob3}
Suppose that $G$ has a linking block pair  $(D_1, D_2)$. 
Then $V(D_2)\cap V(G_1)\subseteq V(D_1)$ and $V(D_1)\cap V(G_2)\subseteq V(D_2)$.
\end{claim}
\begin{proof}
Suppose, for a contradiction, that $V(D_2)\cap V(G_1)\not\subseteq V(D_1)$. 
Choose a shortest path  $P=D_{1,0}X_{1,1}D_{1,1}\dots X_{1,k_1} D_{1,k_1}$  in the $(d+1)$-block tree of $G_1$ from $D_{1,0}:=D_1$ to a $(d+1)$-block  $D_{1,k_1}$ of $G_1$ that contains a vertex $x$ in $(V(D_2)\cap V(G_1))\setminus  V(D_1)$.
Let $G_1'=\bigcup_{j=0}^{k_1} D_{1,j}$ and  $X=V(G_1')\cap V(D_2)$.
Then $k_1\geq 1$ and $G_1'\cup D_2$ is $(d+1)$-connected.

We show that $G_1'\cup D_2$ is globally rigid in $\R^d$.
When $d=3$, this follows from the 4-connectivity of $G_1'\cup D_2$, Lemma~\ref{lem:cleavagepath}, and Theorem~\ref{thm:planeunion},
so we may assume that $d\geq 4$. Then, since $|V(D_1)\cap V(D_2)|\geq d+1$ by the definition of linking block pairs, 
we can choose a vertex $y\in( V(D_1)\cap V(D_2))\setminus X_{1,1}$.
Then $G_1'+xy$ is $(d+1)$-connected and hence is globally rigid in $\R^d$ by Lemma~\ref{lem:braced}.
Moreover, since $D_2$ is complete, $xy\in E(D_2)$. 
Since  $|V(G_1')\cap V(D_2)|\geq d+1$, 
the global rigidity gluing property implies that $G_1'\cup D_2$ is globally rigid in $\R^d$.

Construct $G_1''$ from $G_1$ by replacing $G_1'$ with a complete graph on $V(G_1')$. Then $G_1''$ has the strong $d$-cleavage property and
$G_1''$  has fewer  $(d+1)$-blocks than  $G_1$ since $k_1\geq 1$. Hence $G''=G''_1\cup G_2$ is globally rigid by  induction with respect to the parameter (A1). 
Since $\cl_d^*(G)=\cl_d^*(G_1\cup G_2)=\cl_d^*(G_1''\cup G_2)=\cl_d^*(G'')$, 
the global rigidity of $G''$ implies the global rigidity of $G$, and contradicts our initial choice of $(G_1,G_2)$ as a counterexample to the theorem. A similar argument leads to a contradiction if we suppose that $V(D_1)\cap V(G_2) \not\subseteq V(D_2)$.
\end{proof}

\begin{claim}\label{claim:glob4}
Suppose that $G$ has a linking block pair  $(D_1, D_2)$. 
Then $D_1=D_2$.
\end{claim}
\begin{proof}
Suppose $D_1\neq D_2$.
The definition of a linking block pair implies that $D_1\cup D_2$ is globally rigid
by the global rigidity gluing property when $d\geq 4$ and by Theorem~\ref{thm:planeunion} when $d=3$.
Let $G_i^*$ be obtained from $G_i$ by replacing $D_i$ by a clique on 
$V(D_1)\cup V(D_2)$ for both $i=1,2$. 
Since $D_1\cup D_2$ is globally rigid, the global rigidity of $G_1^*\cup G_2^*$ would imply the global rigidity of $G$.
Hence  $G_1^*\cup G_2^*$  is not globally rigid.

By Claim~\ref{claim:glob3},  
$G_i^*$ has the strong $d$-cleavage property and has the same number of $(d+1)$-blocks as $G_i$.
If either of $D_1$ or $D_2$ is not complete, then parameter (A2) would be decreased in $(G_1^*, G_2^*)$, 
and we could use induction to deduce that $G_1^*\cup G_2^*$ is globally rigid. 
Hence  $D_1$ and $D_2$ are both complete graphs.    
Since $D_1\neq  D_2$,  we have $V(D_1)\neq V(D_2)$, and we can now  apply induction with respect to parameter (A3) to deduce that $G_1^*\cup G_2^*$ is globally rigid, a contradiction.  Thus $D_1=D_2$.
\end{proof}

A path $P_1=D_{1,0}X_{1,1} \dots, X_{1,k_1} D_{1,k_1}$ in the $(d+1)$-block tree of $G_1$ is  said to be a {\em linking path in $G_1$} if 
\begin{enumerate}
\item $k_1\geq 1$, 
\item $D_{1,k_1}$ contains a shared vertex, and 
\item $G_2$ has a $(d+1)$-block $D_2$ with $K_d\subseteq D_{1,0}\cap D_2$ and $X_{1,1}\not\subseteq V(D_2)$.
\end{enumerate}
The block $D_{1,0}$ is said to be the {\em initial block} of $P_1$.
The $(d+1)$-block $D_2$ in the third property of the  above definition is said to be  a {\em certificate} for $P_1$.
A linking path in $G_2$ is defined analogously.

\begin{claim}\label{claim:glob5}
Let $i, j\in \{1,2\}$ with $i\neq j$.
Suppose $G_i$ has a linking path $P_i=D_{i,0}X_{i,1} \dots, X_{i,k_i} D_{i,k_i}$
and $D_j$ is a certificate for $P_i$.
Then $\left(\bigcup_{s=0}^{k_i} D_{i,s}\right)\cap D_j=D_{i,0}\cap D_j\cong K_d$.
\end{claim}
\begin{proof}
We may assume, without loss of generality, that $i=1$ and $j=2$.
By the definition of a linking path, $K_d\subseteq D_{1,0}\cap D_2$.
If $D_{1,0}$ and $D_2$ share more than $d$ vertices, 
then  $(D_{1,0}, D_2)$ would be a linking block pair and 
Claim~\ref{claim:glob4} would give $D_{1,0}=D_{2}$. 
This would imply that $D_2$ contains $X_{1,1}$ and contradict the fact that $D_2$ is a certificate for $P_1$.
Hence $D_{1,0}\cap D_2\cong K_d$.

Choose $s$ with $ 1\leq s\leq k_1$ such that
$(V(D_{1,s})\setminus V(D_{1,0}))\cap V(D_2)\neq \emptyset$ and such that $s$ has been chosen to be as small as possible with this property.
So $(V(D_{1,r})\setminus V(D_{1,0}))\cap V(D_2) =  \emptyset$ for $1 \leq r \leq s-1$. Let
$G_1'=\bigcup_{t=0}^s D_{1,t}$.
We will show that $G_1'\cup D_2$ is globally rigid. 

Let $X=V(G_1')\cap V(D_2)$. 

As a first step we show that $D_2+K(X)$ is globally rigid. If $d\geq 4$ then $D_2=D_2+K(X)$ is a complete graph.
In the case that $d=3$ and $D_2$ is not globally rigid then it is a $4$-connected plane triangulation and therefore cannot contain $K(X)$ since $|X|\geq 4$. By Theorem \ref{thm:globplane} $D_2 + K(X)$ is globally rigid in this case also.

We next consider $G_1'$.
Suppose that $z,u_1,\cdots,u_l,x$ is a path in $G_1'$ of minimal length such that $z \in (V(D_{1,0})\sm X_{1,1}) \cap V(D_2)$ and $x \in (V(D_{1,s})\sm V(D_{1,0})) \cap V(D_2)$. Such a path exists since $V(D_2)$ does not contain $X_{1,1}$. By the minimal choice of $s$ it follows that $x \not\in V(D_{1,r})$ for $1\leq r \leq s-1$. Also, by the minimal choice of $s$ and the choice of $u_1,\cdots, u_l$, it follows that $u_1,\cdots,u_l \not\in X$. So, letting $H=(X,\{xz\})$, we have that $(H,X)$ is a rooted minor of $(G_1',X)$.
Since $G_1'+xz$ is $(d+1)$-connected, $G_1'+xz$ is globally rigid by Lemma~\ref{lem:braced}. 
Hence we can  apply Theorem \ref{lem:globunion} to $G_1'$ and $D_2$ with  
$H=(X,\{xz\})$ to deduce that $G_1'\cup D_2$ is globally rigid. 

We complete the proof of the claim by constructing $G_1''$ from $G_1$ by 
replacing $G_1'$ by a complete graph on $V(G_1')$. 
Then $G_1''$ has the strong $d$-cleavage property by Lemma~\ref{lem:braced} and  $G_1''$  has fewer  $(d+1)$-blocks than $G_1$ since  $s\geq 1$. Hence $G''=G''_1\cup G_2$ is globally rigid by  induction with respect to the parameter (A1). 
Since $G_1'\cup D_2\subseteq G$, the global rigidity of $G_1'\cup D_2$ implies  $\cl_d^*(G)=\cl_d^*(G'')$
and hence $G$ is globally rigid, a contradiction.
\end{proof}

\begin{claim}\label{claim:glob1}
Let $i, j\in \{1,2\}$ with $i\neq j$.
Suppose that $G_i$ has a linking path with initial block $D_{i,0}$.
Then  $D_{i,0}$  is a $(d+1)$-block of $G_j$.
\end{claim}
\begin{proof}
We may assume, without loss of generality, that $i=1$ and $j=2$.
Suppose for a contradiction that $D_{i,0}$ is not a $(d+1)$-block of $G_2$. Let $P_1=D_{1,0}, X_{1,1}D_{1,1} \dots X_{1,k_1} D_{1,k_1}$  be a linking path of $G_1$. 
By the definition of a linking path, $G_2$ has a certificate $D_2$ for $P_1$ and
$D_{1, k_1}$ contains a shared vertex $x$.
Let $P_{2}=D_{2,0} X_{2,1} \dots D_{2,k_{2}}$ be a path in the $(d+1)$-block tree of $G_{2}$ from $D_{2,0}:=D_2$ to  a $(d+1)$-block $D_{2,k_{2}}$ that contains $x$. 
We may assume that the  pair $(P_1,P_2)$ has been chosen so that $k_1+k_2$ is as small as possible.
By Claim~\ref{claim:glob5}, $D_2$ does not contain $x$, so $k_2\geq 1$.
Let $G_i'=\bigcup_{j=0}^{k_i} D_{i,j}$ for $i=1,2$. Consider the following two cases.
\paragraph{\boldmath Case 1: $\left(\bigcup_{j=0}^{k_1-1} D_{1,j} \right)\cap G_2'\neq D_{1,0}\cap D_{2,0}$.} 
Let $y$ be a vertex in  $\bigcup_{j=0}^{k_1-1} D_{1,j} \cap G_2'$ which does not belong to $D_{1,0}\cap D_{2,0}$.
If $y\notin V(D_{1,0})$, then we could replace $P_1$ by a shorter path, since $y$ is a shared vertex, and contradict the minimality of $k_1+k_2$.
Hence $y\in V(D_{1,0})\sm V(D_{2,0})$.  
Since $y$ is in $G_2'$, we can choose  a shortest subpath $P_2'$   
of $P_2$ 
 which starts at $D_{2,0}$ and ends at a $(d+1)$-block that contains $y$. Since $y\notin V(D_{1,0})\cap V(D_{2,0})$, $P_2'$ consists of more than one block.
If $X_{2,1}\not\subseteq V(D_{1,0})$, then $D_{1,0}$ is a certificate for $P_2'$ to be a linking path.
This would contradict Claim~\ref{claim:glob5} since $D_{1,0}$ and $P_2'$ intersect at the vertex $y\notin V(D_{2,0})$.
Hence  $X_{2,1}\subseteq V(D_{1,0})$.

Suppose $D_{2,1}$ contains $X_{1,1}$. Then we have 
$X_{1,1}\cup X_{2,1}\subseteq V(D_{1,0})\cap V(D_{2,1})$.
Since $X_{1,1}\not\subseteq V(D_{2,0})$ (as $D_{2,0}$ is a certificate for $P_1$),  
we have $|X_{1,1}\cup X_{2,1}|\geq d+1$, which means that 
$(D_{1,0}, D_{2,1}) $ is a linking block pair.
By Claim~\ref{claim:glob4}, $D_{1,0}=D_{2,1}$, which contradicts our initial assumption that $D_{1,0}$ is not a $(d+1)$-block of $G_2$.

Hence, $D_{2,1}$ does not contain $X_{1,1}$.
Since $D_{2,1}$ shares  $X_{2,1}$ with $D_{1,0}$ by $X_{2,1}\subseteq V(D_{1,0})$, $D_{2,1}$ can also serve as a certificate for $P_1$.
This contradicts the minimality of $k_1+k_2$ since we could replace $P_2$ by the shorter path $P_2''=D_{2,1} X_{2,2} \dots D_{2,k_{2}}$.
This completes the discussion of Case 1.

\paragraph{\boldmath Case 2: $\left(\bigcup_{j=0}^{k_1-1} D_{1,j} \right)\cap G_2'= D_{1,0}\cap D_{2,0}$.} 
The proof for this case is similar to that of Claim~\ref{claim:glob5}, but we give the details for the sake of completeness.
 We first 
show that $G_1'\cup G_2'$ is globally rigid. 
Let $X=V(G_1')\cap V(G_2')$. 
Since $G_2'$ has the strong $d$-cleavage property and is not $(d+1)$-connected, and $G_2'+K(X)$ is $(d+1)$-connected, $G_2'+K(X)$ is globally rigid by Lemma~\ref{lem:braced}.
We next consider $G_1'$.
By the definition of linking paths, there is a vertex $z$ of $D_{1,0}\cap D_{2,0}$ not contained in $X_{1,1}$.
By the assumption of Case 2, we can also choose a  vertex $x$ with $x\in V(D_{1,k_1})\cap V(G_2')$ and $x\notin \bigcup_{j=0}^{k_1-1} V(D_{1,j})$ such that 
$H=(X,\{xz\})$ is a rooted-minor of $G_1'$.
Since $G_1'+xz$ is $(d+1)$-connected, $G_1'+xz$ is globally rigid by Lemma~\ref{lem:braced}. 
Hence we can  apply Theorem \ref{lem:globunion} with  
$H=(X,\{xz\})$ to deduce that $G_1'\cup G_2'$ is globally rigid. 

Construct $G_1''$ from $G_1$ by 
replacing $G_1'$ by a complete graph on $V(G_1')$. Then $G_1''$ has the strong $d$-cleavage property and  $G_1''$  has fewer  $(d+1)$-blocks than $G_1$ since  $k_1\geq 1$. Hence $G''=G''_1\cup G_2$ is globally rigid by  induction with respect to the parameter (A1). Since $G_1'\cup G_2'$ is globally rigid, $\cl_d^*(G)=\cl_d^*(G'')$ 
and hence $G$ is globally rigid. This contradicts our choice of $(G_1,G_2)$ as a counterexample to the theorem and completes the discussion of Case 2 and the proof of the claim.
\end{proof}

\begin{claim}\label{claim:glob2}
Let $i,j \in \{1,2\}$ with $i\neq j$.
Suppose that $G_i$ has  a linking block $D_i$.
Then $G_j$ has a linking block $D_j$ such that $(D_i,D_j)$ is a linking block pair.
\end{claim}
\begin{proof}
Without loss of generality, assume $i=1$ and $j=2$.
Since $D_1$ is a linking block in $G_1$, we may choose a shortest  path $P=D_{2,0}X_{2,1} \dots D_{2,k_2}$ in the $(d+1)$-block tree of $G_2$  such that $H\subseteq D_1\cap D_{2,0}$  for some  $H\cong K_d$ and $D_1\cap D_{2,k_2}$ contains a shared vertex $x$ with $x\not\in V(H)$. 
If $k_2\geq 1$, then the minimality of $P$ would imply that $P$ is a linking path and $D_1$ is a certificate for $P$.
This would contradict Claim~\ref{claim:glob5} since $D_1\cap D_{2,k_2}$ is non-empty.
Thus, $k_2=0$ and  $(D_1, D_{2,0})$ is a linking block pair.
\end{proof}

We next show:
\begin{claim}\label{claim:glob6}
$G_1$ contains a linking block.
\end{claim}
\begin{proof}
By a hypothesis of the theorem, $G_1$ contains a shared $K_d$.
Let $P=D_{1,0} X_{1,1},\dots D_{1,k_1}$ be a shortest path in the $(d+1)$-block tree of $G_1$ such that 
$H\subseteq D_{1,0}\cap G_2$ for some  $H\cong K_d$,  
and $D_{1,k_1}$ contains a shared vertex $x$ with $x\not\in V(H)$.
Since $G_1\cup G_2$ is $(d+1)$-connected, such a path $P$ exists.

Suppose $k_1\geq 1$. Then the  minimality of $P$ implies that $V(H)\neq X_{1,1}$ and 
$G_1[X_{1,1}]$ is not a shared $K_d$ (since otherwise we could replace $P$ by its subpath from $D_{1,1}$ to $D_{1,k_1}$).
This in turn implies that $P$ is a linking path, and hence $D_{1,0}$ is a 
$(d+1)$-block of $G_2$ by  Claim~\ref{claim:glob1}. This contradicts the fact that $G_1[X_{1,1}]$ is not a shared $K_d$.

Hence $k_1=0$ and $D_{1,0}$ is the desired linking block in $G_1$.
\end{proof}

Claims~\ref{claim:glob6},~\ref{claim:glob2} and~\ref{claim:glob4} imply that there exists   a linking block pair $(D_1, D_2)$ with $D_1=D_2=:D_0$. Let $T_i$ be the $(d+1)$-block tree of $G_i$ for $i=1,2$. Then $D_0$ is a vertex of both $T_1$ and $T_2$. 
Let ${\cal D}$ be the set of all common $(d+1)$-blocks of $G_1$ and $G_2$ which belong to the same connected component of $T_1\cap T_2$ as $D_0$.
A shared vertex $v$ is said to be a {\em gate} of $G$ if $v\notin V(D)$ for all $D\in {\cal D}$.
If no shared vertex is a gate, then the $(d+1)$-connectivity of $G$ would imply that 
$V(D_0)=V(G_1)=V(G_2)=V(G)$ and we could now use Lemma~\ref{lem:braced} to deduce that  $G$ is globally rigid, contradicting the choice of $(G_1,G_2)$.
Thus $G$ has a gate. 

Let $P'=D_{2,0} X_{2,1}D_{2,1},\dots D_{2,k_2}$ be a shortest path in $T_2$ between a $(d+1)$-block $D_{2,0}$ in ${\cal D}$ and a $(d+1)$-block $D_{2,k_2}$ that contains a gate $v$.
Then, since $v\notin V(D_{2,0})$, $k_2\geq 1$. Also, since $D_{2,0}\in {\cal D}$, $D_{2,0}$ belongs to both $T_1$ and $T_2$ and  contains $X_{2,1}$.   In addition,  the choice of $P'$ as a shortest path implies that $D_{2,0}$  does not contain $X_{2,2}$  and $D_{2,1}$ is not in $T_1$.
If $k_2\geq 2$, then $D_{2,1}X_{2,1}\dots D_{2,k_2}$ would be a linking path with  certificate $D_{2,0}$ 
(since $D_{2,0}$ belongs to $T_1$).
We could now use  Claim~\ref{claim:glob1} to deduce that $D_{2,1}$ belongs to $T_1$, which is a contradiction. 
Hence $k_2=1$ and $D_{2,1}$ is a linking block.
Claims~\ref{claim:glob2} and \ref{claim:glob4} now give the same contradiction that
$D_{2,1}$ is in $T_1$. 
This completes the proof of the theorem.
\end{proof}

Combining Theorem \ref{thm:cleavage_globrigid} with Lemma~\ref{lem:cleave1}, we immediately obtain:

\begin{corollary} \label{cor:cleavage_globrigid} Let $G_1$ and $G_2$ be graphs with the strong $d$-cleavage property and with 
$K_d\subseteq G_1\cap G_2$. Then $G_1\cup G_2$ has the strong $d$-cleavage property.
\end{corollary}

\section{Global rigidity of simplicial \texorpdfstring{$k$}{k}-circuits}
\label{sec:globproof}
In this section we prove  Theorems~\ref{thm:globrigid_3_main} and \ref{thm:globrigid_d_main}.
To simplify terminology, we will say that {\em a graph $G$ is a  simplicial $k$-circuit graph} if $G=G(\cS)$ for some simplicial $k$-circuit $\cS$. 

\subsection{Vertex splitting and global rigidity}

The overall strategy in our proof of Theorems~\ref{thm:globrigid_3_main} and \ref{thm:globrigid_d_main} will be to contract an edge $e\in E(G)$, apply induction to $G/e$ and then use vertex splitting to recover $G$.
The main difficulty in this strategy is the lack of a global rigidity analogue of Whiteley's vertex splitting theorem (Lemma~\ref{lem:split}).
Theorem \ref{thm:hendrickson} implies that $(d+1)$-connectivity is necessary for global rigidity in $\R^d$,
so vertex splitting should preserve the $(d+1)$-connectivity. 
The question of whether this condition is sufficient is a long-standing open problem in rigidity theory. More precisely we have the following.
\footnote{Two recent developments concerning this conjecture are given in \cite{J,JT}.}

\begin{conjecture}[\cite{CW08,CW10}]\label{conj:vertex_splitting}
Suppose that $G$ is a $(d+1)$-connected graph and $uv \in E(G)$ such that $|N_G(u) \cap N_G(v)| = d-1$ and $G/uv$ is globally rigid in $\mathbb R^d$. Then $G$ is globally rigid in $\mathbb R^d$.
\end{conjecture}

We will prove the following. 
\begin{theorem}\label{thm:vertex_splitting_k}
Let $G$ be a  $(k+2)$-connected simplicial $k$-circuit graph for some $k\geq 2$ and $uv \in E(G)$ such that $G/uv$ is globally rigid in $\R^{k+1}$. Then  $G$ is  globally rigid  in $\R^{k+1}$.
\end{theorem}

The proof strategy  for Theorem~\ref{thm:vertex_splitting_k} is also an induction based on edge contraction and vertex splitting
but the details are technically more complicated. As a result, we will delay the proof until Section~\ref{sec:coincidence}.

We have seen that if $uv\in E(\cS)$ where $\cS$ is a nontrivial simplicial $k$-circuit then $u$ and $v$ have at least $k$ common neighbours in $G(\cS)$. We shall see in Section \ref{sec:coincidence} that the case of Theorem \ref{thm:vertex_splitting_k} where $u$ and $v$ have at least $k+1$ common neighbours is relatively straightforward. Thus the remaining (and most difficult) part of Theorem \ref{thm:vertex_splitting_k} is precisely the case that is a special case of Conjecture \ref{conj:vertex_splitting}.

\subsection{Global rigidity in  dimension at least four}

\paragraph{Proof of Theorem~ \ref{thm:globrigid_d_main}} We will simplify terminology by suppressing reference to the ambient space $\R^{k+1}$ throughout this proof. Necessity follows immediately from Theorem \ref{thm:hendrickson}. To prove sufficiency we suppose, for a 
contradiction that there exists a $(k+2)$-connected simplicial $k$-circuit graph $G=(V,E)$ which is not globally rigid and such that  $|E|$ is as small as possible. 

\begin{claim}\label{clm:glob1_k}
For  each $e\in E$, there exists a simplicial $k$-circuit $\cS$ such that $G=G(\cS)$ and $\cS/e$ is a $k$-simplicial circuit. In particular, $G/e$ is a simplicial $k$-circuit graph for all $e\in E$.
\end{claim}
\begin{proof}
Choose a simplicial $k$-circuit $\cS$ such that $G=G(\cS)$.  Since $G\neq K_{k+1}$, $\cS$ is non-trivial. Let $\{\cS_1,\cS_2,\ldots,\cS_m\}$ be a Fogelsanger decomposition of $\cS$ with respect to $e$ and put
$G(\cS_i)=G_i=(V_i,E_i)$ for all $1\leq i\leq m$. 
By Lemma~\ref{lem:3}\ref{en1:0}, $\cS_i/e$ is a simplicial $k$-circuit. Thus, if $E_i=E$ for some $1\leq i\leq m$, then we can put $G=G(\cS_i)$ and the claim holds 
since $\cS_i/e$ is a simplicial $k$-circuit. 
Hence, we may assume that $E_i\neq E$  and Lemma~\ref{lem:3}\ref{en1:c} now gives 
$|E_i|<|E|$ for all $1\leq i\leq m$.  By Lemma~\ref{lem:3}\ref{en1:d}, we can also assume the Fogelsanger  decomposition of $\cS$ is ordered so that $(\bigcup_{j=1}^{i-1} G_j)\cap G_i$ contains a copy of $K_{k+1}$ for all $2\leq i\leq m$.  

Since each $G_i$ is the graph of a non-trivial simplicial $k$-circuit, it has the $(k+1)$-cleavage property. The minimality  of $G$ and Lemma~\ref{lem:mincleavage} now imply that  every $(k+2)$-block of each $G_i$ is  globally rigid   in $\R^{k+1}$ 
and hence each  $G_i$ has the strong $(k+1)$-cleavage property.
An easy induction using Corollary \ref{cor:cleavage_globrigid} now implies that  $H_j:=\bigcup_{i=1}^{j} G_i$ has the strong $(k+1)$-cleavage property
for all $1\leq j\leq m$. 
Thus $H_m=G$ is a $(k+2)$-connected graph with the strong $(k+1)$-cleavage property and so is globally rigid. This contradicts the choice of $G$ and completes the proof of 
the claim.
\end{proof}

Choose an edge $e \in E$. 
If  $G/e$ is  $(k+2)$-connected, then it is a $(k+2)$-connected simplicial $k$-circuit  graph by Claim \ref{clm:glob1_k},
and we can apply induction and Theorem~\ref{thm:vertex_splitting_k} 
to deduce that $G$ is globally rigid, a contradiction.  
Hence, 
\begin{equation}\label{eq:glob2_k}
\mbox{$G/e$ is not $(k+2)$-connected for each $e\in E$.}
\end{equation} 

By (\ref{eq:glob2_k}), each edge $e\in E$ is induced by  a $(k+2)$-vertex separator $X_e$ of $G$.
We claim that there is at least one edge $e\in E$ that is not induced by a $(k+2)$-clique-separator.
To see this, choose an edge $f$ induced by a $(k+2)$-clique-separator $X_f$ 
such that some component $F$ of $G-X_f$ is as small as possible and put $H_f=G[V(F)\cup X_f]$.
Then, the minimality of $F$ implies that 
no edge $e\in E(H_f)\setminus E(X_f)$ is induced by   a $(k+2)$-clique-separator.

Thus we may choose an edge $e\in E$ such that $e$ is induced by a  $(k+2)$-vertex separator $X_e$ of $G$ which is not a clique. By Claim \ref{clm:glob1_k}, we can choose a simplicial $k$-circuit $\cS$ such that $G=G(\cS)$ and $\cS/e$ is a simplicial $k$-circuit.

Let $H$ be a component of $G-X_e$ and put $\cS_1=\{U \in \cS:U\subseteq V(H)\cup X_e\}$,
and $\cS_2=\cS\sm \cS_1$. 
Then Lemma~\ref{lem:1} and the fact that $\cS$ is a simplicial $k$-circuit give $\partial \cS_1=\cT=\partial \cS_2$ for some nonempty simplicial $(k-1)$-cycle $\cT$  with $V(\cT)\subseteq X_e$.

Suppose $|V(\cT)|={k+1}$. Then Lemma~\ref{lem:34sep_a} implies that 
$V(\cT)\not\in \cS_i$ 
and $\cS_i'=\cS_i\cup \{V(\cT)\}$ is a simplicial $k$-circuit for both $i=1,2$. 
Moreover, since $V(\cS_i)\sm X_e$ is nonempty, it follows that $|E(G_i)| < |E|$ for $i=1,2$.
By the minimal choice of $G$ we deduce that $G_i=G(\cS_i)$ has the strong $(k+1)$-cleavage property  for both $i=1,2$. 
Since $G$ is $(k+2)$-connected and $K(V(\cT))\subseteq G_1\cap G_2$, Theorem~\ref{thm:cleavage_globrigid} now implies that $G=G_1\cup G_2$ is globally rigid, a contradiction.

Hence $|V(\cT)|={k+2}$, so $V(\cT)=X_e$.  
Since $X_e$ is not a clique, Lemma~\ref{lem:34sep_b} implies that
$\cT\cong \cL_{k-1}$. By the definition of $\cL_{k-1}$,
$\cT=\cT'\triangle\cT''$ where $\cT'\cong\cK_{k-1}\cong \cT''$ and $\cT'\cap \cT''=\{T\}$ for some common $(k-1)$-simplex $T$ of  $\cT'$ and $\cT''$.   
Since $|X_e\setminus T|=2$, we can denote $X_e\sm T=\{u,v\}$.
 Then  Lemma~\ref{lem:34sep_b} 
implies that  $\cS_i'=\cS_i\triangle \{X_e-u,X_e-v\}$ is a simplicial $k$-circuit for both $i=1,2$.
We can now apply induction  
to deduce that $G_i=G(\cS_i')$ has the strong $(k+1)$-cleavage property
for both $i=1,2$. If there exists a vertex $w\in X_e\sm V(\cS_i')$ then $X_e-u$ and $X_e-v$ would be the only possible $k$-simplices in $\cS_i$ which contain $w$. This would imply that there are no edges in $G$ from $w$ to $V(\cS_i)\sm X_e$ and $X_e-w$ would be a $(k+1)$-vertex separator of $G$, contradicting the $(k+2)$-connectivity of $G$. Hence  $X_e\subseteq V(\cS_i')$ and we can use Lemma~\ref{lem:braced} to deduce that $G_i'=G_i\cup G[X_e]$ has the strong $(k+1)$-cleavage property for both $i=1,2$.

Since $k\geq 3$, $G(\cL_{k-1})$ is a complete graph on $k+2$ vertices with one edge deleted, and hence $G[X_e]=K(X_e)-uv$.\footnote{This is the only point in the proof of Theorem \ref{thm:globrigid_d_main} that we use $k\geq 3$.}  This implies that $G_1'\cap G_2'$ contains a copy of $K_{k+1}$.
Since $G$ is $(k+2)$-connected, Theorem~\ref{thm:cleavage_globrigid} now implies that $G=G_1'\cup G_2'$ is globally rigid.
This contradicts our choice of $G$  and completes the proof of the theorem.
\qed

\subsection{Global rigidity in  dimension three}

We will prove Theorem~\ref{thm:globrigid_3_main}.  We first derive two lemmas on 4-vertex separators in graphs which we need to deal with the fact that $G(\cL_1)=C_4$ is not a complete graph on four vertices with one edge deleted (cf. the last paragraph of the proof of Theorem~\ref{thm:globrigid_d_main}). 

\begin{lemma}~\label{lem:4con1} Suppose $G$ is a $4$-connected graph distinct from the octahedron,  $v$ is a vertex of degree four in $G$, and $G[N_G(v)]$ contains a $4$-cycle $C=u_1u_2u_3u_4u_1$. Then either
\begin{itemize}
    \item[(a)] $G-v+u_iu_{i+2}$ is $4$-connected for some $i=1,2$ or
    \item[(b)] $d_G(u_i)=4$ and $G[N_G(u_i)]$ does not contain a  $4$-cycle for some $1\leq i\leq 4$.
\end{itemize}
\end{lemma}
\begin{proof}  Suppose, for a  contradiction that neither (a) nor (b) hold. Then $G-v+u_iu_{i+2}$ is not 4-connected for both $i=1,2$. Hence there exists  a 4-vertex separator $S_i$ of $G$ and a component $H_i$ of $G-S_i$ such that $v\in S_i$ and $\{u_i,u_{i+2}\}\cap V(H_i)=\emptyset$ for both $i=1,2$. Since $G$ is 4-connected, some neighbour of $v$ belongs to $H_i$ and, relabelling if necessary, we may suppose that $u_{i+1}\in V(H_i)$.  The fact that  $u_i,u_{i+2}\in N(u_{i+1})$ implies that $\{u_i,u_{i+2}\}\subseteq S_i$. 
Since $v$ is adjacent to each connected component of $G-S_i$ by $v\in S_i$, we have $u_{i+3}\in V(G)\sm (S_i\cup V(H_i)$, and the configuration of $v, u_1,u_2,u_3, u_4$ and $S_1, S_2, H_1,H_2$ is as shown in Figure~\ref{fig:lemma63}.

\begin{figure}[ht]
\centering
\includegraphics[scale=0.6]{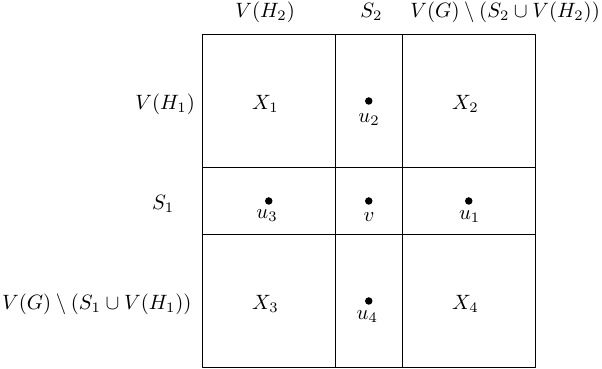}
\caption{Proof of Lemma~\ref{lem:4con1}. 
The 
rows represent a partition of $V(G)$ into $\{V(H_1), S_1, V(G)\setminus (V(H_1)\cup S_1)\}$
and the 
columns represent a partition of $V(G)$ into $\{V(H_2), S_2, V(G)\setminus (V(H_2)\cup S_2)\}$.}
\label{fig:lemma63}
\end{figure}

Let $X_1, X_2, X_3, X_4$ be subsets of $V(G)$ as in Figure~\ref{fig:lemma63}.
Since $|S_i|=4$, we have $|S_1\cap S_2|\leq 2$. 
Suppose $S_1 \cap S_2 = \{v,x\}$ for some $x \in V(G)$. If $X_1 \neq \emptyset$ then, since $N_G(v) \cap X_1 = \emptyset$, $\{x,u_2,u_3\}$ would be a $3$-vertex separator for $G(\cS)$, contradicting the fact that $G$ is 4-connected. Therefore $X_1$, and similarly $X_2$, $X_3$ and $X_4$, must be empty and so $S_1 \cup S_2 = V(\cS)$ and $G(\cS)$ is the octahedron, contradicting our hypothesis.

Hence $S_1\cap S_2=\{v\}$. 
Then, without loss of generality, we may assume $S_2\cap V(H_1)=\{u_2\}$ and $S_1\cap V(H_2)=\{u_3\}$.
The 4-connectivity of $G$ and the hypothesis that $d_{G(\cS)}(v)=4$ now imply that
$X_1=X_2=X_3=\emptyset$,
and hence $V(H_1)=\{u_2\}$ and $V(H_2)=\{u_3\}$.
This gives $N_G(u_2)=S_1$ so $d_G(u_2)=4$ and, since (b) does not hold,  $G[N_G(u_2)]=G[S_1]$ contains a $4$-cycle. This in turn implies that $u_3$ is adjacent to a vertex of $S_1-v$  and contradicts the fact that $S_2$ is a 4-vertex separator of $G$ (since $u_3$ and $S_1-v$ are contained in different components of $G-S_2$).
\end{proof}

\begin{lemma}\label{lem:4sep4con} Let $G$ be a  $4$-connected 
graph and  $S$ be a $4$-vertex separator in  $G$. Suppose that $C=u_1u_2u_3u_4u_1$ is a $4$-cycle in $G[S]$ and $H$ is a component of $G-S$. Then the graph $G'$ obtained from $G[V(H)\cup S]$ by adding a new vertex $v$  adjacent to each vertex in $S$ is $4$-connected.
In addition, if $G$ is the graph of a simplicial 2-circuit and $G'$ is not isomorphic to the octahedron, then $G[V(H)\cup S]+u_iu_{i+2}$   is $4$-connected for some $i=1,2$.
\end{lemma}
\begin{proof}  We first show that $G'$ is 4-connected.   Suppose to the contrary that $G'$ has a vertex separator $S'$ with $|S'|\leq 3$. Since $G$ is 4-connected, every component of $G'-S'$ must contain a vertex of $S$ and hence $S'$ must separate $S$ in $G'$. Since $G'[S+v]$ is a wheel centred on $v$, we have  $|S'|=3$ and $v\in S'\subseteq S+v$. 
Since $S$ is a minimal separator in $G$,  each vertex of $S\sm S'$ is adjacent to each component of $G-S$, which now contradicts the assumption that $G'-S'$ is disconnected. Hence $G'$ is 4-connected.

To prove the second part of the lemma we assume that $G$ is the graph of a simplicial 2-circuit $\cT$ and $G'$ is not isomorphic to the octahedron.
Suppose for a contradiction that  $G[V(H)\cup S]+u_iu_{i+2}$   is not $4$-connected for both $i=1,2$.
We may apply Lemma~\ref{lem:4con1} to $G'$ to deduce that,  for some $1\leq i\leq 4$, $d_{G'}(u_i)=4$ and $G'[N(u_i)]$ does not contain a 4-cycle. We may assume without loss of generality that $i=1$. Then the neighbours of $u_1$ in $G'$ are $u_2,u_4, v,w$ for some $w\in V(H)$.

By Lemma~\ref{lem:1}, $G(\partial \cS_{u_1})$ forms an even subgraph on $N_G(u_1)$. This implies that $wu_2,wu_4\in E(G)$ and $G'[N_{G'}(u_1)]$ contains the 4-cycle $wu_2vu_4w$, a contradiction.  
\end{proof}

We will also need the following lemma about simplicial $2$-circuits with the property that contracting an edge yields a 4-connected plane triangulation.

\begin{lemma}
\label{lem:planecontractions}
Let $\cT$ be a nontrivial simplicial 2-circuit such that $H =G(\cT)$ is nonplanar. Suppose that $e$ and $e'$ are distinct edges of $H$ and $H/e$ is  4-connected plane triangulation. Then $H/e'$ is not a 4-connected plane triangulation.
\end{lemma}
\begin{proof}
Suppose, for a contradiction, that $H/e'$ is a 4-connected plane triangulation. Suppose that $e=ab$.
The hypothesis that $H/e$ is 4-connected implies that $H/e$ has at least 5 vertices and so $\cT/e$ is a nontrivial simplicial $2$-circuit.
By assumption, $H$ is not planar and so $|E(H)| \geq 3|V(H)|-5$ by Theorem \ref{thm:plane}. But $|E(H/e)| \le 3|V(H/e)|-6$ since $H/e$ is planar, so $a,b$ must have at least 3 common neighbours in $H$. Since $H/e'$ is a 4-connected plane triangulation,  $a,b$ have exactly 2 common neighbours in $H/e'$. Therefore $a,b$ must have exactly 3 common neighbours in $H$.
Since $\cT$ is a nontrivial simplicial $2$-circuit, $\{a,b\}$ must be contained in exactly two 2-simplices of $\cT$ and one 3-clique of $H$ which is not a $2$-simplex of $\cT$. Suppose that $\{a,b,c\}, \{a,b,d\} \in \cT$ and
$\{a,b,z\} \not\in \cT$ is a 3-clique of $H$. Then by Lemma~\ref{lem_cont_k-1}, 
\begin{equation}
    \label{eqn:mult1}
    |(\cT/e)_{az}| = |\cT_{az}|+|\cT_{bz}| \geq 4. 
\end{equation}
Since $\cT$ does not contain any repeated 2-simplex, a $2$-simplex $\{a,z,p\}$ of $\cT/e$ is a repeated 2-simplex if and only if $\{a,z,p\}$ and $\{b,z,p\}$ are 2-simplices of $\cT$. In particular, 
since $\{c,d,z\} = N_H(a)\cap N_H(b)$,
$p \in \{c,d\}$.
Now, since $H/e$ is a 4-connected plane triangulation, $\{a,z\}$ is contained in exactly two $3$-cliques of $H/e$. Using (\ref{eqn:mult1}) and the observations in the previous two sentences, it follows that $\{a,z,c\}, \{b,z,c\},\{a,z,d\},\{b,z,d\}$ are all 2-simplices in $\cT$ and so $\cT$ contains 
$$ \cT' = \{ \{a,b,c,\},\{a,b,d\},\{a,z,c\}, \{b,z,c\},\{a,z,d\},\{b,z,d\}\},$$
which is a copy of the simplicial $2$-circuit $\cL_2$.
Since $\cT$ is a simplicial $2$-circuit, $\cT = \cT'$, contradicting
the hypothesis that $H/ab$ is 4-connected.
\end{proof}

\begin{proof}[\bf Proof of Theorem \ref{thm:globrigid_3_main}]
Necessity follows from Hendrickson's necessary conditions for global rigidity, Theorem \ref{thm:hendrickson}, using the fact that that plane triangulations do not have enough edges to be redundantly rigid. 
To prove sufficiency we suppose, for a 
contradiction, that $G=(V,E)$ is a non-planar, $4$-connected simplicial $2$-circuit graph which is not globally rigid and such that  $|E|$ is as small as possible. 
We first show:

\begin{claim}\label{clm:glob1_2}
For  each $e\in E$, there exists a simplicial $2$-circuit $\cT$ such that $G=G(\cT)$ and $\cT/e$ is a simplicial $2$-circuit. In particular, $G/e$ is a simplicial $2$-circuit graph for all $e\in E$.
\end{claim}
\begin{proof}
Choose a simplicial $2$-circuit $\cT$ such that $G=G(\cT)$.  Since 
$G\neq K_{3}$,  $\cT$ is non-trivial. Let $\{\cT_1,\cT_2,\ldots,\cT_m\}$ be a Fogelsanger decomposition of $\cT$ with respect to $e$ and put
$G(\cT_i)=G_i=(V_i,E_i)$ for all $1\leq i\leq m$. 
By Lemma~\ref{lem:3}\ref{en1:0}, $\cT_i/e$ is a simplicial $2$-circuit. Thus, if $E_i=E$ for some $1\leq i\leq m$, then we can put $G=G(\cT_i)$ and the claim holds 
since $\cT_i/e$ is a simplicial $2$-circuit. 
Hence, we may assume that $E_i\neq E$  and Lemma~\ref{lem:3}\ref{en1:c} now gives 
$|E_i|<|E|$ for all $1\leq i\leq m$.  By Lemma~\ref{lem:3}\ref{en1:d}, we can also assume the Fogelsanger  decomposition of $\cT$ is ordered so that $(\bigcup_{j=1}^{i-1} G_j)\cap G_i$ contains a copy of $K_{3}$ for all $2\leq i\leq m$.  

Since each $G_i$ is the graph of a non-trivial simplicial $2$-circuit, it has the  $3$-cleavage property. The minimality  of $G$ and Lemma~\ref{lem:mincleavage} now imply that  every $4$-block of each $G_i$ is  either a plane triangulation or is globally rigid   in $\R^{3}$ 
and hence each  $G_i$ has the strong $3$-cleavage property.
An easy induction using Corollary \ref{cor:cleavage_globrigid} now implies that  $H_j:=\bigcup_{i=1}^{j} G_i$ has the strong $3$-cleavage property
for all $1\leq j\leq m$. 
Thus $H_m=G$ has the strong $3$-cleavage property. Since $G$ is 4-connected and non-planar,
$G$ is globally rigid. This contradicts the choice of $G$ and completes the proof of 
the claim.
\end{proof}

We next show:

\begin{claim}\label{clm:glob2}
$G/e$ is $4$-connected for all $e\in E$.
\end{claim}
\begin{proof}
		Suppose, for a contradiction, that $G/e$ is not $4$-connected for some $e\in E$. Then $e$ is induced by  a 4-vertex separator $X_e=\{w,x,y,z\}$ of $G$. By Claim \ref{clm:glob1_2}, we can choose a simplicial 2-circuit $\cT$ such that $G=G(\cT)$ and $\cT/e$ is a simplicial 2-circuit.
Let $H_1$ be a component of $G-X_e$, $\cT_1=\{T\in \cT: T\subseteq  V(H_1)\cup X_e\}$ and $\cT_2=\cT\sm \cT_1$. Then Lemma~\ref{lem:1} and the fact that $\cT$ is a simplicial 2-circuit give $\partial \cT_1=E(C)=\partial \cT_2$ for some cycle $C$ of $G$  of length three or four with $V(C)\subseteq X_e$.

Suppose $C$ is a 3-cycle. By symmetry we can assume $V(C)=\{x,y,z\}$. Then Lemma~\ref{lem:34sep_a} implies that $\{x,y,z\}\not\in\cT$ and $\cT_i'=\cT_i\cup \{\{x,y,z\}\}$ is a simplicial 2-circuit for both $i=1,2$. We can now apply induction to 
$G_i=G(\cT_i')$ to deduce that $G_i$ has the strong $3$-cleavage property
for both $i=1,2$. Then $K(\{x,y,z\})\subseteq G_1\cap G_2$ and hence  $G=G_1\cup G_2$ has the strong $3$-cleavage property by Corollary \ref{cor:cleavage_globrigid}. Since $G$ is 4-connected and non-planar, this implies that $G$ is globally rigid, a contradiction.

Hence $C$ is a 4-cycle.   We can assume by symmetry that $C=wxyzw$ and $e=wx$. 
Since $\cT/e$ is a simplicial 2-circuit, we can apply Lemma~\ref{lem:34sep_b} to both non-adjacent pairs of vertices in $C$ to deduce that $\cT_i'=\cT_i\triangle \{\{w,x,y\},\{w,y,z\}\}$ and $\cT_i''=\cT_i\triangle \{\{w,x,z\},\{x,y,z\}\}$\ are simplicial 2-circuits for both $i=1,2$. Let $G_i'=G(\cT_i')$ and $G_i''=G(\cT_i'')$ for $i=1,2$. 
If $G[X_e]$ contains $xy$ or $wy$, then we can show that $G$ is globally rigid as in the last two paragraphs of the proof of 
Theorem~\ref{thm:globrigid_d_main}.
This would contradict the choice of $G$ so we have $G[X_e]=K(X_e)-xz-wy= C$
and hence none of the 2-simplices $\{w,x,y\},\{w,y,z\},\{w,x,z\},\{x,y,z\}$ belong to $\cT$.

Suppose $G_1'$ and $G_1''$ are both non-planar.
Then the graph obtained from $G(\cT_1)$ by adding a vertex of 
degree $4$ that is adjacent to $w,x,y,z$ is not planar (since it has $G_1'$ as a minor), and, in particular, is not the octahedral graph.
Using Lemma~\ref{lem:4sep4con} and symmetry, we may assume that $G_1'$ is 4-connected. 
Since $G_1'$ is non-planar, it is globally rigid by induction, and hence $G_1'-wy$ is rigid by Hendrickson's theorem (Theorem~\ref{thm:hendrickson}). 
On the other hand, $G_2'+K(X_e)$ has the strong $3$-cleavage property by induction and Lemma~\ref{lem:braced}, and it is globally rigid by 4-connectivity. 
We can now apply Theorem \ref{lem:globunion} (with $G_1:=G_1'-wy$,  $G_2:=G_2'-wy$ and $H:=(X_e,wy)$) to deduce that $G$ is globally rigid. 
 
Hence at least one of  $G_1', G_1''$ is planar. Since $|E(G_1')|=|E(G_1'')|$, Theorem \ref{thm:plane} implies that both $G_1', G_1''$ are planar.
If either $G_2'$ or $G_2''$ is a plane triangulation then this would imply that $|E(G)|=3|V(G)|-6$ and we could now apply  Theorem \ref{thm:plane} to deduce that  $G$ is a plane triangulation and contradict the choice of $G$.
Hence $G_2'$ and $G_2''$ are  both non-planar and we can proceed as in the previous paragraph to deduce that $G$ is globally rigid. This again contradicts the choice of $G$ and completes the proof of the claim.
\end{proof}

We can now complete the proof of the theorem. By Claims \ref{clm:glob1_2} and \ref{clm:glob2}, $G/e$ is a 4-connected simplicial circuit graph for every $e \in E$. If $G/e$ is nonplanar then, by the minimal choice of $E$, $G/e$ is globally rigid and then, by Theorem~\ref{thm:vertex_splitting_k}, $G$ is globally rigid contradicting our choice of $G$. Therefore $G/e$ is a $4$-connected plane triangulation for every $e \in E$ and, by Lemma~\ref{lem:planecontractions}, $G$ is a plane triangulation, again contradicting our choice of $G$.

\end{proof}

\section{Coincident rigidity}
\label{sec:coincidence}

In this section, we prove Theorem~\ref{thm:vertex_splitting_k}.
A key ingredient in our proof is the concept of `coincident rigidity'.
Given a graph $G=(V,E)$ and $u,v\in V$, we say that a realisation $p:V\to \R^d$ of $G$ is {\em $uv$-coincident} if $p(u)=p(v)$, and that $G$ is {\em $uv$-coincident rigid in $\R^d$} if $G$ has an infinitesimally rigid $uv$-coincident realisation in $\R^d$. It can be seen that $G$ is {$uv$-coincident rigid  in $\R^d$} if and only if every generic $uv$-coincident realisation of $G$ in $\R^d$ is infinitesimally rigid (where a {\em  generic $uv$-coincident realisation} is one in which the coordinates of $p|_{V-v}$ are algebraically independent).   
The following theorem states that  $uv$-coincidence rigidity is a sufficient condition for the vertex splitting operation to preserve global rigidity.

\begin{theorem}[\cite{J}]\label{thm:vsplit} Let $H$ be a graph which is globally rigid  in $\R^d$ and $v$ be a vertex of $H$. Suppose that $G$ is obtained
from $H$ by a vertex splitting operation which splits $v$ into two vertices $u$ and $v$,  and that $G$ is $uv$-coincident rigid  in $\R^d$. Then
$G$ is globally rigid  in $\R^d$.
\end{theorem}

In light of Theorem \ref{thm:vsplit}
this section is devoted to proving the following two results on coincident rigidity for simplicial $k$-circuit graphs. Even 
though we only need the case where $u$ and $v$ are adjacent for Theorem ~\ref{thm:vertex_splitting_k}, our inductive proofs require a weaker hypothesis which allows $u$ and $v$ to be non-adjacent.

\begin{theorem}\label{thm:coin_k} 
    For $k\geq 3$ let $\cS$ be a simplicial $k$-circuit and $u,v$ a pair of  distinct vertices of $\cS$. Suppose that $G(\cS)$ is $(k+2)$-connected and $G(\cS)/uv$ is rigid in $\mathbb R^{k+1}$. Then $G(\cS)$ is $uv$-coincident rigid in $\mathbb R^{k+1}$.
\end{theorem}
\begin{theorem}\label{thm:coin_2}
    
    Let $\cS$ be a simplicial $2$-circuit 
    and $u,v$ a pair of distinct vertices of $\cS$ such that either $G(\cS)$ is not a plane triangulation or $uv \not\in E(\cS)$. Suppose that $G(\cS)$ is $4$-connected and $G(\cS)/uv$ is rigid in $\mathbb R^3$. Then $G(\cS)$ is $uv$-coincident rigid in $\mathbb R^3$. 
        
\end{theorem}

The proofs of Theorems~\ref{thm:coin_k} and~\ref{thm:coin_2} are given in 
Section~\ref{sec:coin_2_k}.

\begin{proof}[\bf Proof of Theorem \ref{thm:vertex_splitting_k}]
Let $\cS$ be a simplicial $k$-complex satisfying $G(\cS) = G$ and suppose that $e=uv$. Since $G/uv$ is globally rigid in $\mathbb R^{k+1}$, $G/uv$ is rigid in $\mathbb R^{k+1}$. 

In the case that $k \geq 3$, Theorem \ref{thm:coin_k} implies that $G$ is $uv$-coincident rigid in $\mathbb R^{k+1}$ and then Theorem \ref{thm:vsplit} implies that $G$ is globally rigid in $\mathbb R^{k+1}$. 

Now suppose $k=2$. Suppose, for a contradiction, that $G$ is a plane triangulation. Since $G$ is 4-connected, $G/e$ is a plane triangulation and the hypothesis that $G/e$ is globally rigid in $\mathbb R^3$ implies that $G/e = K_4$. Therefore $|V(G)| = 5$ which is a contradiction since there is no 4-connected plane triangulation with 5 vertices. So $G$ is  not a plane triangulation and by Theorem \ref{thm:coin_2}, $G$ is $uv$-coincident rigid in $\mathbb R^3$ and Theorem \ref{thm:vsplit} implies that $G$ is globally rigid in $\mathbb R^3$.
\end{proof}

\subsection{Sufficient conditions for coincident rigidity}
\label{subsec:pre_coin}

We will derive several sufficient conditions for coincident rigidity in $\R^d$ which we will use in Section \ref{sec:coin_2_k}.

Let $G$ be a graph with $u, v\in V(G)$ and $f=xy\in E(G)$.
We will use the vertex splitting lemma,  Lemma~\ref{lem:split}, as a key tool to verify the $uv$-coincident rigidity of $G$ from the $uv$-coincident rigidity of $G/f$. Since $uv$-coincident realisations are not generic, we need to be careful to choose $f$ in such a way that we can apply  Lemma~\ref{lem:split}. 
This motivates our next definition.
We say that the edge $f=xy$ is {\em $uv$-admissible (in $\R^d$)} if
$x$ and $y$ have a set $N$ of  $d-1$ common neighbours  in $G$ with $\{u,v\}\not\subseteq N\cup \{x,y\}$.
\begin{lemma}\label{lem:split_uv}
Let $G$ be a graph, $u, v\in V(G)$, and $f=xy\in E(G)$.
Suppose that $f$ is $uv$-admissible and $G/f$ is $uv$-coincident rigid in $\R^d$.
Then $G$ is $uv$-coincident rigid in $\R^d$.
\end{lemma}
\begin{proof}
By the definition of $uv$-admissibility, $x$ and $y$ have a set $N$ of  $d-1$ common neighbours  in $G$ with $\{u,v\}\not\subseteq N\cup \{x,y\}$.
If $y$ is contracted to $x$ in $G/f$,  then $\{p(w): w\in N\cup \{x\}\}$ is in general position for any generic $uv$-coincident realisation of $G/f$, and we can apply Lemma~\ref{lem:split} to conclude that $G$ is $uv$-coincident rigid.
\end{proof}

Given two vertices $u,v$ in a graph $G=(V,E)$, 
the {\em $uv$-coincident closure of $G$ in $\mathbb R^d$}
is the graph ${\rm cl}_d^{uv}(G)$ obtained from $G$ by adding all edges $wz\in K(V)$ satisfying $\rank R(G+wz,p)=\rank R(G,p)$ for some generic $uv$-coincident realisation $p$ of $G$ in $\R^d$. Note that the spaces of infinitesimal motions of $(G,p)$ and $({\rm cl}_d^{uv}(G),p)$ will be the same for any generic $uv$-coincident realisation $p:V\to \R^d$.

\begin{lemma}\label{lem:uvrigidsubgraph}
    Let $G=(V,E)$ be a graph and $u, v$ be distinct vertices in $V$ such that $G/uv$ is rigid in $\R^d$.
    Suppose that 
    $u$ and $v$ are contained in a   subgraph $H$ of ${\rm cl}_d^{uv}(G)$ which is $uv$-coincident rigid in $\R^d$. 
    Then $G$ is $uv$-coincident rigid in $\R^d$.
\end{lemma}
\begin{proof}
Suppose, for a contradiction,  that $p$ is a generic $uv$-coincident realisation of $G$ and that $(G,p)$ has a non-trivial infinitesimal motion $q$.
Since $H$  is $uv$-coincident rigid in $\R^d$, we have $q(u)=q(v)$.
We may assume that  $\{u,v\}$ is contracted onto $u$ in $G/uv$.
Let $p'$ and $q'$ be the restrictions of $p$ and $q$ to $V-v$.
Since $p(u)=p(v)$ and $q(u)=q(v)$, $q'$ is a non-trivial infinitesimal motion of $(G/uv,p')$.
Since $p'$ is generic, this contradicts the rigidity of $G/uv$ in $\R^d$.
\end{proof}

\begin{lemma}\label{lem:neighbor_inclusion}
    Let $G = (V,E)$ be a graph and  $u,v$ be distinct vertices in $G$ such that $|N_G(u) \cap N_G(v)| \geq d$. Then $G$ is $uv$-coincident rigid in 
    $\mathbb R^d$ if and only if $G/uv$ is rigid in $\mathbb R^d$.
\end{lemma}
\begin{proof}
    
    Let $w_1,\cdots,w_d$ be $d$ common neighbours of $u,v$ in $G$.
    Suppose that $(G,p)$ is a generic $uv$-coincident framework and let $q:V \to \mathbb R^d$ be an 
    infinitesimal motion of $(G,p)$. For $i =1,\cdots,d$ we have 
    \begin{eqnarray*}
        q(u)\cdot(p(u) - p(w_i)) &= & q(w_i)\cdot(p(u)-p(w_i)) \\
        &= & q(w_i)\cdot(p(v)-p(w_i)) \\
        &= & q(v)\cdot(p(v)-p(w_i)) \\
        &= & q(v)\cdot(p(u) - p(w_i) \\
    \end{eqnarray*}
    Since $p$ is a generic $uv$-coincident realisation it follows that $\{p(u) - p(w_i):i = 1,\cdots, d\}$ is a basis for 
    $\mathbb R^d$. Therefore $q(u) = q(v)$ for any motion of $(G,p)$. This implies that the space of infinitesimal motions of $(G/uv,p|_{V\setminus \{v\}})$ is linearly isomorphic to the space of infinitesimal motions of $(G,p)$. The lemma follows immediately from this.
    
\end{proof}

\begin{lemma}\label{lem:rigidsubgraphs}
Let $G=(V,E)$ be a graph and $u, v$ be distinct vertices in $V$ such that $G/uv$ is rigid in $\R^d$. Suppose that $H_1,H_2\subseteq G$ are rigid in $\R^d$  with $u\in V(H_1)\sm V(H_2)$ and $v\in V(H_2)\sm V(H_1)$. Suppose further that $|V(H_1)\cap V(H_2)|\geq d-1$ and that, if equality holds, then there exists an edge $wz\neq uv$ in ${\rm cl}_d^{uv}(G)$ with $w\in V(H_1)\sm V(H_2)$ and $z\in V(H_2)\sm V(H_1)$.
Then $G$ is $uv$-coincident rigid in $\R^d$.
\end{lemma}
\begin{proof}
Choose a generic $uv$-coincident realisation $p$ of $H_1\cup H_2$  in $\R^d$. Then $(H_1,p|_{H_1})$ and $(H_2,p|_{H_2})$ are infinitesimally rigid. If
$|V(H_1)\cap V(H_2)|\geq d$ then we can apply the 
infinitesimal rigidity gluing property for frameworks
to $(H_1,p|_{H_1})$ and $(H_2,p|_{H_2})$ to deduce that $(H_1\cup H_2,p)$ is infinitesimally rigid. Then $G$ is $uv$-coincident rigid in $\R^d$ by Lemma~\ref{lem:uvrigidsubgraph}. Hence we may assume that $|V(H_1)\cap V(H_2)|= d-1$. Then the only infinitesimal motion of $(H_1\cup H_2,p)$ which fixes $H_1$ is a rotation of $H_2$ about the affine subspace spanned by $p(V(H_1)\cap V(H_2))$. Since this rotation would change the length of $wz$, $((H_1\cup H_2)+wz,p)$ is infinitesimally rigid in $\R^d$. Then $G$ is $uv$-coincident rigid in $\R^d$ by Lemma~\ref{lem:uvrigidsubgraph}.
\end{proof}

\subsection{Proofs of Theorems \ref{thm:coin_k} and Theorem \ref{thm:coin_2}}
\label{sec:coin_2_k}
%
Lemma~\ref{lem:neighbor_inclusion} will allow us to easily reduce to the case that $u,v$ have at most $k$ common neighbours. 
In that case our strategy will be to consider a Fogelsanger decomposition $\cS_1^+,\cdots, \cS_m^+$ of $\cS$ with respect to a suitably chosen edge $f$, apply induction to $\cS_i^+/f$ and then use Lemma~\ref{lem:split_uv} to return to $\cS_i^+$ and ultimately to $\cS$.

Given that the number of common neighbours of $u$ and $v$ plays an important role in the proofs, it is natural to consider edges whose contraction does not change that number. Let $G$ be a graph and $u,v\in V(G)$. We say that an edge $f=xy$ of $G$ is {\em normal with respect to $\{u,v\}$} if $f$ is not incident to $u$ or $v$ and  $\{u,v,x,y\}$ does not induce $C_4$ or $K_4$ in $G+uv$. Note that, if $f$ is normal with respect to $\{u,v\}$, then
\begin{equation}
    \label{eqn:normal1}
    \mbox{$|N_G(u)\cap N_G(v)| = |N_{G/f}(u) \cap N_{G/f}(v)|$.}
\end{equation}
We also note that if $G$ is the graph of a simplicial $k$-circuit and $f$ is normal with respect to $uv$ then it is also $uv$-admissible in $\mathbb R^{k+1}$. Indeed, by Lemma~\ref{lem:1.1.4} we have the following stronger implication.

\begin{lemma}
    \label{lem:normal2}
    Suppose that $\cS$ is a nontrivial simplicial $k$-circuit for some $k \geq 2$, $u,v$ are distinct vertices of $\cS$ and $xy \in E(\cS)$ is not incident to $\{u,v\}$.
    If $\{u,v,x,y\}$ does not induce $K_4$ in $G(\cS)+uv$ then $xy$ is  $uv$-admissible in $\mathbb R^{k+1}$.
\end{lemma}
\begin{proof}
    Let $X = N_{G(\cS)}(x) \cap N_{G(\cS)}(y)$. Lemma~\ref{lem:1.1.4} implies that $|X| \geq k$ and the hypothesis that $\{u,v,x,y\}$ does not induce $K_4$ in $G(\cS)+uv$ implies that $\{u,v\} \not\subseteq X$ which proves the lemma.
\end{proof}

\begin{lemma}
    \label{lem:no-normal}
    Suppose that $k \geq 2$, $\cS$ is a nontrivial simplicial $k$-circuit, $u,v \in V(\cS)$,  and there is no edge of $\cS$ that is normal with respect to $\{u,v\}$. Then $V(\cS) = \{u,v\} \cup (N_G(u) \cap N_G(v))$.
\end{lemma}

\begin{proof}
    Let $G = (V,E) = G(\cS)$, $X = N_G(u) \cap N_G(v)$, $Y = N_G(u)\setminus N_G(v)$, $Z = N_G(v)
    \setminus N_G(v)$. Since $G$ has no normal edge it follows that 
    \begin{itemize}
        \item $V = X \cup Y \cup Z \cup \{u,v\}$.
        \item $E[Y] = E[Z] = E[X,Y] = E[X,Z] = \emptyset$.
    \end{itemize}
    If both $X$ and $Y \cup Z$ are nonempty then $\{u,v\}$ is 
    a 2-vertex separator for $G$, a contradiction since $G$ is rigid in $\mathbb R^{k+1}$ and $k \geq 2$. Hence either $X=\emptyset$ or $Y\cup Z=\emptyset$. 
    Suppose, for a contradiction, that $X = \emptyset $. Let 
    $F$ be a 2-face of $\cS$. Then $F$ cannot contain both of $u,v$ since
    $X$ is empty. On the other hand $F$ cannot contain 2 elements of $Y$ or two elements of $Z$. So $F = \{w,y,z\}$ where $w \in \{u,v\}$, 
    $y \in Y$ and $z \in Z$, contradicting the fact that $yv , zu \not\in E$.
    Therefore $Y \cup Z = \emptyset$ as required.
\end{proof}

The next two lemmas provide some structural information about pairs of vertices  in a simplicial $k$-circuit that have at most $k$ common neighbours.

\begin{lemma}\label{lem:jim}
Let $\cS$ be a  simplicial $k$-circuit for some $k\geq 2$ and let $u,v$ be distinct vertices in $\cS$. Suppose that $|N_{G(\cS)}(u)\cap N_{G(\cS)}(v)|\leq k$.  Then $G(\cS)/uv$ is rigid in $\mathbb R^{k+1}$. Moreover if $|N_{G(\cS)}(u)\cap N_{G(\cS)}(v)|\leq k-1$, or, $uv\in E(\cS)$ and $|N_{G(\cS)}(u)\cap N_{G(\cS)}(v)|=k$, then ${\cal S}/uv$ is a  simplicial $k$-circuit.
\end{lemma}
\begin{proof}
    Let $G = (V,E) = G(\cS)$ and let $X = N_G(u) \cap N_G(v)$. If $\cS/uv$ is a simplicial
    $k$-circuit then $G/uv$ is rigid and the lemma holds. Therefore we may assume $\cS/uv$ is not a simplicial $k$-circuit. We will show that $|X| = k$, $uv \not\in E$ and $G/uv$ is rigid. 
    
    We proceed as in the definition of a Fogelsanger decomposition in Section \ref{sec:fog}. Since $\cS/uv$ is a simplicial $k$-cycle, we may choose  a decomposition $\cS'_1,\cdots, \cS'_m$ of $\cS/uv$ into simplicial $k$-circuits. Since $\cS/uv$ is not a simplicial $k$-circuit $m \geq 2$. 
   For all $1\leq i\leq m$, let 
$\cS_i=\gamma_{uv}^{-1}(\cS_i')$. 
Then $\{u,v\} \subseteq V(\cS_i)$ since, if not, $\cS_i$ would be a simplicial $k$-circuit properly contained in $\cS$, contradicting the fact that $\cS$ is a simplicial $k$-circuit.
Let 
\begin{equation}
\cS_{i}^*=\{ K\subseteq V(\cS_i): \{u,v\}\subseteq K \mbox{ and }K-u,K-v\in  \partial \cS_i\} \mbox{ and }\cS_i^+=\cS_i\triangle \cS_{i}^*.
\end{equation}
Then $\cS_i^+=\cS_i\cup \cS_{i}^*  $, since $\cS_i\cap \cS_{i}^*=\emptyset$, and  $\cS_i^+$ is a 
simplicial $k$-circuit  by Lemma~\ref{lem:1.5}(c). Let $G_i=(V_i,E_i)=G(\cS_i)$.

     For $1\leq i\leq m$, $uv \in E_i$ and so 
    $|N_{G_i}(u)\cap N_{G_i}(v)| \geq k$.  Since $G_i\subseteq G+uv$, it follows that $N_{G_i}(u) \cap 
    N_{G_i}(v) = X$ and $|X|=k$. 
    Since $lk_{\cS_i^+}(uv)$ is a nonempty simplicial $(k-2)$-cycle without repeated $(k-2)$-simplices, it follows from Lemma~\ref{lem:1.1.1}
    that, for each $i = 1,\cdots,m$, $lk_{\cS_i^+}(uv) \cong \cK_{k-2}$
    and $V(lk_{\cS_i^+}(uv)) = X$. In particular $(\cS_1^+)_{\{u,v\}} = 
    (\cS_2^+)_{\{u,v\}}$.
    This gives $\cS_1^*=\cS_2^*$ and hence
    $\cS_1^+ \triangle \cS_2^+=\cS_1\cup \cS_2$ is a 
    nonempty simplicial $k$-cycle that is contained in $\cS$. Hence  $m=2$,
    $\cS = \cS_1^+ \triangle \cS_2^+$ and $uv \not\in E(\cS)$. 
    Recall that $\cS_i = \cS_i^+/uv$ is 
    a simplicial circuit and so is rigid in $\mathbb R^{k+1}$. Now
    $\cS/uv = \cS_1 \cup \cS_2$ and 
    $X \cup \{u\} \subseteq V(\cS_1) \cap V(\cS_2)$. By the gluing property
    for rigidity $G(\cS)/uv$ is rigid in $\mathbb R^{k+1}$. 
\end{proof}

\begin{lemma}
    \label{lem:uvblock}
    For $k\geq 2$ let $\cS$ be a nontrivial simplicial $k$-circuit and $G = (V,E) = G(\cS)$. Let $uv \in E$ such that $|N_G(u)\cap N_G(v)| = k$. Then 
    \begin{itemize}
        \item[(a)] $\{u,v\}$ is not contained in any $(k+1)$-vertex separator of $G$.
        \item[(b)] If $\cS_1^+,\cdots, \cS_m^+$ is a Fogelsanger decomposition of $\cS$ with respect to some 
        edge $xy$ of $\cS$
        then $uv \in E(\cS_j^+)$ for exactly one $j \in \{1,\cdots,m\}$.
    \end{itemize}
\end{lemma}
\begin{proof}
    By Lemma~\ref{lem:jim} $\cS/uv$ is a simplicial $k$-circuit. Therefore $G/uv$ is $(k+1)$-connected which implies (a).
    
    If $\{x,y\} = \{u,v\}$ then $\cS/xy = \cS/uv$ which is a simplicial $k$-circuit by Lemma~\ref{lem:jim}. Hence
    $m=1$ in the Fogelsanger decomposition of the statement of (b). 
    So we may assume that $\{x,y\} \neq \{u,v\}$. 
    For $1\leq i \leq m$ let $G_i = (V_i,E_i) = G(\cS_i^+)$ and suppose,
    for a contradiction, that $uv \in E_i \cap E_j$ for some $i \neq j$. Then $u,v$ have at least $k$ common neighbours in both $G_i$ and $G_j$. Therefore $N_{G_i}(u) \cap N_{G_i}(v) = N_{G_j}(u) \cap N_{G_j}(v) = N_{G}(u) \cap N_{G}(v)$ and, by Lemma~\ref{lem:1.1.4},  
    $lk_{\cS_i^+}(uv) = lk_{\cS_j^+}(uv)$ and so $(\cS_i^+)_{uv} = (\cS_j^+)_{uv}$. By Lemma~\ref{lem:3}(c) this implies that every $k$-simplex of $\cS_i^+$ that contains $\{u,v\}$ must also contain $\{x,y\}$. Since $\cS_i^+$ is a nontrivial 
    simplicial $k$-circuit and $\{u,v\} \neq \{x,y\}$, this contradicts Lemma~\ref{lem:1.1.4}(a) and completes the proof of (b).
\end{proof}

\begin{lemma}\label{lem:b1}
    For $k \geq 2$ let $\cS$ be a  simplicial $k$-circuit and $u,v\in V(\cS)$ with $uv\not\in E(\cS)$. Suppose that 
    $G=G(\cS)$ 
    and $|N_G(u)\cap N_G(v)|\leq k$.
    Then $G$ is $uv$-coincident rigid in $\mathbb{R}^{k+1}$.
\end{lemma}
\begin{proof}
We will simplify terminology throughout the proof by suppressing reference to the ambient space $\R^{k+1}$.

We proceed by contradiction. Suppose the lemma is false and that $({\cal S},u,v)$ is a counterexample  with $|E({\cal S})|$ as small as possible. 
Since $|N_G(u)\cap N_G(v)|\leq k$, Lemma~\ref{lem:jim} implies that $G({\cal S})/uv$ is rigid. Lemma~\ref{lem:uvrigidsubgraph} and the fact that $({\cal S},u,v)$ is a counterexample to the lemma now imply
\begin{equation}\label{eq:subgraph_k_b}
\mbox{$\{u,v\}$ is not contained in any $uv$-coincident rigid  subgraph of $\cl^{uv}(G({\cal S})$).}
\end{equation}

The following claim is the most important step in our proof of the lemma.

\begin{claim}\label{claim:0_k_b_1}
No edge of $G$ is normal.
\end{claim}
\begin{proof}
Suppose, for a contradiction, that $f$ is a normal edge of $G$.
Choose  a  Fogelsanger decomposition $\{\cS_1^+,\dots, \cS_m^+\}$ of $\cS$  with respect to $f$. 
Denote $G(\cS^+_i)$ by $G_i=(V_i, E_i)$ for simplicity. 
We  show:
 \begin{equation}
 \label{eq:coinclaim14_k_b_1}
\mbox{$\{u, v\}\not\subseteq V_i$ for all $1\leq i\leq m$.}
\end{equation}
Suppose that $\{u, v\} \subseteq V_i$.
Then $\cS_i^+/f$ is a  simplicial $k$-circuit with fewer edges than $\cS$. Also $uv \not\in E(\cS_i^+)$, since $E(\cS_i^+) \subseteq E(\cS)$ by Lemma \ref{lem:3}(d), and (\ref{eqn:normal1}) implies that $|N_{G(\cS_i^+)/f}(u) \cap N_{G(\cS_i^+)/f}(v)| \leq k $. By the minimal choice of $\cS$ we deduce that $G(\cS_i^+/f)=G(\cS_i^+)/f$ is $uv$-coincident rigid.  By Lemma~\ref{lem:normal2} $f$ is $uv$-admissible and Lemma~\ref{lem:split_uv} implies that $G(\cS_i^+)$ is $uv$-coincident rigid, contradicting (\ref{eq:subgraph_k_b}). Thus (\ref{eq:coinclaim14_k_b_1}) holds.

Since $V({\cal S})=\bigcup_{i=1}^m V_i$ by Lemma~\ref{lem:3}(d),  we have $u\in V_i$ and $v\in V_j$ for some $1\leq i\neq j\leq m$. 
Relabelling if necessary, we may assume that $u\in V_1$ and $v\notin V_1$ by (\ref{eq:coinclaim14_k_b_1}). Then Lemma~\ref{lem:3}(e) implies that we can reorder $\cS^+_2,\cS_3^+,\ldots,\cS_m^+$ so that $|(\bigcup_{i=1}^h V_i)\cap V_{h+1}|\geq k+1$ for all $1\leq h<m$. Since each $G_i$ is rigid, a simple induction using the gluing lemma tells us that $G_h':=\bigcup_{i=1}^h G_i$ is rigid for all $1\leq h<m$. We have $v\in V_j$ for some $2\leq j\leq m$ and we may suppose that $j$ has been chosen to be as small as possible. Then  $G_{j-1}',G_j$ are both rigid, $u\in V(G_{j-1}')\sm V(G_j)$, and $v\in  V(G_j)\sm V(G_{j-1}')$.  
 Since $|V(G_{j-1}')\cap V(G_j)|\geq k+1$, $G_{j-1}'\cup G_j$ is $uv$-coincident rigid by Lemma~\ref{lem:rigidsubgraphs}. This contradicts (\ref{eq:subgraph_k_b}) and completes the proof of the claim. 
 \end{proof}
Now, Lemma~\ref{lem:no-normal}, Claim \ref{claim:0_k_b_1} and the fact that $|N_G(u)\cap N_G(v)| \leq k$ imply  that $|V(\cS)| \leq k+2$ and so, since $\cS$ must be a nontrivial simplicial $k$-circuit, $\cS = \cK_k$ contradicting the fact that $uv \not\in E(\cS)$. This completes the 
proof of Lemma~\ref{lem:b1}
\end{proof}

\begin{proof}[\bf Proof of Theorem \ref{thm:coin_k}]
    Let $\cS$ be a simplicial circuit for some $k \geq 3$ such that $G= (V,E) = G(\cS)$ is $(k+2)$-connected, and let $u,v$ be distinct vertices in $V$ such that $G/uv$ is globally rigid in $\mathbb R^{k+1}$. 
    
    Let $X = N_G(u)\cap N_G(v)$. If $|X| \geq k+1$ then the theorem follows immediately from Lemma~\ref{lem:neighbor_inclusion}. So we may assume that  $|X| \leq k$. 
    We can now use Lemma~\ref{lem:b1} to deduce that the theorem holds when $uv\not\in E$ and so  we can assume that $uv \in E$.   

    By Lemma~\ref{lem:1.1.4} $|X| = k$ and $lk_\cS(uv) \cong \cK_{k-2}$.
    Therefore $\cS_{\{u,v\}} = \{U \subseteq X\cup \{u,v\}: u,v \in U, |U| = k+1\}$. Now let $\cT$ be the copy of $\cK_k$ with vertex set $X \cup \{u,v\}$. Then $\cS \triangle \cT = (\cS \setminus \cS_{\{u,v\}}) \triangle \{ X\cup \{u\}, X \cup \{v\}\})$ is a simplicial $k$-cycle. If $\cS \triangle \cT = \emptyset$ then $\cS = \cT$, a contradiction since $G(\cT)$ is not $(k+2)$-connected. Therefore $\cS \triangle \cT$ is nonempty and, since $k \geq 3$,
    $G(\cS \triangle \cT)$ is a subgraph of $ G-uv$. Note that the degree of every vertex in the $(k+2)$-connected graph $G$ is at least $k+2$, so it follows that  $V(\cS \triangle \cT) = V(\cS)=V$.
    If $\cS \triangle \cT$ is a simplicial
    $k$-circuit then by Lemma~\ref{lem:b1}, $G-uv$ is $uv$-coincident rigid in 
    $\mathbb R^{k+1}$. So we can assume that $\cS \triangle \cT =
    \cU \sqcup \cV$ where $\cU$ and $\cV$ are both nonempty simplicial 
    $k$-cycles. Since $\cS$ is a simplicial $k$-circuit, $\cU, \cV \not\subseteq \cS$ and so it follows, after relabelling if necessary, that
    $\cU \setminus \cS = \{ X \cup \{u\}\}$
    and $\cV \setminus \cS = \{ X \cup \{v\}\}$. Moreover,
    $\cU$ must be a simplicial $k$-circuit, since if $\cU'$ is a proper subcomplex of $\cU$ that is a simplicial $k$-cycle, then either 
    $\cU'$ or $\cU\triangle \cU'$ is a proper subcomplex of $\cS$, contradicting the fact that $\cS$ is a simplicial $k$-circuit. 
    Similarly $\cV$ is a simplicial $k$-circuit. Also $G(\cU) \cup G(\cV) = G-uv$ since $X \cup \{u\} \in \cU$ and $X \cup \{v\} \in \cV$.

    Now, if $v \in V(\cU)$ then, since $uv \not\in E(\cU)$, Lemma~\ref{lem:b1} implies that $G(\cU)$ is a $uv$-coincident rigid subgraph 
    of $G$ and, using Lemma~\ref{lem:jim} and Lemma~\ref{lem:uvrigidsubgraph}, that $G$ is $uv$-coincident rigid. Thus we can assume that $v \not\in V(\cU)$ and
    similarly that $u \not\in V(\cV)$. Since $G$ is $(k+2)$-connected,
    $X$ cannot be a vertex separator for $G-uv$ and so 
    $|V(\cU) \cap V(\cV)| \geq k+1$. Now by Lemma~\ref{lem:rigidsubgraphs}, $G-uv$ is 
    $uv$-coincident rigid and so $G$ is $uv$-coincident rigid.
    \end{proof}
\medskip


We now turn to the case $k=2$ and the proof of Theorem \ref{thm:coin_2}. 
We will simplify terminology by frequently suppressing reference to the ambient space $\mathbb R^3$ in the remainder of this section.
So the reader should henceforth interpret all references to (coincident) rigidity as (coincident) rigidity in $\mathbb R^3$.

The overall strategy is the same as for the proof of Theorem \ref{thm:coin_k}.
The main difference 
is that the argument used in
the second paragraph of the proof of 
Theorem \ref{thm:coin_k}
fails since $G(\cS\triangle \cT)$ is not necessarily 
a subgraph of $G - uv$ when $k=2$. 

Before proving that theorem we will need to establish some further results on simplicial 2-circuits.
The first result from \cite{J} concerns the coincident rigidity of
graphs that have a spanning subgraph that is a plane triangulation (in that paper we refer to such graphs as `braced plane triangulations').

\begin{theorem}[{\cite[Theorem 1.4]{J}}] \label{lem:coinbraced}
Suppose that $G = (V,E)$ is a plane triangulation, 
$B$ is a nonempty subset of $K(V)\setminus E$ such that $G+B$ is 4-connected, and $uv\in E$ is an edge of $G$ which does not belong to a separating 3-cycle of $G$. Then $G + B$ is $uv$-coincident rigid.
\end{theorem}

We will also need the following technical result about contractions of simplicial $2$-circuits. 
\begin{lemma}
    \label{lem:2circuit1}
    Let $\cS$ be a simplicial 2-circuit such that $G =(V,E)= G(\cS)$ is 4-connected and $uv \in E$ such that $|N_G(u) \cap N_G(v)| = 3$. Then
    $G/uv$ is rigid in $\mathbb R^3$.
\end{lemma}
\begin{proof}
    If $\cS/uv$ is a circuit then $G/uv = G(\cS/uv)$ is rigid by Theorem \ref{thm:fogelsanger} (Fogelsanger's rigidity theorem). So we may suppose that $\cS/uv$ is not a 
    simplicial $2$-circuit. Let $X = N_G(u)\cap N_G(v)$ and suppose $\cS_1^+,\cdots,\cS_m^+$ is a Fogelsanger 
    decomposition of $\cS$ with respect to $uv$. Then $m\geq 2$ and for 
    each $i=1,\cdots, m$ $lk_{\cS_i^+}(uv)$ is a simplicial $0$-cycle  contained in $X$ that contains no repeated 0-faces. In other words
    $lk_{\cS_i^+}(uv)$ is a 2-element subset of the 3-set $X$. If 
    $lk_{\cS_i^+}(uv) = lk_{\cS_j^+}(uv)$ for $i\neq j$ then $(\cS_i^+)_{uv}
    = (\cS_j^+)_{uv}$ and,  by Lemma~\ref{lem:3}(c), $\cS_i^+ \triangle \cS_j^+$ is a nonempty simplicial 
    $2$-cycle contained in $\cS$. Furthermore $uv \not\in E(\cS_i^+\triangle
    \cS_j^+)$ so $\cS_i^+\triangle \cS_j^+$ is properly contained in $\cS$ contradicting the fact that $\cS$ is a simplicial $2$-circuit. Therefore 
    $lk_{\cS_i^+}(uv) \neq lk_{\cS_j^+}(uv)$ for $i\neq j$. Since $|X| = 3$ it follows that $m\leq 3$ (since there are only 3 different 2-element subsets of $X$). If $m = 3$ then $(\cS_1^+)_{uv} \triangle (\cS_2^+)_{uv} \triangle (\cS_3^+)_{uv} =\emptyset$. 
    Then, using Lemma~\ref{lem:3}(c), $\cS_1^+ \triangle \cS_2^+ \triangle \cS_3^+$ would be a nonempty simplicial $2$-cycle properly contained in $\cS$, since it would not contain the edge $uv$, a contradiction. Therefore $m=2$ and $\cS/uv = \cS_1 \sqcup \cS_2$ where 
    $\cS_1, \cS_2$ are both simplicial $2$-circuits. Since $G$ is 4-connected, 
    $G/uv$ is 3-connected and it follows that $|V(\cS_1) \cap V(\cS_2)| \geq 3$. Now $G/uv$ is rigid in $\mathbb R^3$ by Fogelsanger's Theorem and the gluing property for rigidity.
\end{proof}

\begin{proof}[\bf Proof of Theorem \ref{thm:coin_2}]
Suppose, for a contradiction,  that $(\cS, u, v)$ is a counterexample with $|E(\cS)|$ minimal. Let $G = (V,E) = G(\cS)$ and $X = N_G(u) \cap N_G(v)$. 
Then $\cS$ is a simplicial $2$-circuit, $G$ is $4$-connected and not planar and $G$ is not $uv$-coincident rigid.
If $|X| \geq 3$ then the $uv$-coincident rigidity of $G$ would follow from Lemma~\ref{lem:neighbor_inclusion},  so  $|X| \leq 2$. If $uv\not\in E$ then the $uv$-coincident rigidity of $G$ would follow  from Lemma~\ref{lem:b1}. Hence  $uv \in E$ and $|X| = 2$.

In particular, by Lemma~\ref{lem:jim},
$\cS/uv$ is a simplicial $2$-circuit and so $G/uv$ is rigid in $\mathbb R^2$ by Theorem \ref{thm:fogelsanger}. Since $\cS$ is a counterexample, Lemma
    \ref{lem:uvrigidsubgraph} implies that
    \begin{equation}
        \label{eqn:nosub}
        \textrm{no subgraph of $\cl^{uv}(G)$ containing $u,v$ is $uv$-coincident rigid in $\mathbb R^3$.}
    \end{equation}

    Our first claim shows that the inductive hypothesis extends automatically
    to a somewhat larger class of graphs.
    \begin{claim}
        \label{clm:inductive}
        Let $\cU$ be a simplicial $2$-circuit such that $|E(\cU)| 
        < |E(\cS)| $. Let $H = (W,F) = G(\cU)$. Suppose that  $B \subseteq E(K(W))\setminus F$ such 
        that $ H+B$ is 4-connected and nonplanar. Then $H+B$ is 
        $ab$-coincident rigid in $\mathbb R^3$ for every $ab\in F$ such that
        $H/ab$ is rigid.
    \end{claim}
    \begin{proof}
        If $H$ is $4$-connected and not planar then this follows from the minimal choice of 
        $\cS$. If $H$ is $4$-connected and planar then $B\neq \emptyset$ since $H+B$ is not planar, and the claim 
        follows from Theorem \ref{lem:coinbraced}. 
        So we can assume that $H$ is not $4$-connected. 
        Since $H/ab$ is rigid, $\{a,b\}$ is not contained in any 3-vertex separator of $H$. If $Y$ is a $3$-vertex separator of $H$ and $a$ and $b$ lie in distinct components of $H-Y$, then, since the cleavage graphs of $H$ at $Y$ are rigid, Lemma~\ref{lem:rigidsubgraphs} implies that $H$ is $ab$-coincident rigid. Hence 
        there is a unique $4$-block, $C$, of $H$ that contains both $a$ and $b$. By the minimal choice of $\cS$ and 
        Lemma~\ref{lem:uvrigidsubgraph}, we can assume that $C$ is a plane triangulation. 

        Since $H$ is not 4-connected but $H+B$ is, there are some edge $cd\in B$ and a path $D_1,X_1,\dots, X_{s-1}, D_s$ in the 4-block tree of $H$ such that $c \in V(D_1)\setminus X_1$, $d \in V(D_s)\setminus X_{s-1}$ and $C = D_i$ for some $1 \leq i \leq s-1$.
        Furthermore $\{a,b\} \not\subseteq X_{i-1}$  and $\{a,b\} \not\subseteq X_i$, since $H/ab$ is rigid. 
        
        Suppose 
        $2 \leq i $. Then $\bigcup_{j=1}^{i-1}D_j$ and 
        $\bigcup_{j = i+1}^s D_j$ are both rigid and neither of them 
        contain $\{a,b\}$. It follows that $\{cw: w \in X_{i-1}\} \subseteq \cl^{ab}(H)$ and $\{dz: z \in X_{i}\} \subseteq \cl^{ab}(H)$.
        Now $C+c+d+cd+\{cw: w\in X_{i-1}\}+\{dz: z \in X_{i}\}$ is an 
        $ab$-coincident rigid graph by Theorem \ref{lem:coinbraced}.
        Since $H/ab$ is rigid, Lemma~\ref{lem:uvrigidsubgraph} implies
        that $H+B$ is $ab$-coincident rigid. 

        If $i=1$ then, as in the previous paragraph $\{dz:z \in X_1\} \subseteq \cl^{ab}(H)$. By Theorem \ref{lem:coinbraced},
        $C+d+cd+\{dz:z \in X_1\}$ is $ab$-coincident rigid, and 
        Lemma~\ref{lem:uvrigidsubgraph}, together with the fact that $H/ab$ is rigid, imply that $H+B$
        is $ab$-coincident rigid. 
        %
        %
    \end{proof}

    Next we turn to the analysis of normal edges in our minimal counterexample $\cS$.
    
    \begin{claim}
        \label{clm:circuitchange}
        Suppose that $f$ is an edge of $G$ that is normal with respect to $\{u,v\}$. Then there is some simplicial $2$-circuit $\cS^f$ such that $G(\cS^f) = G(\cS)$ 
        and such that $\cS^f/f$ is a simplicial $2$-circuit.
    \end{claim}
    \begin{proof}
        Let $\cS_1^+,\cdots, \cS_m^+$ be a Fogelsanger decomposition of $\cS$ with respect to $f = xy$. 
        Let $G_i = (V_i,E_i) = G(\cS_i^+)$. 
        By Lemma~\ref{lem:uvblock}(b), 
        $uv$ belongs to exactly one of $E_1,\cdots,E_m$. Relabelling if necessary, we can assume that $uv \in E_1$ and $uv \not\in E_i$ for $i \geq 2$. Suppose that $\{u,v\}
        \subseteq V_i$ for some $i \geq 2$. 
        Since $uv \not\in E_i$ and $|N_{G_i}(u)\cap N_{G_i}(v)|\leq |X| = 2$,
        Lemma~\ref{lem:b1} implies that $G_i$ is $uv$-coincident rigid, contradicting (\ref{eqn:nosub}). Therefore $\{u,v\} \not\subseteq V_i$ for all $i \geq 2$. 

        By Lemma~\ref{lem:uvblock}(a),
        $\{u,v\}$ is contained in 
        a unique $4$-block $D$ of $G_1$. 
        If $|E(D)|=|E|$ then $G(\cS_1^+)=G$ and, by Lemma~\ref{lem:3}(a), $\cS_1^+/f$ is a simplicial circuit which proves the claim. So we may assume that $|E(D)|<|E|$. We will show that this assumption leads to a contradiction.
        
        If $D$ is not a plane triangulation then by the minimal choice of $\cS$, and the fact that $D/uv$ is rigid by Lemma \ref{lem:jim}, $D$ is $uv$-coincident rigid, contradicting (\ref{eqn:nosub}). Therefore $D$ is a plane triangulation. 
        Let $\cT_0$ be a simplicial $2$-circuit whose graph is $D$.
        Since $D$ is either a copy of $K_4$ or a 4-connected plane triangulation, there are no 3-vertex separators in $D$, and so every edge of $D$ belongs to exactly two 3-cliques of $D$. It follows that 
        $W$ is a $2$-simplex of $\cT_0$ if and only if $W$ is a $3$-clique in $D$. 

        Since $D = G(\cT_0)$ is a $4$-block of $G_1$, by repeated applications of Lemma~\ref{lem:1.1.2} and Lemma~\ref{lem:34sep_a}, there exist simplicial $2$-circuits $\cT_1,\cdots,\cT_r$ such that 
        $\cS_1^+ = \cT_0 \triangle \cT_1 \triangle \cdots \triangle \cT_r$, $\cT_i \cap \cT_j = \emptyset$ for $1\leq i \neq j \leq r$ and $\cT_i \cap \cT_0$ consists of a single $2$-simplex for $i=1,\cdots,r$. ($\cT_1,\cdots,\cT_r$ correspond to the branches of the 4-block tree of $G_1$ which are incident to $D$.)
        As a notational convenience, for $i=r+1,\cdots,r+m-1$, let $\cT_i = \cS_{i-r+1}^+$.
        Now let $F$ be the intersection graph of 
        $\cT_0, \cT_1,\cdots, \cT_{r+m-1}$. So $V(F) = \{t_0,\cdots, t_{r+m-1}\}$ and $t_it_j \in E(F)$ if and only if the simplicial $2$-circuits  $\cT_i,\cT_j$ have a $2$-simplex in common. 
        Then $\triangle_{i=0}^{r+m-1} \cT_i = \triangle_{j=1}^m \cS_j^+ = \cS$, using Lemma~\ref{lem:3}(c), and so $\bigcup_{i=0}^{r+m-1}G(\cT_i) = G$. Since $|E(\cT_0)| = |E(D)| <  |E|$, $F$ has at least 2 vertices. By Lemma~\ref{lem:3}(e) $F$ is connected. 
        
        Let $J\subseteq \{t_1,\cdots,t_{r+m-1}\}$ be the vertex set of a component of $F-t_0$ and put $H_J = \bigcup_{t_j \in J}G(\cT_j)$. Then
        \begin{equation}
            \label{eqn:HJ}
            \text{$H_J$ is rigid in $\mathbb R^3$}
        \end{equation}
        by the rigidity gluing property. 
        We next show that 
        \begin{equation}
            \label{eqn:nouv}
            \text{$\{u,v\}\not\subseteq V(H_J)$.} 
        \end{equation}
        Suppose that $\{u,v\} \subseteq V(H_J)$. Let $t_{i_1}t_{i_2}\ldots t_{i_s}$ be a shortest path in $F[J]$ joining two vertices $t_{i_1},t_{i_s}$ with $u\in V(\cT_{i_1})$ and $v\in V(\cT_{i_s})$. Since $\{u,v\} \not\subseteq V(\cT_i)$ for all $i \geq 1$, $s\geq 2$. Then Lemma~\ref{lem:rigidsubgraphs} implies that $G(\cT_{i_1}) \cup \cdots \cup  G(\cT_{i_s})$ is $uv$-coincident rigid, contradicting (\ref{eqn:nosub}). This proves (\ref{eqn:nouv}).
        Now let $X_J = V(\cT_0) \cap V(H_J)$. 
        Since $t_0$ is adjacent in $F$ to a vertex in $J$ we have $|X_J| \geq 3$. Suppose that $|X_J|\geq 4$. Then, since $\{u,v\} \not\subseteq X_J$ by (\ref{eqn:nouv}), $D$ has at least 5 vertices and hence $D$ is a 4-connected plane triangulation. Moreover, using (\ref{eqn:HJ}) and (\ref{eqn:nouv}),
        it follows that $K(X_J) \subseteq \cl^{uv}(G)$.
        Now $K(X_J) \not\subseteq E(D)$ since a 4-connected plane triangulation cannot contain a copy of $K_4$. Therefore 
        $D+K(X_J)$ is $uv$-coincident rigid by Theorem \ref{lem:coinbraced} contradicting (\ref{eqn:nosub}). 
        
        Hence $|X_J| = 3$. Then
        $X_J$ is  a 3-vertex separator for $D \cup H_J$ and $X_j\in \cT_0\cap \cT_j$ 
        for some $t_j \in J$.  
        Since $G$ is $4$-connected, $X_J$ is not a vertex separator for $G$. So there must 
        be another subset $J' \subseteq \{t_1,\cdots,t_{r+m-1}\}$, disjoint from $J$, that is the vertex set of a
        component of $F-t_0$ such that $(V(H_{J'}) \cap V(H_J))\setminus V(D) \neq \emptyset$ and $X_J\neq X_{J'}$. Let $w \in (V(H_{J'}) \cap V(H_J))\setminus V(D)$. Then 
        since $H_J$ and $H_{J'}$ are both rigid, it follows that $xw \in \cl^{uv}(G)$ for all $x \in X_J \cup X_{J'}$. Now, it follows easily from Theorem \ref{lem:coinbraced} and the facts that
        $\{u,v\} \not\subseteq X_J$ and $X_J \neq X_{J'}$ that $D + x +\{xw: x \in X_J \cup X_{J'}\}$ is $uv$-coincident rigid, contradicting (\ref{eqn:nosub}).

        This final contradiction completes the proof of Claim \ref{clm:circuitchange}.
    \end{proof}

    \begin{claim}
        \label{clm:nonormal2}
        Suppose that $f=xy$ is an edge of $G$ that is normal with respect to $\{u,v\}$. Then $G/f$ is 4-connected.
    \end{claim}

    \begin{proof}
        Since the statement of the claim is only concerned with $G$, we can assume that  $\cS = \cS^f$ is the simplicial 2-circuit given by Claim~\ref{clm:circuitchange} and hence  $\cS/f$ is a simplicial $2$-circuit.
        Suppose, for a contradiction, that $G/f$ is not $4$-connected. Then there exists a 4-vertex separator $Z$  of $G$ such that $Z$ contains $\{x,y\}$.
        Since $\cS/f$ is a simplicial $2$-circuit and $G-Z = G/f-(Z-y)$, $G-Z$ has exactly two components, and we may apply (\ref{eqn:normal1}) and Lemma~\ref{lem:uvblock}(a) to deduce that $\{u,v\} \not\subseteq Z$.

        Let $Y_1, Y_2$ be the vertex sets of the components of $G-Z$, labelled so that $\{u,v\} \cap Y_2 = \emptyset$. We may assume  that $f$ and $Z$ have been chosen so that $|Y_2|$ is minimal.
        Let $\cS_1 = \{U \in \cS:U \subseteq Y_1 \cup Z\}$ and
        $\cS_2 = \cS\setminus \cS_1$. Then $\cS = \cS_1 \triangle \cS_2$
        and so $\partial \cS_1 = \partial \cS_2 \neq \emptyset$ and $V(\partial \cS_1) \subseteq Z$.
        We also recall that, since $f$ is normal, $|N_{G/f}(u) \cap N_{G/f}(v)| = 2$.

        We first show that 
        \begin{equation}
            \label{eqn:bdS1}
            \text{$\partial \cS_1$ is isomorphic to $\cL_1$.}
        \end{equation}
        Suppose not. Then, since  $\partial \cS_1$ is a nonempty simplicial $1$-cycle with no repeated $1$-simplices, it must be isomorphic to $\cK_1$ by Lemma~\ref{lem:1.1.1} and Lemma~\ref{lem:1.1.3}. Let 
        $F = V(\partial \cS_1)$ and suppose that $\{z\} = Z\setminus F$. By Lemma~\ref{lem:34sep_a}, both $\cS_1 \cup\{ F\}$ and $\cS_2 \cup \{F\}$ are simplicial $2$-circuits.
        Therefore $G(\cS_2) = G(\cS_2 \cup \{F\})$ is rigid by Theorem \ref{thm:fogelsanger} and, since $\{u,v\} \not\subseteq V(\cS_2)$, it follows that $K(Z) \subseteq \cl^{uv}(G)$. But 
        $G(\cS_1)\cup K(Z)$ is $4$-connected and therefore is $uv$-coincident rigid by Claim \ref{clm:inductive}, contradicting (\ref{eqn:nosub}). This proves (\ref{eqn:bdS1}) and implies, in particular that $Z=V(\partial \cS_i) = \{a,b,c,d\}$, say, where $\partial \cS_1 = \{\{a,b\},\{b,c\},\{c,d\},\{d,a\}\}$.

        Let $A= \{d,a,b\}$, $B = \{a,b,c\}$, $C = \{b,c,d\}$ and $D = \{c,d,a\}$. We next show that 
        \begin{equation}
            \label{eqn:xyinbd}
            f \in E(\partial \cS_1).
        \end{equation}
        Suppose not. Then, relabelling if necessary, we may assume that $f=ac$. For $i=1,2$ let $\cS'_i = \cS_i \triangle \{B,D\}$. Then $\cS'_i$ is a nonempty simplicial 2-cycle without repeated $2$-faces and so it follows that $\cS'_i/ac$ is a 
        nonempty simplicial $2$-cycle. But $\cS_i/ac = \cS'_i/ac$, 
        since $\{x,y\} \subseteq B,D$, and so $\cS/f$ is not a simplicial $2$-circuit, contradicting our assumption that $\cS = \cS^f$. This proves (\ref{eqn:xyinbd}).

        Henceforth we assume that the labels have been chosen so that $c=x$ and  $d = y$ and hence $\{x,y\} = C\cap D$. We next prove that, for $i=1,2$,
        \begin{equation}
            \label{eqn:2circuits}
            \text{$\cS_i \triangle\{A,C\}$ and
            $\cS_i\triangle \{B,D\}$ are both simplicial $2$-circuits.}
        \end{equation}
        We will prove that $\cS_1 \triangle \{A,C\}$ is a simplicial $2$-circuit (all other cases are proved by analogous arguments). Since $\partial \cS_1 = \partial \{A,C\}$, $\cS_1 \triangle \{A,C\}$ is a simplicial $2$-cycle. Suppose it is not a simplicial $2$-circuit. Then $\cS_1\triangle \{A,C\} = \cA \sqcup \cC$ where $\cA, \cC$ are nonempty simplicial $2$-cycles. Since 
        $\cS$ is a simplicial $2$-circuit and $\cS_1\subsetneq\cS$ it follows that $\cA \not\subseteq \cS_1 $ and $\cC \not\subseteq \cS_1$. Therefore, after relabelling, we can assume that $\{A\} = \cA\setminus \cS_1$ and $\{C\} = \cC \setminus \cS_1$. Now $\cC/f$ is a nonempty simplicial $2$-cycle that is properly contained in $\cS/f$, contradicting the fact that $\cS/f = \cS^f/f$ is a simplicial $2$-circuit and proving (one of the cases of) (\ref{eqn:2circuits}).

        We can now combine the facts that  $\{u,v\} \subseteq V(\cS_1)$ and $G(\cS_1) +K(Z)$ is 4-connected and  nonplanar (since a 
        4-connected plane triangulation cannot contain a 4-clique) with 
        Claim~\ref{clm:inductive} and (\ref{eqn:2circuits}) to deduce that $G(\cS_1)+K(Z)$ is $uv$-coincident rigid. By (\ref{eqn:nosub}), it follows that 
        \begin{equation}
            \label{eqn:noKZ}
            K(Z) \not\subseteq \cl^{uv}(G).
        \end{equation}
        
        Observe that if $ac \in E$, respectively $bd \in E$, then 
        $G(\cS_2 \triangle \{B,D\}) \subseteq G$,
        respectively 
        $G(\cS_2 \triangle \{A,C\}) \subseteq G$. Since $\{u,v\} \not\subseteq G(\cS_2)$ and both $G(\cS_2 \triangle \{B,D\})$
        and $G(\cS_2 \triangle \{A,D\})$ are rigid by (\ref{eqn:2circuits}), this contradicts 
        (\ref{eqn:noKZ}). Therefore 
        \begin{equation}
            \label{eqn:nodiags}
            ac, bd \not\in E.
        \end{equation}

        We say that an edge $g = pq \in E$ is {\em suitable} if $p,q \in Y_2$. Note that every suitable 
        edge is normal with respect to $\{u,v\}$ since $\{u,v\} \not\subseteq V(\cS_2)$. 
        We next show that 
        \begin{equation}
            \label{eqn:suitable1}
            \text{$G/g$ is 4-connected for each suitable edge $g \in E$.}
        \end{equation}
        
        Suppose, for a contradiction, that $G/g$ is not 4-connected. 
        Since $g$ is normal with respect to $\{u,v\}$, we can apply the same argument we used for $f$ to deduce that $G$ has a 4-vertex separator $Z'$ such that $G-Z'$ has exactly two components with vertex sets $Y_1'$ and $Y_2'$ where  $\{u,v\}\subseteq Y_1'\cup Z'$. Since $G[Y_1\cup Z]$ is 3-connected by Lemma~\ref{lem:4sep4con} and $|Z'\cap (Y_1\cup Z)|\leq 2$, we have $Y_1\cup Z\subsetneq Y_1'\cup Z'$. This contradicts the choice of $(f,Z)$ to minimise $Y_2$, and completes the proof of (\ref{eqn:suitable1}).  
        
        We next show that
        \begin{equation}\label{eqn:bill}
        \text{$G/g$ is planar for every suitable edge $g$.}
        \end{equation}
        Suppose for a contradiction that $G/g$ is nonplanar. 
        Then $G/g$ is a simplicial $2$-circuit graph by Claim~\ref{clm:circuitchange} and $G/g$ is 4-connected by (\ref{eqn:suitable1}). Also,
        since $g$ is normal, $|N_{G/g}(u)\cap N_{G/g}(v)|\leq 2$, so  Lemma \ref{lem:jim} implies that $(G/g)/uv$ is rigid.
        Now the minimal choice of $\cS$
        implies that $G/g$ is $uv$-coincident rigid.
        Since $G$ is suitable,  Lemma~\ref{lem:normal2} implies that $g$ is $uv$-admissible and Lemma~\ref{lem:split_uv} now implies that $G$ is
        $uv$-coincident rigid, a contradiction. 
        This completes the proof of (\ref{eqn:bill}).

        Statements (\ref{eqn:suitable1}) and (\ref{eqn:bill}) imply that $G/g$ is a 4-connected plane triangulation for every suitable edge $g$. Lemma~\ref{lem:planecontractions} now tells us that $G$ has at most one suitable edge. We complete the proof by considering the following two cases.

\medskip
\noindent
        {\bf Case 1:} 
        $G$ has no suitable edge.
        We first show that $|Y_2|=1$. Choose $w \in Y_2$ and let $U$ be a $2$-simplex of $\cS_2$ that contains $w$. Since $G$ has no suitable edges, (\ref{eqn:nodiags}) implies that $U-w$
        is a $1$-simplex of $\partial \cS_2$. Therefore $lk_\cS(w)$
        is a simplicial 1-cycle contained in $\partial \cS_2$. Since
        $\partial \cS_2 \cong \cL_1$ is a simplicial 1-circuit it follows that $lk_\cS(w) = \partial \cS_2$. Thus 
        $\{\{a,b,w\},\{b,c,w\},\{c,d,w\},\{d,a,w\}\} \subseteq \cS_2$ for any 
        $w \in V(\cS_2)-Z$. 
        
        Hence if $w_1,w_2$ are distinct vertices in $V(\cS_2) -Z$ then $G[Z\cup \{w_1,w_2\}]$ is an 
        octahedral graph which is rigid. This would imply that $K(Z) \subseteq 
        \cl^{uv}(G)$ and contradict (\ref{eqn:noKZ}). 
        Thus $V(\cS_2) = Z \cup \{w\}$ and
        $wz \in E(\cS_2)$ for each $z \in Z$.
        
        Using Lemma~\ref{lem:4sep4con}, 
        and after relabelling $a,b,c,d$ if necessary, we may assume that
        $a \notin \{u,v\}$ and that
        $G/aw = G(\cS_1)+ac = G(\cS_1 \triangle\{B,D\})$ is a 4-connected simplicial $2$-circuit graph. Also since $G$ is nonplanar, Theorem \ref{thm:plane} implies that $|E| \geq 3|V|-5$. Now since $a,w$ have exactly 2 common neighbours in $G$, $|E(G/aw)| = |E|-3$, $|V(G/aw)| = |V|-1$ and so $G/aw$ is also nonplanar.
        
        Now $N_G(u) \cap N_G(v) \subseteq N_{G/aw}(u) \cap N_{G/aw}(v) \subseteq 
        (N_G(u) \cap N_G(v)) \cup \{a\}$. Hence $|N_{G/aw}(u)
        \cap N_{G/aw}(v)| \leq 3$ and we may apply Lemma~\ref{lem:jim}, Lemma~\ref{lem:2circuit1} and Claim \ref{clm:inductive} to deduce that 
        $G/aw$ is $uv$-coincident rigid. Also $aw$ is $uv$-admissible since 
        $b,d$ are common neighbours of $a$ and $w$ and $\{b,d\} \neq \{u,v\}$. Therefore, by Lemma~\ref{lem:split_uv}, $G$ is $uv$-coincident rigid, a contradiction.

\medskip
\noindent
{\bf Case 2:} 
        $G$ has exactly one suitable edge $g=pq$.
        Using Lemma~\ref{lem:4sep4con} and relabelling if necessary, we may 
        assume that $G(\cS_2)+ac$ is $4$-connected. Also $G(\cS_2)+ac = G(\cS_2 \triangle \{B,D\})$ is a simplicial $2$-circuit graph by (\ref{eqn:2circuits}). 
        
        We first show that $G(\cS_2)+ac$ is nonplanar. Suppose, for a contradiction, that $G(\cS_2)+ac$ is planar. Then $p,q$ must have exactly two common neighbours in $G$. Since $G$ is a
        nonplanar simplicial $2$-circuit graph, and using 
        Theorem \ref{thm:plane}, we can deduce that $|E(G/g)| \geq 3|V(G/g)| - 5$ and so $G/g$ is also 
        nonplanar, contradicting (\ref{eqn:bill}).
        Therefore $G(\cS_2)+ac$ is nonplanar. 
        Also since $G/g$ is a plane triangulation, it follows, using (\ref{eqn:nodiags}) and the fact that $G(\cS_1) =(G/g)[V(\cS_1)]$, that both $G(\cS_1)+ac$ and $G(\cS_1)+bd$ are plane triangulations.

        We next show that $Y_2 = \{p,q\}$. Suppose for a contradiction that
        $ r\in Y_2 \setminus\{p,q\}$. Then, since $pq$ is the unique suitable edge in $G$, $r$ is only adjacent to vertices in $\partial \cS_2$. We may now deduce, as in Case 1, that $\cW = \{\{a,b,r\},\{b,c,r\},\{c,d,r\},\{d,a,r\}\} \subseteq \cS_2$. This contradicts the fact that $\cS$ is a simplicial $2$-circuit since $\cW \cup \cS_1$ is a simplicial $2$-cycle. Therefore $Y_2 = \{p,q\}$ and so $V(\cS_2) = \{a,b,c,d,p,q\}$. 
        
        Now since $G(\cS_2)+ac = G(\cS_2 \triangle \{B,D\})$ is a nonplanar simplicial $2$-circuit graph by (\ref{eqn:2circuits}), Theorem \ref{thm:plane} implies that $G(\cS_2)+ac$ has at least 13 edges. So $G(\cS_2)$ has at least 12 edges.
        It follows that $G(\cS_2) = K-\{ac,bd\}$ or $G(\cS_2) = K - \{ac,bd, zw\}$ where $K$ is the complete graph on $\{a,b,c,d,p,q\}$, and  
        $z \in \{a,b,c,d\}$,  $w \in \{p,q\}$. Both of these graphs are rigid in $\mathbb R^3$, contradicting (\ref{eqn:noKZ}).
        
        This final contradiction completes the proof of Claim \ref{clm:nonormal2}.
    \end{proof}

    Now suppose that $f$ is a normal edge of $G$. If $G/f$ is 
    not a plane triangulation then by Claim \ref{clm:nonormal2}, Claim \ref{clm:circuitchange} and the minimal choice of $\cS$, $G/f$ is $uv$-coincident rigid. 
    Now, using Lemma~\ref{lem:normal2} and Lemma~\ref{lem:split_uv} it follows that $G$ is $uv$-coincident rigid, a contradiction. 
    Therefore for every normal edge $f$ in $G$, $G/f$ is a 4-connected  plane triangulation.
    Lemma~\ref{lem:planecontractions} now implies that 
    \begin{equation}
    \label{eqn:5}
    \text{there is at most one normal edge in $G$.}
    \end{equation}
    
    Since every vertex in the 
    4-connected graph $G$ has degree at least 4, 
    (\ref{eqn:5}) implies that every vertex in $V\setminus \{u,v\}$ is adjacent to at least one of $u,v$.
    Thus $V\setminus \{u,v\}$ can be partitioned into three sets:
    $X=N_G(u)\cap N_G(v)$,
    $Y=N_G(u)\setminus N_G(v)$,
    and $Z=N_G(v)\setminus N_G(u)$.
    Now (\ref{eqn:5}) further implies that there is at most one edge between $X$ and $Y\cup Z$. The 4-connectivity of $G$  and the fact that $|X| = 2$ now gives $Y \cup Z = \emptyset$. Therefore $|V| = 4$ which contradicts the hypothesis that $G$ is 4-connected.
    
    This completes the proof of Theorem~\ref{thm:coin_2}.
    \end{proof}

\section{Applications and conjectures}

\subsection{Rigidity and global rigidity of \texorpdfstring{$M$}{M}-connected simplicial complexes}
Let ${\cal S}$ be a simplicial $k$-complex.
Recall that the $k$-simplicial matroid ${\cal M}({\cal S})$ is a linear matroid on ${\cal S}$ represented by the $k$-th boundary operator $\partial_k$ over $\mathbb{Z}_2$ of the simplicial chain complex of ${\cal S}$. 
We  say that ${\cal S}$ is {\em $M$-connected} if  ${\cal M}({\cal S})$ is a connected matroid i.e.~${\cal M}({\cal S})$ has a circuit that contains $S_1$ and $S_2$ for any $S_1, S_2\in {\cal S}$. We can use the fact that simplicial $k$-circuits have the (strong) $(k+1)$-cleavage property to extend Theorems \ref{thm:fogelsanger}, \ref{thm:globrigid_3_main} and \ref{thm:globrigid_d_main} to $M$-connected simplicial complexes.

\begin{theorem}\label{thm:M_connected}
Let ${\cal S}$ be an $M$-connected simplicial $k$-complex
and $k\geq 2$. Then $G({\cal S})$ is rigid in $\mathbb{R}^{k+1}$.
In addition, if $B \subseteq E(K(V(\cS)))$ then $G({\cal S})+B$ is globally rigid in $\mathbb{R}^{k+1}$ if and only if it is $(k+2)$-connected and is not a plane triangulation when $k=2$.
\end{theorem}
\begin{proof}
Rigidity follows immediately from Theorem \ref{thm:fogelsanger} since it implies that every pair of vertices of $G({\cal S})$ is contained in a rigid subgraph. The necessity of the conditions for global rigidity follows immediately from Theorem \ref{thm:hendrickson}. To see sufficiency,   
 choose a simplex $S_0\in {\cal S}$.
Since ${\cal S}$ is $M$-connected,
for each $S\in {\cal S}\setminus\{S_0\}$, ${\cal M}({\cal S})$ has a circuit ${\cal C}_S$ (or equivalently, a $k$-simplicial circuit) that contains $S_0$ and $S$.
By Theorems~\ref{thm:globrigid_3_main} and~\ref{thm:globrigid_d_main}, $G({\cal C}_S)$ has the strong $(k+1)$-cleavage property.
Note that, for any $S_1, S_2\in {\cal S}\setminus \{S_0\}$,
both ${\cal C}_{S_1}$ and ${\cal C}_{S_2}$ contain $S_0$,
so $G({\cal C}_{S_1})$ and $G({\cal C}_{S_2})$ share at least $k+1$ vertices.
Hence, by Corollary~\ref{cor:cleavage_globrigid}, 
$\bigcup_{S\in {\cal S}\setminus \{S_0\}} G({\cal C}_{S})$ has the strong $(k+1)$-cleavage property.
Moreover, $\bigcup_{S\in {\cal S}\setminus \{S_0\}} G({\cal C}_{S})$ is a spanning subgraph of $G({\cal S})$.
So $G({\cal S})+B$ has the strong $(k+1)$-cleavage property by Lemma~\ref{lem:braced}.
This and the $(k+2)$-connectivity of $G({\cal S})+B$ imply sufficiency.
\end{proof}


Theorem~\ref{thm:M_connected} yields an algorithm that can verify if a graph $G$ is rigid or globally rigid in $\R^{k+1}$. 

\begin{enumerate}
    \item Compute the set ${\cal A}_G$ of  all $(k+1)$-cliques in $G$.
    \item If $V(\cA_G) \neq V(G)$ then stop.
    \item If $V(\cA_G) = V(G)$ compute the $(k-1)$-face-to-$k$-face incidence matrix of $\cA_G$ (with rows indexed by $k$-faces). Let $\cal{M}$ be the row matroid of this matrix over $\mathbb Z/2\mathbb Z$.
    \item Compute the connected components of the matroid $\cal{M}$. If none of these components is spanning (i.e. has vertex set equal to $V(G)$) then stop. If some component of $\cal{M}$ is spanning then
    \begin{enumerate}
        \item $G$ has the strong cleavage property.
        \item $G$ is rigid in $\mathbb R^{k+1}$.
        \item $G$ is globally rigid in $\mathbb R^{k+1}$ if and only if it is $(k+2)$-connected and, in the case $k=2$, non-planar.  
    \end{enumerate}
\end{enumerate}
In general the algorithm above cannot be used to decide if a given graph is globally rigid in $\mathbb R^{k+1}$. For example $K_{6,6}$ is known to be globally rigid in $\mathbb R^3$, but $\cA_{K_{6,6}}$ is empty (with $k=2$). 

However, for any simplicial $k$-complex $\cS$,
$G=G(\cS)$ implies $\cS\subseteq \cA_G$.
This fact and Theorem~\ref{thm:M_connected} imply that the above
algorithm (combined with tests for $(d+1)$-connectivity  and  planarity) correctly verifies the rigidity and the global rigidity of $G$ in $\mathbb{R}^{k+1}$ if $G$ contains a spanning subgraph that is the graph of an $M$-connected simplicial $k$-complex.
In particular, it correctly verifies if a simplicial $k$-circuit graph $G$ is globally rigid  in polynomial-time without necessarily computing a simplicial $k$-circuit $\cS$ with $G=G(\cS)$. Developing an efficient recognition algorithm for the class of simplicial $k$-circuit graphs is an interesting, but perhaps difficult, open problem. 

We also note that Theorem \ref{thm:M_connected} and the algorithm above illustrate an advantage of the class of simplicial $k$-circuits over the class of simplicial $k$-manifolds. It does not seem easy to find an analogous certifying algorithm that relies only on the global rigidity characterisation for simplicial $k$-manifolds.

\subsection{Rigidity of abstract simplicial complexes} 
Theorem~\ref{thm:M_connected}
 gives a sufficient condition for the graph of an abstract simplicial complex to be rigid in $\R^d$.
The body and hinge theorems of Tay \cite{T89} and Katoh and Tanigawa \cite{KT} give other such conditions.
It would be of interest to find yet more sufficient conditions for the graph  of a simplicial complex to be  rigid in $\R^d$. For example: 

\begin{conjecture}
The graph of an abstract simplicial $(d+1)$-complex is rigid in $\R^d$ whenever it is $(2d-1)$-connected.
\end{conjecture}
A suitably large ring of copies of $K_{2d-2}$ in which consecutive pairs have a $K_{d-1}$ in common and non-consecutive pairs are disjoint shows that the connectivity bound of $2d-1$ in this conjecture would be best possible. The conjecture holds when $d=1$ since every connected graph is rigid in $\R^1$. We can use the proof technique of \cite[Theorem 3.2]{JJ} to verify the conjecture when $d=2$, and  \cite[Theorem 7.1]{CJT} tells us that the hypotheses of the  conjecture imply $C^1_2$-cofactor rigidity (which is conjectured to be the same as rigidity in $\R^3$).

\subsection{A lower bound conjecture for simplicial circuits}
As mentioned in the introduction, Fogelsanger's Rigidity Theorem implies that the bound for the number of edges (1-faces) in Barnette's Lower Bound Theorem extends from connected triangulated $(d-1)$-manifolds to simplicial $(d-1)$-circuits. 
We conjecture that this extension applies to faces of other dimensions as well.
More precisely let $f_j(\cS)$ be the number of $j$-dimensional faces of a simplicial complex $\cS$ and, for $k\geq 1$, let 
\[\varphi_j(n,k) = \left\{ \begin{array}{ll} 
            n\binom{k+1}{j} - j\binom{k+2}{j+1} & \text{ for }1 \leq j \leq k-1 \\
kn - (k+2)(k-1)& \text{ for }j=k \end{array}\right. \]

\begin{conjecture}
    Suppose that $\cS$ is a non-trivial simplicial $k$-circuit with $n$ vertices. Then, for $1\leq j \leq k$, 
    \begin{equation}
        \label{eqn:lbtconj} 
        f_j(\cS) \geq \varphi_j(n,k).
    \end{equation}
    Moreover, equality holds in (\ref{eqn:lbtconj}) for any $1 \leq j \leq k$ if and only 
    if either $k=1$, or $k=2$ and $\cS$ is  a triangulated 2-sphere, or $k\geq 3$ and $\cS$ is a stacked $k$-sphere.
\end{conjecture}

Note that the standard proof technique of reducing the Lower Bound Theorem to the case $j=1$ does not translate easily to this setting since the class of simplicial circuits is not link closed. In general the link of a vertex in a simplicial $k$-circuit will be a simplicial $(k-1)$-cycle but not necessarily a simplicial $(k-1)$-circuit.

\subsection{Polytopes and stresses}
For a framework $(G,p)$, an element of the  left kernel of its rigidity matrix $R(G,p)$ is called a {\em stress} of $(G,p)$.
Let $S(G,p)$ be the space of all stresses of $(G,p)$.
In the context of the reconstruction problem for polytopes, Kalai, see ~\cite[Conjecture 5.7]{NZ},  conjectured that a simplicial 
$d$-polytope $P$ is uniquely determined up to affine equivalence by its underlying graph and its space of stresses of its underlying framework whenever $d\geq 4$ and $G(P)$ is $(d+1)$-connected.
Since the boundary complex of a simplicial 
$d$-polytope is a triangulated $(d-1)$-manifold,  our characterisation of global rigidity for graphs of simplicial circuits solves this conjecture whenever $p$ is generic. We will need the following results on stresses to see this.

For a  stress $\omega$ of a framework $(G,p)$, the corresponding {\em stress matrix} $\Omega$ is  the Laplacian matrix of $G$ weighted by $\omega$.
Gortler, Healy, and Thurston~\cite{GHT} proved that  a generic $d$-dimensional framework $(G,p)$ with $n$ vertices is globally rigid if and only if it has a {\em full rank stress} i.e., there is a stress $\omega\in S(G,p)$ such that $\rank \Omega=n-(d+1)$.
Also, a well-known observation by Connelly (see, e.g., \cite{C,GHT}) says that, if $\omega$ is a full rank stress of $(G,p)$ and $(G,q)$ is another realisation of $G$ which has $w$ as a stress, then $q$ is an affine image of $p$.

\begin{theorem}\label{thm:affine}
Let $d\geq 3$, and $(G,p)$ be a generic framework in $\mathbb{R}^{d}$ such that $G$ is the graph of a simplicial $(d-1)$-circuit, $G$ is $(d+1)$-connected and $G$ is non-planar when $d=3$.
Suppose that $(G,q)$ is a realisation of $G$ in $\mathbb{R}^d$ such that $S(G,p)=S(G,q)$.
Then $q$ is an affine image of $p$.
\end{theorem} 
\begin{proof}
Since $G$ is globally rigid in $\R^d$ by Theorems \ref{thm:globrigid_3_main} and \ref{thm:globrigid_d_main}, and $p$ is generic, the above mentioned result of Gortler, Healy, and Thurston implies that $(G,p)$ has a full rank stress $\omega$. We can now use Connelly's observation and the hypothesis that  
$S(G,p)=S(G,q)$ to deduce that $q$ is an affine image of $p$.
\end{proof}
Specializing Theorem~\ref{thm:affine} to simplicial $d$-polytopes, we obtain Theorem~\ref{thm:kalai}. We note that, after this paper was first submitted, Murai, Novik and Zheng proved Kalai’s conjecture for all convex simplicial $d$-polytopes where $d \geq 5$ using different methods \cite{MNZ}. The case of non-generic convex $4$-polytopes remains open.

\section*{Acknowledgements}
We thank the anonymous referee for detailed and constructive feedback and for several suggestions that improved the paper.

We thank the American Institute of Mathematics for hosting the 2019 workshop on the rigidity and flexibility of microstructures during which this work was started. We also thank the Heilbronn Institute and the Fields Institute  for hosting further rigidity workshops.
The discussions we had during these workshops played a crucial role in the writing of this paper.

This work was supported by JST PRESTO Grant Number JPMJPR2126,  JST ERATO Grant Number JPMJER1903, JSPS KAKENHI Grant Number 18K11155, 20H05961, and EPSRC overseas travel grant EP/T030461/1.

\printbibliography

\end{document}